\documentclass[12pt, reqno]{amsart}

\usepackage[T1]{fontenc}
\usepackage[latin1]{inputenc}
\usepackage[english]{babel}
\usepackage{amsthm}
\usepackage{amscd}
\usepackage{amssymb,amsmath,bbm,MnSymbol}
\usepackage[all]{xy}
\usepackage{mathrsfs}
\usepackage[left=0.0cm,right=0.0cm,top=1.7cm,bottom=1.7cm]{geometry}
\usepackage{hyperref}
\usepackage{verbatim}
\usepackage{stmaryrd}

\theoremstyle{plain}
\newtheorem{theorem}{Theorem}[section]

\newtheorem{lemma}[theorem]{Lemma}
\newtheorem{proposition}[theorem]{Proposition}

\newtheorem{corollary}[theorem]{Corollary}

\newtheorem*{theorem*}{Theorem}

\theoremstyle{definition}
\newtheorem{definition}[theorem]{Definition}

\theoremstyle{remark}
\newtheorem{remark}[theorem]{Remark}

\usepackage[usenames,dvipsnames]{color}

\usepackage{tikz}
\usetikzlibrary{calc}

\def\Q{{\mathbb{Q}}}
\def\Z{{\mathbb{Z}}}
\def\C{{\mathbb{C}}}

\def\R{{\mathbb{R}}}

\def\O{{\mathcal{O}}}

\def\zp{{\Z_p}}

\def\Norm{\mathrm{N}}

\def\A{{\mathbb{A}}}
\def\Af{{\A_{f}}}

\def\ord{\mathrm{ord}}

\def\nord{\mathrm{n}.\mathrm{ord}}

\def\det{\mathrm{det}}

\def\GL{\mathrm{GL}}

\def\frakp{\mathfrak{p}}

\def\matrix#1#2#3#4{{\big(\begin{smallmatrix}#1&#2\\ #3&#4\end{smallmatrix}\big)}}

\def\mr#1{\mathrm{#1}}

\title{Hirzebruch--Zagier cycles in $p$-adic families and adjoint $L$-values}

\author{Antonio Cauchi}\thanks{A.C.'s research was partly funded by the NSERC grant RGPIN-2018-04392, the FRQ-NT Team grant 2019-PR-256236, and the Concordia Horizon postdoc fellowship n.8009, by the JSPS Postdoctoral Fellowship for Research in Japan, and by Taighde \'{E}ireann -- Research Ireland under Grant number IRCLA/2023/849 (HighCritical).}
\address{Antonio Cauchi\newline UCD School of Mathematics and Statistics\\ University College Dublin\\ Belfield\\Dublin 4\\Ireland}
\email{antonio.cauchi@ucd.ie}

\author{Marc-Hubert Nicole}\thanks{M.-H.N.'s research was partly funded by NSERC grant RGPIN-2019-06957.}
\address{Marc-Hubert Nicole \newline LMNO - Laboratoire de Math\'ematiques Nicolas Oresme, CNRS - UMR 6139, UNICAEN, Normandie Univ., 14000 Caen, France }
\email{marc-hubert.nicole@unicaen.fr}

\author{Giovanni Rosso}\thanks{G.R.'s research was partly funded by the NOVA-FRQNT-CRSNG grant 325940, the FRQ-NT Team grant 2019-PR-256236, and the NSERC grant RGPIN-2018-04392}

\address{Giovanni Rosso \newline Department of Mathematics and Statistics, Concordia University, 1455 De Maisonneuve Blvd. W.,
Montreal, QC  H3G 1M8
Canada }
\email{giovanni.rosso@concordia.ca}

\makeatletter
\def\@tocline#1#2#3#4#5#6#7{\relax
  \ifnum #1>\c@tocdepth 
  \else
    \par \addpenalty\@secpenalty\addvspace{#2}%
    \begingroup \hyphenpenalty\@M
    \@ifempty{#4}{%
      \@tempdima\csname r@tocindent\number#1\endcsname\relax
    }{%
      \@tempdima#4\relax
    }%
    \parindent\z@ \leftskip#3\relax \advance\leftskip\@tempdima\relax
    \rightskip\@pnumwidth plus4em \parfillskip-\@pnumwidth
    #5\leavevmode\hskip-\@tempdima
      \ifcase #1
       \or\or \hskip 1em \or \hskip 2em \else \hskip 3em \fi%
      #6\nobreak\relax
    \dotfill\hbox to\@pnumwidth{\@tocpagenum{#7}}\par
    \nobreak
    \endgroup
  \fi}
\makeatother
\usepackage{hyperref}
\makeindex

\begin{document}

\begin{abstract}
Let $E/F$ be a quadratic extension of totally real number fields. We show that the generalized Hirzebruch--Zagier cycles arising from the associated Hilbert modular varieties can be put in $p$-adic families. As an application, using the theory of base change, we give a geometric construction of the multivariable $p$-adic adjoint $L$-function twisted by the Hecke character of $E/F$, attached to Hida families of Hilbert modular forms over $F$.
\end{abstract}

\maketitle

\setcounter{tocdepth}{1}
\tableofcontents

\section{Introduction}

 The theme of $p$-adic variation has irrigated much of modern arithmetic geometry in the last half-century especially through modular forms and Galois representations. In recent years, more and more examples of $p$-adic families of cohomology classes of algebraic cycles, e.g. Heegner points, have been unearthed and utilized in arithmetic applications. More generally, for a large class of Shimura varieties, their associated special cycles can be packaged compatibly in $p$-power level towers and further in $p$-adic families. In this paper, we take on the task of filling a classical gap in the literature: Hilbert modular varieties. In particular, a central technical contribution of this paper is to exhibit the $p$-adic variation of Hirzebruch--Zagier divisors and more generally so-called Hirzebruch--Zagier cycles in higher dimension.
 Indeed, Hirzebruch and Zagier in their epoch-making paper showed that explicitly defined divisors (essentially, images of modular curves) on a Hilbert modular surface over $\Q(\sqrt{p})$ for $p \equiv 1 \mod{4}$ are Fourier coefficients of a classical modular form of weight 2. A vast web of deep conjectures generalizing the Hirzebruch--Zagier theorem on Hilbert modular surfaces as well as the Gross--Kohnen--Zagier theorem on modular curves, has been developed under the banner of the Kudla program especially for orthogonal and unitary groups. Since it befits squarely our purpose, we rely on the framework of the treatise of Getz and Goresky \cite{GetzGore} that reworked and much generalized the original Hirzebruch--Zagier theorem from Hilbert modular surfaces to higher dimensional Hilbert modular varieties, by crucially using as a central tool relative quadratic base change for automorphic representations.

\subsection{Main results} Let $E/F$ be a quadratic extension of totally real number fields with $[F:\Q]=d$ and let \[H= {\rm Res}_{\mathcal{O}_F/\Z}\GL_2,\,\, G= {\rm Res}_{\mathcal{O}_E/\Z}\GL_2.\] There are Hilbert modular varieties $Y_H$ and $Y_G$ of dimension $d$ and $2d$ associated with $H$ and $G$, respectively. The natural embedding of $H$ into $G$ induces an embedding $\iota: Y_H \hookrightarrow Y_G$ of codimension $d$. The cycle determined by $\iota$ is referred to as the Hirzebruch--Zagier cycle. Let $\Sigma_E$ denote the set of archimedean places of $E$ and $\Phi$ be the $F$-type for $E$ defined in \eqref{Ftype}. Then, using branching laws for the restriction to $H$ of algebraic representations of $G$ and the pushforward by $\iota$, one can construct the cycle in middle degree cohomology of $Y_G$ \[\mathscr{Z}_K^{(\underline{k},m)} \in H^{2d}(Y_G(K), \mathscr{V}^{(\underline{k},m)}). \] 
Here $K \subseteq G(\widehat{\Z})$ denotes a suitable neat compact open subgroup and $\mathscr{V}^{(\underline{k},m)}$ is the local system attached to the algebraic representation of $G$ of highest weight $ (\underline{k},m) \in (\Z_{\geq 0})^{\Sigma_E} \times \Z$ s.t. $k_\sigma \equiv m$ (mod $2$) and $k_\sigma = k_{\sigma'}$ for every $\sigma \in \Phi$. These cycles, or rather an analogous construction in intersection cohomology, play a crucial role in \cite{GetzGore}. Indeed, their Hecke translates are used to construct a Fourier series which defines a Hilbert modular form on $H$ with coefficients in the part of the middle degree intersection cohomology of $Y_G$ supported by base change Hilbert modular forms from $H$. The significance of these cycles is underscored by their link with the special value at $s=1$ of the adjoint $L$-function for $H$, which we now recall in a simplified form. We invite the reader to consult \S \ref{sec:cohomologygroups} and \S \ref{sec:RSintegrals} for further details. 

Let $g$ be a cusp eigenform for $G$ of level $K$, weight $(\underline{k},0)$, and trivial Nebentypus and consider its twist $g_0 = g \otimes \theta^{-1} $ with $\theta$ a unitary Hecke character on $\A_E^\times$ whose restriction to $\A_F^\times$ equals to the quadratic Hecke character $\eta_{E/F}$ of $E/F$. Attached to $g_0$ and an $F$-type $J$, there is a rapidly decreasing differential form $\omega_J(g_0^{-\iota})$, see \eqref{eq:RDdiffForm}. By an analytic formula of Hida, the Poincar\'e pairing between $\omega_J(g_0^{-\iota})$ and the Hirzebruch--Zagier cycle $\mathscr{Z}_K^{(\underline{k},0)}$ gives 
\begin{align}\label{intformulaintro}
    \langle \omega_J(g_0^{-\iota}),  \mathscr{Z}_K^{(\underline{k},0)}\rangle = C \cdot {\rm Res}_{s=1} L({\rm As}(g_0),s), 
\end{align}
with $C$ an explicit non-zero constant and $L({\rm As}(g_0),s)$ the (partial) Asai $L$-function of $g_0$. Note that the non-triviality of the left hand side implies that the form $g_0$ is $H$-distinguished, which is equivalent to $g$ being a base change form on $H$. Indeed, an analysis of the possible pole at $s=1$ (see \cite[Theorem 10.1]{GetzGore}) shows that $L({\rm As}(g_0),s)$ has a (simple) pole at $s=1$ if and only if $g$ is the base change of a cusp form $f$ on $H$ with trivial Nebentypus, in which case \[{\rm Res}_{s=1} L({\rm As}(g_0),s) = C' \cdot L({\rm Ad}(f) \otimes \eta_{E/F},1), \]
where $L({\rm Ad}(f) \otimes \eta_{E/F},s)$ is the (partial) adjoint $L$-function of $f$ twisted by the Hecke character $\eta_{E/F}$ and $C'$ is an explicit non-zero constant coming from the residue at $s=1$ of a partial Dedekind zeta function. Moreover, as shown by Hida in \cite{Hidanoncritical}, after appropriately normalizing the differential form and cycle, the pairing in \eqref{intformulaintro} recovers the algebraic part of the adjoint $L$-value $L({\rm Ad}(f) \otimes \eta_{E/F},1)$.\\ 

Let $p$ be an odd prime which is unramified in $E$. Our main results may be summarized briefly as follows : we investigate the $p$-adic interpolation of Hirzebruch--Zagier cycles and, by exploiting the theory of base change and the integral formula \eqref{intformulaintro}, we construct a $p$-adic $L$-function that interpolates the values $L({\rm Ad}(f) \otimes \eta_{E/F},1)$ as $f$ varies in a Hida family of Hilbert modular forms on $H$.

Regarding our first main result, we prove the existence of a Big Hirzebruch--Zagier cycle in middle-dimensional (ordinary) Iwasawa cohomology. 
To achieve this, we use the methods introduced in \cite{LoefflerSphericalvarieties} and \cite{LRZ}. We would like to point out that the results in \emph{loc. cit.} do not apply directly to our setting exactly because we work with Hilbert modular varieties that do not satisfy Milne's axiom SV5. We however provide the necessary details to show that this hypothesis can be removed in our case. 

Let $L$ be a finite extension of $\Q_p$ over which $G$ and $H$ split and let $\mathcal{O}$ be its ring of integers. As $(G,H)$ is a spherical pair, we can choose $u \in G(\Z_p)$ such that $H(\Z_p) \cdot u \cdot \overline{B}_G(\Z_p)$ is dense and open in $G(\Z_p)$, with $\overline{B}_G(\Z_p)$ the lower-triangular Borel subgroup of $G$. Following the recipe in \cite[\S 4.4]{LoefflerSphericalvarieties}, from this open orbit, we define a tower of level subgroups $V_{\pmb{r}} = K^{(p)} V_{\pmb{r},p}$, for suitable neat open compact $K^{(p)} \subseteq G(\hat{\Z}^{(p)})$, indexed by $\pmb{r} = (r_\frakp)_\frakp \in \Z_{\geq 0}^{|\Upsilon_{F,p}|}$, where $\Upsilon_{F,p}$ denotes the set of primes of $F$ above $p$. Along this tower, we consider the  Hecke operator $T_\eta = \prod_{\frakp \in \Upsilon_{F,p}} T_{\eta_\frakp}$, where $T_{\eta_\frakp}$ is the suitably normalised correspondence $[V_{\pmb{r}} \eta_\frakp^{-1} V_{\pmb{r}}]$ and $\eta_\frakp = \left( \begin{smallmatrix}
    \pi_\frakp & \\ & 1
\end{smallmatrix} \right) \in \GL_2(F_\frakp)$ (see Remark \ref{rmk:normofalletap}). We then have the middle degree ordinary Iwasawa cohomology group  
 \[
 e^{\rm n.ord} H^{2d}_{\rm Iw}(Y_G(V_\infty), \mathcal{O})  \cong e^{\rm n.ord} H^{2d}(Y_G(K_0(p^2)),\Lambda_{G/H}),
 \]
where  $K_0(p^2) = K^{(p)} \Gamma_0(p^2)$, $\Lambda_{G/H}$ is the Iwasawa algebra introduced in \eqref{eqdefIwalg}, and $ e^{\rm n.ord}$ is the ordinary projector associated to $T_\eta$. Interestingly, $\Lambda_{G/H}$ is closely related to the Iwasawa algebra  $\Lambda_{H/Z_H} = \mathcal{O}\llbracket T_H(\Z_p)/Z_H(\Z_p)\rrbracket$, where $Z_H$ denotes the center of $H$, and isomorphic to it if Leopoldt conjecture for $E$ holds. We note that the Hecke eigensystems attached to the base change of nearly ordinary cusp forms on $H$ with trivial Nebentypus contribute to $e^{\rm n.ord} H^{2d}_{\rm Iw}(Y_G(V_\infty), \mathcal{O})$. 

Twisting the Hirzebruch--Zagier cycles  of weight $ (\underline{k},0)$ with the operator $u \cdot \prod_{\frakp \in \Upsilon_{F,p}}  \eta_\frakp^{r_\frakp} \in G(\Q_p)$, we obtain cohomology classes \[\mathscr{Z}_{\pmb{r}}^{\underline{k}} \in H^{2d} ( Y_G( V_{\pmb{r}}), \mathscr{V}^{\underline{k},0}_{\mathcal{O}})\] 
and show that they can be packaged into an Iwasawa cohomology class. Namely, we prove the following , see Theorem \ref{BIGCLASS} and Corollary \ref{coro:EulerFactor1}.

\begin{theorem}\label{mainthm1}

There is a Big Hirzebruch--Zagier class \[ \mathscr{Z}_{\infty} \in  e^{\nord} H^{2d}_{\rm Iw}(Y_G(V_\infty), \mathcal{O}). \]
Let $\underline{k} \in \Z^{\Sigma_E}$ be such that $k_\sigma \equiv 0$ (mod $2$) and $k_\sigma = k_{\sigma'}$ for all $\sigma \in \Phi$. Then, \begin{enumerate}
    \item for every $\pmb{r} \in \Z_{\geq 1}^{|\Upsilon_{F,p}|}$, we have
\[
\mathrm{mom}^{\underline{k}}_{\pmb{r}}\, \mathscr{Z}_{\infty}= T_{\eta}^{-\pmb{r}} e^{\nord} \,\mathscr{Z}_{\pmb{r}}^{\underline{k}},
\]
with $T_{\eta}^{-\pmb{r}} = \prod_\frakp T_{\eta_\frakp}^{-r_\frakp}$;
\item if $\pmb{r}=\pmb{0}$,  
\[\mathrm{mom}^{\underline{k}}_{0}\, \mathscr{Z}_{\infty} =  \prod_{\frakp \in \Upsilon_{F,p}} \left( 1 - q_{\frakp} \cdot p^{\sum_{\sigma \in \Sigma_{\frakp}}  k_{\sigma}}\, T_{\eta_{\frakp}}^{-1} \right) \cdot 
e^{\nord} \mathscr{Z}_0^{\underline{k}}.\]
\end{enumerate} 
\end{theorem}

Our second main result is the construction of the adjoint $p$-adic $L$-function, interpolating the values $L(\mr{Ad}(f) \otimes \eta_{E/F},1)$ as $f$ varies in a nearly ordinary Hida family for $F$ in $d$ variables. Let $\mathbf{f}$ be a Hida family of tame character $\chi_F$ of cusp forms on $H$,  and $\mathbf{f}_E$ its base change to $E/F$ of tame character $\chi = \chi_F \circ {\rm N}_{E/F}$, see Definition \ref{def:BaseChangefamily}. For a unitary Hecke character  $\theta$ of $\A_E^\times$ with conductor coprime with $p$ and such that its restriction to $\A_F^\times$ is $\chi_F \eta_{E/F}$, we consider the twist $\mathscr{Z}_{\infty,\theta}$ of the Big Hirzebruch--Zagier using the correspondence of \cite[Proposition 9.5]{GetzGore}. The $p$-adic $L$-function attached to $\mathbf{f}$ is defined as the image of the Big Hirzebruch--Zagier class $\mathscr{Z}_{\infty,\theta}$ under a suitable idempotent associated to $\mathbf{f}_E$. 

\begin{theorem}\label{mainthm2}
   Let $\mathbf{f}$ a Hida family of cuspidal Hilbert modular forms for $F$ in the sense of Definition \ref{def:HidaFamilies} and let $\mathcal{K}_{\mathbf{f}}$ be the fraction field of the  coefficient of $\mathbf{f}$. Suppose that the image of its residual Galois representation is not solvable. Then there exist a non-zero element of $\mathcal{K}_{\mathbf{f}}$, denoted by $L_p(\mathrm{Ad}(\mathbf{f})\otimes \eta_{E/F})$, which  interpolates the algebraic part of the special value $L(\mathrm{Ad}(f) \otimes \eta_{E/F}, s)$ at $s=1$.
\end{theorem}

We call $L_p(\mathrm{Ad}(\mathbf{f})\otimes \eta_{E/F})$ the adjoint $p$-adic $L$-function twisted by $\eta_{E/F}$ of $\mathbf{f}$ because it satisfies the expected interpolation properties at classical weights. We refer to Theorem  \ref{thmmain} for the precise statement of the interpolation formulae as well as the description of the Euler factors at $p$ appearing therein. 

We expect the Hirzebruch--Zagier theorem to be naturally compatible with the $p$-adic variation of Hirzebruch--Zagier divisors. As a step in that direction, in Remark \ref{BeforeBigPairing}, we define an element of $e^{\nord} H^{2d}_{\rm Iw}(Y_G(V_\infty), \mathcal{O}_{\mathfrak{m}}) \otimes_{\Lambda_{G/H}} \mathcal{K}$, where $\mathcal{K}$ is a large enough extension of $\mathrm{Frac}(\Lambda_{G/H})$,  by taking the sum of the adjoint $p$-adic $L$-functions of each primitive base-change family contributing to the Iwasawa cohomology group. This can be seen as a family version of the element $\Phi_{Q([Z]),\chi_E}$ of \cite[Theorem 8.4]{GetzGore}. It would be interesting to upgrade Remark \ref{BeforeBigPairing} and construct a $\Lambda_{H/Z_H}$-adic Hilbert modular form valued in the ordinary Iwasawa cohomology group studied here. We shall come back to this in future work. Note that the $p$-adic variation of the Gross--Kohnen--Zagier theorem via Big Heegner points has been treated by Longo and Nicole in \cite{LongoNicole1}, and \cite{LongoNicole2}.

\subsection{Relation to previous literature} 
 
 The adjoint $p$-adic $L$-functions in several variables for Hilbert modular forms and their twist interpolating the values of the adjoint $L$-function at Deligne critical values have been constructed in \cite{H6, Wu, RossoIsrael}, see also \cite{RosTh} for the interpolation of critical values on the right side of the center of the functional equation. Still, due to the presence of trivial zeroes, these $p$-adic $L$-functions generally vanish at the near central value $0$ (or $1$).  
 
 If one removes the cyclotomic variable, than one can get rid of the trivial zeroes and interpolate the special value at $1$: in the literature, this has been done by interpreting the congruence module of a family as the adjoint $p$-adic $L$-function, see for example \cite{H2, Bellaiche}. This is done thanks to a formula of Shimura that interprets the Petersson norm of a Hilbert modular form $f$ as the algebraic part of $L(\mr{Ad}(f),1)$ and, in turn, the congruence module can be related to the (algebraically normalised) Petersson norm. 
 Hence, the congruence module corresponds, up to a $p$-adic unit, with the algebraic part of $L(\mr{Ad}(f),1)$ and thence the name of $p$-adic adjoint $L$-function for the congruence module, see also \cite{DimitrovENS,GhateCong,BLP} for the generalisation to Hilbert modular forms. 
 
 In our paper, we construct a many-variable $p$-adic $L$-function interpolating the values $L(\mr{Ad}(f) \otimes \eta_{E/F},1)$, for $E$ a fixed quadratic totally real extension of $F$, $\eta_{E/F}$ the adelic character of $F$ corresponding to this extension under class field theory, and where $f$ varies in a nearly ordinary Hida family for $F$ in $d$ variables. We do not use an interpretation of these values as a congruence number but rather use an analytic formula of Hida \cite{Hidanoncritical}. Still, note that a suitable integral normalisation of these values should control congruences between the base change of $f$ to $E$ and non-base change forms, see \cite[Conjecture 1.1]{DHI}.
 
A rather similar approach as ours albeit without the construction of the big class in Iwasawa cohomology has been used for to construct a $p$-adic $L$-function interpolating $L(\mr{Ad}(f) \otimes \eta_{E/\mathbb{Q}},1)$ where $E$ is a quadratic imaginary field, see \cite{PakInLee}. See also \cite{ForneaJin} where, building on ideas of Darmon and Rotger e.g., \cite{DarmonRotgerAstDiag}, the authors construct $p$-adic families of twisted diagonal cycles associated to the extension $(E \times \Q)/\Q$, where $E$ is a real quadratic field. Similar techniques apply also for other cycles and/or automorphic $L$-functions, see for example \cite{BSV,LRZ}. 
We finally note that that the techniques introduced in \cite{Rockwoodnonord} could be used to generalize Theorem \ref{mainthm1} and Theorem \ref{mainthm2} to non-ordinary situations, which we plan to explore in future work.

\subsection{Structure of the manuscript}

We describe the content of the paper. In Section 2, we present some prolegomenas on algebraic representations and associated local systems with emphasis on a description via an explicit invariant vector of the branching law associated to the diagonal embedding associated to the relatively quadratic totally real field extension $E/F$. In Section 3, we recall basic notations on Hilbert modular forms and differential forms, $L$-functions, and adjoint $L$-values. In Section 4, we describe the $p$-adic tower of $p$-power level subgroups and the moment map allowing to pass from the limit in cohomology to any finite level. We conclude the section with a Control Theorem for this $p$-adic tower of level subgroups, see Theorem \ref{ControlTheorem}. In Section 5, we consider the Hilbert modular varieties of $p$-power level specified in the previous Section. We describe the Hirzebruch--Zagier cycles arising therein, forming a Big Hirzebruch--Zagier class, see Theorem \ref{BIGCLASS}. We also compare our cycles with the cycles of Getz and Goresky, see Theorem \ref{proponV1toV0}. In Section 6, we recall briefly elements of Hida theory for Hilbert modular forms and define base change Hida families. In Section 7, as a first application of the existence of the $p$-adic family of Hirzebruch--Zagier cycles developed in Section 5, we construct geometrically the $p$-adic adjoint $L$-function, see Theorem \ref{thmmain} in the text. 

\subsection{Acknowledgements}
We thank Raul Alonso, Matteo Longo, Peter Neamti, and Ju-Feng Wu for useful conversations related to this project. We are also indebted to the anonymous referee, whose comments and suggestions have led to many improvements of the article. M.-H. N. is pleased to acknowledge the financial support by the Laboratoire CRM-CNRS IRL (ex-UMI) 3457 for his extended stay in Winter 2021 at the Centre de recherches math\'ematiques (CRM), and for three short visits in 2022, 2023 and 2024.

\section{Preliminaries on algebraic representations and their local systems}\label{secNotation}

Let $F$\index{$F$} be a totally real field over $\Q$ of degree $[F:\Q]=d$\index{$d$}.  Denote $\Sigma_F$\index{$\Sigma_F$} the set of archimedean places of $F$. Let $p$ be an odd prime number. We fix an isomorphism $ i_p : \C \cong \overline{\Q}_p $\index{$i_p$}, in order to identify elements in $\Sigma_F$ with $p$-adic places eventually. We consider $E/F$\index{$E$} a quadratic extension of totally real fields, and suppose that $E=F(\sqrt{\Delta})$, for a totally positive $\Delta \in F$\index{$\Delta$}. We suppose henceforth that $p$ is an unramified prime in $E$. In analogy with the notion of a CM-type, we say that a subset $J$\index{$J$} of $\Sigma_E$ is a $F$-type for $E$ if for every $\sigma \in \Sigma_F$ there is exactly one embedding extending $\sigma$ in  $J$, so $\Sigma_E = J \sqcup \overline{J}$ with respect to real conjugation in ${\rm Gal}(E/F)$. We denote by $\mathcal{O}_E$ (resp. $\mathcal{O}_F$) the ring of integers of $E$ (resp. $F$) and  we write $D_E$\index{$D_E$} (resp. $D_F$\index{$D_F$}) for the (absolute) discriminant of $E$ (resp. $F$) over $\mathbb{Q}$. We also write $\mathfrak{d}_{E/F}$\index{$\mathfrak{d}_{E/F}$} for the discriminant of $E$ over $F$, which is an ideal of $\mathcal{O}_F$.

Let ${\rm Gal}(E/F) = \langle \tau \rangle$. For every $ \sigma \in \Sigma_E $, we denote the pre-composition by $\tau$ by $\sigma' := \sigma \circ \tau$. We define the $F$-type $\Phi$ for $E$ as 
\begin{align}\label{Ftype}
     \Phi :=\{ \sigma \in {\rm \Sigma_E} \,:\,\sigma(\sqrt{\Delta}) < 0 \}. 
\end{align} 
Then 
\begin{align*}
    {\rm \Sigma_E} = \Phi \sqcup \Phi',
\end{align*} with $\Phi' := \{ \sigma'\}_{\sigma \in \Phi}$. 

\subsection{Groups and varieties} 

Let ${\GL_2}$ be the group scheme over $\Z$ of invertible 2-by-2 matrices. 
For $E/F$ a quadratic extension of totally real fields, we let \[H \index{$H$}:= {\rm Res }_{\mathcal{O}_F/\Z} \GL_2,\] \[ G\index{$G$}:= {\rm Res }_{\mathcal{O}_E/\Z} \GL_2. \]
\noindent We denote by $Z_{H}, Z_G$ the centers of $H$ and $G$ respectively. 

Note that we have the natural diagonal embedding $\iota: H \hookrightarrow G$ obtained from viewing $G$ as the restriction of scalars of $H$ with respect to $E/F$. Now, let $L$ be a field over which $H$ and $G$ are split. Then, $\iota : H_L \to G_L $ is induced by the diagonal embedding \[ \prod_{\sigma \in \Sigma_F} \GL_{2,L}  \hookrightarrow   \prod_{\sigma \in \Phi} \Big( \GL_{2,L}  \times \GL_{2,L} \Big), \; (g_\sigma)_\sigma \mapsto (g_\sigma,g_{\sigma'})_\sigma,\]
where $\Phi$ is the $F$-type defined by \eqref{Ftype}.

We denote by $Y_H$ (resp. $Y_G$) the Hilbert modular variety associated with $H$ (resp. $G)$. Recall that the associated Hermitian symmetric domains are the product of $[F:\Q]=d$ (resp. $[E:\Q]=2d$) copies of the Poincar\'e upper-half plane $\mathfrak{H}$.

\begin{definition}
 We define $K_{\infty}^+$  to be the connected component  containing the identity of the maximal compact subgroup $K_\infty$ of $G(\R)$. Let $K$ be a compact-open subgroup of $G(\A_{f})$. We write 
 \[
 G(\A_f)= \bigcup_i  G(\Q) t_i K K_{\infty}^+.
 \]
 We say that $K$ is neat if 
 \[
 \overline{\Gamma}_i(K):={\rm Im}( t_i K K_{\infty}^+t_i^{-1} \cap G(\Q) \rightarrow {\rm PGL}_2 (E))
 \]
 acts without fixed points on $\mathfrak{H}^{2d}$.
\end{definition}
An analogous definition can be made for a compact-open subgroup of $H(\A_{f})$. When $K$ and $K'$ are neat compact-open subgroups respectively of $G(\A_{f})$ and $H(\A_{f})$, $Y_G(K)$ and $Y_H(K')$ are smooth quasi-projective schemes over ${\rm Spec}\, \Q$ of dimension $2d$ and $d$ respectively. 

\subsection{Representations}\label{algebraicreps}

From now on, $L$ denotes a field for which $H$ and $G$ are split. Let $V$ be the standard representation  $\textbf{GL}_{2,L}$ acting on the left on $V \simeq L^2$ with given basis $(e_1, e_2)$, and denote ${\rm det} := \bigwedge^2 V$ by minor abuse of notation. Every irreducible algebraic representation of $\textbf{GL}_{2,L}$ is of the form $V^{k,m}:= {\rm Sym}^k(V) \otimes {\rm det}^m$, with $(k,m) \in \Z_{\geq 0} \times \Z$. Here we denote by ${\rm Sym}^k(V)$ the quotient module of co-invariants of $V^{\otimes k}$ for the action of the symmetric group $\mathfrak{S}_k$ on $k$ elements.

\begin{definition} For $ (\underline{k},m) \in (\Z_{\geq 0})^{\Sigma_E} \times \Z$ s.t. $k_\sigma \equiv m$ (mod $2$), where we write $\underline{k} = (k_\sigma)_{\sigma \in \Sigma_E}$, we define  \[V^{\underline{k},m} :=  \bigotimes_{\sigma \in \Sigma_E}  {\rm Sym}^{k_\sigma} (V_\sigma) \otimes \mathrm{det}_\sigma^{\frac{m-k_{\sigma}}{2}}.\]
\end{definition}
\begin{remark}
By a slight abuse of notation, we shall consider algebraic representation of $H$ of weights   $ (\underline{k},m) \in (\Z_{\geq 0})^{\Sigma_F} \times \Z$   that will be denoted by the same symbol $V^{\underline{k},m} $. The weight $\underline{k}$ and the context will make clear if we are working with representations of $H$ or $G$. 
\end{remark}
\begin{remark}
The highest weight vector of $V^{\underline{k},m}$ is given by $\bigotimes_\sigma e_{1,\sigma}^{\otimes k_\sigma}$, where $(e_{1,\sigma},e_{2,\sigma})$ is the standard basis of each $V_\sigma$.     
\end{remark}
\begin{remark}
The contragredient of 
$V^{\underline{k},m}$ is isomorphic to $V^{\underline{k},-m}$. 
\end{remark}

\noindent We can realize $V^{\underline{k},m}$ explicitly as follows. The $L$-vector space $L[X_\sigma,Y_\sigma]^{=k_\sigma}$ of homogeneous polynomials of degree $k_\sigma$ is a model for the dual representation of ${\rm Sym}^{k_\sigma} (V_\sigma)$. The element $g \in \textbf{GL}_{2,L}$ acts on $L[X_\sigma,Y_\sigma]^{=k_\sigma}$ via the left action given by
\[
g \cdot P(X,Y)=\mathrm{det}(g)^{-k_\sigma} P\left( \left(g^* \left(\begin{array}{c}
 X \\
 Y 
\end{array} \right)\right)^t \right)=\mathrm{det}(g)^{-k_\sigma} P(dX-bY,-cX+aY), 
\]
where for $g=\left(\begin{array}{cc}
a & b \\
c & d
\end{array} \right)$, $g^*=\mathrm{det}(g)g^{-1}=\left(\begin{array}{cc}
d & -b \\
-c & a
\end{array} \right)$.  Observe that, as $\mathrm{SL}_{2,L}$-representations, ${\rm Sym}^{k_\sigma} (V_\sigma)$ and $L[X_\sigma,Y_\sigma]^{=k_\sigma}$  are isomorphic via sending \[
\alpha e_1+\beta e_2  \mapsto -\beta X+\alpha Y.
\]
This implies that $L[X_\sigma,Y_\sigma]^{=k_\sigma} \simeq {\rm Sym}^{k_\sigma} (V_\sigma) \otimes {\rm det}_\sigma^{-k_\sigma}$ as $\textbf{GL}_{2,L}$-representations. Via this identification, $e_{1,\sigma}$ hence $Y_\sigma$ are the highest weight vectors for $V_\sigma$ and its dual. This implies that the dual of $\bigotimes_{\sigma \in \Sigma_E}  {\rm Sym}^{k_\sigma} (V_\sigma)$ as a $G_L$-representation is isomorphic to $\bigotimes_{\sigma \in \Sigma_E} L[X_\sigma,Y_\sigma]^{=k_\sigma}$.

We introduce some useful lattices in ${\rm Sym}^k(V)$. Let $\mathcal{O} := \mathcal{O}_L$ be the ring of integers of $L$, and identify $\mathcal{O}^2$ with $\mathcal{O}e_1 \oplus \mathcal{O}e_2$ inside $V$. We have

\[
{\rm TSym}^k(\mathcal{O}^2) \subset {\rm Sym}^k(\mathcal{O}^2)\subset {\rm Sym}^k(V).
\]

\noindent
Recall that ${\rm Sym}^k(\mathcal{O}^2)$ is the module of $\mathfrak{S}_k$-coinvariants, and ${\rm TSym}^k(\mathcal{O}^2)$ is the submodule of $\mathfrak{S}_k$-invariants in ${(\mathcal{O}^2)}^{\otimes k}$, that we then identify with its image in ${\rm Sym}^k(\mathcal{O}^2)$. Recall that a lattice in ${\rm Sym}^k(V)$ is \emph{admissible} if it is invariant under the Kostant $\Z$-form in the universal enveloping algebra of the Lie algebra of $\textbf{GL}_{2,L}$. Every admissible lattice is a direct sum of its weight components, thus an admissible lattice in ${\rm Sym}^k(V)$ is proportional to one whose intersection with the highest weight subspace is $\mathcal{O} \cdot e_1^{\otimes k}$.  Hence, without loss of generality, we can suppose this intersection to be precisely $\mathcal{O} \cdot e_1^{\otimes k}$. Finally, recall that any admissible lattice is invariant under the action of $\textbf{GL}_{2,\mathcal{O}}$ (for instance, see \cite{LinZZ} for more details). Note that ${\rm Sym}^k(\mathcal{O}^2)$ is the maximal admissible lattice in ${\rm Sym}^k(V)$, while ${\rm TSym}^k(\mathcal{O}^2)$ is the minimal admissible one. Indeed, by duality, we have two lattices in $L[X,Y]^{=k}$: $\mathcal{O}[X,Y]^{=k}$, the set of homogeneous polynomials of degree $k$ with integer coefficients, which is a maximal lattice and it is dual to ${\rm TSym}^k(\mathcal{O}^2)$,
and $\mathcal{O}\left[ \binom{k}{i} X^{k-i}Y^i \right]$, which is a minimal lattice and is dual to ${\rm Sym}^k(\mathcal{O}^2)$. It coincides with the duality of \cite[\S 7.5]{GetzGore} (under the isomorphism above).

Given a finite order Hecke character $\chi$ of $E$ of conductor dividing an ideal $\mathfrak{c}$ of $\mathcal{O}_E$, we can define a character of \[K_0(\mathfrak{c}) : = \{ g \in \GL_2(\widehat{\mathcal{O}}_E)\,:\, g = {\matrix{a}{b}{c}{d}}, \text{ with } c \equiv 0 \text{ mod } \mathfrak{c}\}\]
via
${\matrix{a}{b}{c}{d}} \mapsto \chi(d)$ and consequently a one dimensional representation $L({\chi})$ of $K_0(\mathfrak{c})$. Following \cite[\S 6.8]{GetzGore}, we denote $V^{\underline{k},{m}}(\chi) :=V^{\underline{k},{m}} \otimes_L L(\chi)$. Explicitly, for $v \in V^{\underline{k},{m}}(\chi)$ and $g \in K_0(\mathfrak{c})$, we have $g.v= \chi(g) v$.

We conclude by defining two more representations of $G(L)$, that are not necessarily algebraic.

\begin{definition} 
For $ (\underline{k},m) \in (\Z_{\geq 0})^{\Sigma_E} \times \Z$   we define the complex representation of $G(L)$  \[V^{\underline{k},m}_\C :=  \bigotimes_{\sigma \in \Sigma_E}  {\rm Sym}^{k_\sigma} (V_\sigma \otimes \C) \otimes \mathrm{det}_\sigma^{\frac{m-k_{\sigma}}{2}}.\]

If, instead, we consider a $p$-adic field $L_{\mathfrak{m}}$ that contains all the square roots of elements of $E \otimes_{\Z} \Q_p$ (note that this is a finite extension of $\Q_p$) and denote by $\mathcal{O}_{\mathfrak{m}}$ its ring of integers, we define, for $ (\underline{k},m) \in (\Z_{\geq 0})^{\Sigma_E} \times \Z$, the $p$-adic representation of $G(\mathcal{O})$ \[V^{\underline{k},m}_{\mathcal{O}_{\mathfrak{m}}} :=  \bigotimes_{\sigma \in \Sigma_E}  {\rm Sym}^{k_\sigma} (\mathcal{O}^2 \otimes \mathcal{O}_{\mathfrak{m}}) \otimes \mathrm{det}_\sigma^{\frac{m-k_{\sigma}}{2}}.\]
\end{definition}

\subsection{Branching laws}\label{branchinglaws}

Recall that $\Phi$ denotes the $F$-type for $E$ given by \eqref{Ftype} and that the embedding $\iota:H_L \to G_L $ is induced by the diagonal embedding: \[ \prod_{\sigma \in \Sigma_F} \GL_{2,L}  \hookrightarrow   \prod_{\sigma \in \Phi} \Big( \GL_{2,L}  \times \GL_{2,L} \Big), \; (g_\sigma)_\sigma \mapsto (g_\sigma,g_{\sigma'})_\sigma.\]

Given the $G_L$-representation $V^{\underline{k},m}$, we can view it as an $H_L$-representation via the embedding $\iota$. It  decomposes into a sum of irreducible sub-$H_L$-representations. In the following, we characterize the set of representations of $G_L$ which contain a trivial $H_L$-representation.

\begin{lemma}\label{lemma:branching} The representation $V^{\underline{k}, m}_{|_{H_L}}$ contains a unique copy of the trivial $H_L$-representation (up to a twist) if and only if $k_\sigma = k_{\sigma'}$ for all $\sigma \in \Phi$. In which case, we have an embedding 
\[ \iota^{\underline{k}} : \bigotimes_{\sigma \in \Sigma_F} {\rm det}_\sigma^{m} \hookrightarrow V^{\underline{k},m},\]
explicitly realized by the vector \begin{equation} \label{InvVectus}
    \Delta^{\underline{k}}:=\bigotimes_{\sigma \in \Phi}(e_{1,\sigma}e_{2,\sigma'}-e_{2,\sigma}e_{1,\sigma'})^{\otimes k_\sigma} \in V^{\underline{k},m}.
\end{equation} 
\end{lemma}

\begin{proof}
Recall that we have the decomposition as $\mathrm{SL}_2(L)$-representations \[(\operatorname{Sym}^{k}(V) \otimes \operatorname{Sym}^{j}(V))_{|_{\mathrm{SL}_2(L)}} = \oplus_{r = 0}^{{\rm min}(k,j)} \operatorname{Sym}^{k+j-2r}(V),\] 
see \cite[Exercise 11.11]{FultonHarris}.
This in turn implies that \[ \bigotimes_{\sigma \in \Phi} (\operatorname{Sym}^{k_\sigma}(V_\sigma) \otimes \operatorname{Sym}^{k_{\sigma'}}(V_{\sigma'}))_{|_{\mathrm{SL}_2(L)}} = \bigotimes_{\sigma \in \Sigma_F} \bigoplus_{r_\sigma = 0}^{{\rm min}(k_\sigma,k_{\sigma'})} \operatorname{Sym}^{k_\sigma + k_{\sigma'} - 2 r_\sigma }(V_\sigma). \]
Hence we have an embedding of the trivial $\mathrm{SL}_2(L)^d$-representation \[ \mathbf{1} \hookrightarrow \bigotimes_{\sigma \in \Phi} (\operatorname{Sym}^{k_\sigma}(V_\sigma) \otimes \operatorname{Sym}^{k_{\sigma'}}(V_{\sigma'})) \]
if and only if $k_\sigma = k_{\sigma'}$ for all pairs of conjugate embeddings $\sigma, \sigma'$, as $k_{\sigma} + k_{\sigma'} - {2 \rm min} (k_\sigma,k_{\sigma'}) = 0$ iff $k_{\sigma} = k_{\sigma'}$ for all such pairs. By matching the central characters, we get an embedding of $H_L$-representations
\[ \iota^{\underline{k}} : \bigotimes_{\sigma \in \Sigma_F} {\rm det}_\sigma^{m} \hookrightarrow \bigotimes_{\sigma \in \Phi} (\operatorname{Sym}^{k_\sigma}(V_\sigma) \otimes \operatorname{Sym}^{k_{\sigma}}(V_{\sigma}) \otimes {\rm det}_\sigma^{m- k_\sigma}) = V^{\underline{k},m}_{|_{H_L}}. \]
Finally, note that the invariant vector (\textit{cf}. $\Delta$ in \cite[Proposition 9.1]{GetzGore}) which generates the image of $\iota^{\underline{k}}$ is explicitly given by: 
  \begin{equation*}
    \bigotimes_{\sigma \in \Phi}(e_{1,\sigma}e_{2,\sigma'}-e_{2,\sigma}e_{1,\sigma'})^{\otimes k_\sigma}.
\end{equation*}
Indeed, if $g = {\matrix{a}{b}{c}{d}} \in \GL(V_\sigma)$, we have that
\begin{align*}
(g,g)(e_{1,\sigma}e_{2,\sigma'}-e_{2,\sigma}e_{1,\sigma'}) = & 
(ae_{1,\sigma}+ce_{2,\sigma})(be_{1,\sigma'}+d e_{2,\sigma'})- (be_{1,\sigma}+d e_{2,\sigma})(ae_{1,\sigma'}+ce_{2,\sigma'})\\
= \  & ab e_{1,\sigma}e_{1,\sigma'}+cde_{2,\sigma}e_{2,\sigma'}+ade_{1,\sigma}e_{2,\sigma'} - (-bc)e_{2,\sigma}e_{1,\sigma'} +\\
& - ab e_{1,\sigma}e_{1,\sigma'}-cde_{2,\sigma}e_{2,\sigma'}-bce_{1,\sigma}e_{2,\sigma'}-ade_{2,\sigma}e_{1,\sigma'}\\
= \ &\mathrm{det}(g)(e_{1,\sigma}e_{2,\sigma'}-e_{2,\sigma}e_{1,\sigma'}).
\end{align*}

\end{proof}

\begin{remark}
Lemma \ref{lemma:branching} can be rephrased by saying that \[{\rm dim}( V^{\underline{k},m} \otimes \bigotimes_{\sigma \in \Sigma_F} {\rm det}_\sigma^{-m} )^H  \leq 1\] and it is equal to 1 if and only if $k_\sigma = k_{\sigma'}$ for every $\sigma \in \Phi$.     
\end{remark}

\begin{lemma}\label{invariantvectorintegral}
Suppose that $k_\sigma = k_{\sigma'}$ for every $\sigma \in \Phi$. Then $\iota^{\underline{k}}$ sends the $\mathcal{O}$-span of $\Delta^{\underline{k}}$ inside an admissible lattice of $V^{\underline{k},m}$ and it generates the $H$-invariants of this lattice.
\end{lemma}

\begin{proof}

Expanding $\Delta^{\underline{k}}$, we get 
\[
\Delta^{\underline{k}}= \bigotimes_{\sigma \in \Phi} \sum_{i=0}^{k_{\sigma}}(-1)^i \binom{k_{\sigma}}{i} e_{1,\sigma}^{\otimes^ {k_{\sigma}-i}}e_{2,\sigma}^{\otimes^i}e_{1,\sigma'}^{\otimes^i}e_{2,\sigma'}^{k_{\sigma}-i}
\]
which is a vector in $\left( \bigotimes_{\sigma \in \Phi} \mathrm{Sym}^{k_\sigma}(\mathcal{O}^2) \otimes \det_\sigma^{\frac{m-k_{\sigma}}{2}} \right) \otimes \left(\bigotimes_{\sigma' \in \Phi'} \mathrm{Sym}^{k_{\sigma'}}(\mathcal{O}^2) \otimes \det_{\sigma'}^{\frac{m-k_{\sigma'}}{2}}\right)$.
\end{proof}

\begin{remark}
    Note that in general $\Delta^{\underline{k}}$ is not in the minimal admissible lattice
\[
\bigotimes_\sigma \left( \mathrm{TSym}^{k_\sigma}(\mathcal{O}^2) \otimes \det_{\sigma}^{\tfrac{m-k_{\sigma}}{2}}\right) \otimes \left(\bigotimes_{\sigma'} \mathrm{TSym}^{k_{\sigma'}}(\mathcal{O}^2) \otimes \det_{\sigma'}^{\tfrac{m-k_{\sigma'}}{2}}\right).
\]
\end{remark}

\subsection{Cohomology}\label{sec:cohomology}

Associated with the irreducible algebraic representation $V^{\underline{k},{m}}$ and its maximal admissible lattice $V^{\underline{k},{m}}_{\mathcal{O}}$,  we have two local systems $\mathscr{V}^{\underline{k},{m}}$ and $\mathscr{V}^{\underline{k},{m}}_{\mathcal{O}}$ on $Y_G(K)$, which can be defined as in \cite[\S 6.8]{GetzGore}, as long as $K \subset G(\A_f)$ is a neat open compact subgroup. Similarly, by \cite[Proposition 6.3]{GetzGore}, we can define local systems\footnote{In \emph{loc.cit.}, the authors use a different normalization for the weights of the algebraic representations of $G$. They write the weights as $\kappa = (\underline{t},\underline{s}) \in (\Z_{\geq 0})^{\Sigma_E} \times (\tfrac{1}{2}\Z)^{\Sigma_E}$ such that $\underline{t}+2\underline{s} \in \Z \mathbf{1}$. Translating from one convention to the other, our representation $V^{\underline{k},{m}}$ corresponds to the dual of $L(\kappa)$, where, in the notation of \emph{loc.cit.}, $\kappa = (\underline{t},\underline{s})$ is given by $\underline{t} = \underline{k}$ and  $\underline{s} = \left( \frac{m - k_\sigma}{2}  \right)_\sigma$. Similarly, $V^{\underline{k},{m}}(\chi)$ is the dual of $L(\kappa,\chi)$. Thus, the local system $\mathcal{L}(\kappa,\chi)$ of \emph{loc.cit.} coincides with our $\mathscr{V}^{\underline{k},{-m}}(\chi^{-1})$.} $\mathscr{V}^{\underline{k},{m}}(\chi)$ on $Y_G(K_0(\mathfrak{c}))$ attached to a finite order Hecke character $\chi$  of $E$ of conductor dividing $\mathfrak{c}$ and the representation $V^{\underline{k},{m}}(\chi)$ of $K_0(\mathfrak{c})$. Similarly, we can define local systems over $Y_H(K')$. 

Let $\underline{k} \in \Z^{\Sigma_E}$ be such that $k_\sigma = k_{\sigma'}$ for all $\sigma \in \Phi$, then the branching map   $\iota^{\underline{k}}_{\mathcal{O}} : \mathcal{O} \hookrightarrow V^{\underline{k},0}_{\mathcal{O}}$ induces a map of local systems 
\[ \iota^{\underline{k}}_{\mathcal{O}} : \mathcal{O} \to \iota^*\mathscr{V}^{\underline{k},0}_{\mathcal{O}}.\] 
Twisting the vector $\Delta^{\underline{k}}$ by the finite order Hecke character $\chi$ as in \cite[(9.2.5)]{GetzGore}, we can also define a map of local systems 
\[ \tilde{\iota}^{\underline{k}}_{\mathcal{O}_\chi} : \mathcal{O}_\chi \to \iota^*\mathscr{V}^{\underline{k},0}_{\mathcal{O}_\chi}(\chi),\]
where $\mathcal{O}_\chi$ is the extension of $\mathcal{O}$ containing the values of $\chi$.

\section{Preliminaries on Hilbert modular forms and their \texorpdfstring{$L$}{L}-functions}

\subsection{Hilbert modular forms and their differential forms}

In this section we quickly recall some notation for Hilbert modular forms, following mainly \cite[\S 5]{GetzGore}. 

\subsubsection{Hilbert cusp forms and their differential forms}

Let $F$ be a totally real field of degree $d$ over $\mathbb{Q}$ and let $E=F(\sqrt{\Delta})$, for $\Delta$ a totally positive element of $F$.
Consider a weight $\kappa= (\underline{k},m) \in (\Z_{\geq 0})^{\Sigma_E} \times \Z$.
For an ideal $\mathfrak{c}$ of $\mathcal{O}_E$, we consider the compact-open subgroup \[K_0(\mathfrak{c})  = \{ g \in \GL_2(\widehat{\mathcal{O}}_E)\,:\, g = {\matrix{a}{b}{c}{d}}, \text{ with } c \equiv 0 \text{ mod } \mathfrak{c}\}.\]
Let $\chi$ be a Hecke  character for $E$ of conductor dividing $\mathfrak{c}$,  such that its infinity type is given by $\chi_{\infty}(x) = \prod_{v \in \Sigma_E} x_v^{-m}$. 
The space of cusp forms of weight $\kappa$ and Nebentypus $\chi$ is defined in \cite[(5.4.3)]{GetzGore} and as in {\it loc. cit.} we denote it by\footnote{As in the previous footnote, we observe that our weight $\kappa$ corresponds to the weight $(k_\sigma,\left( \frac{m - k_\sigma}{2}  \right) )_\sigma$ of \cite{GetzGore}.}  $S_{\kappa}(K_0(\mathfrak{c}),\chi)$. We shall denote by $f$ an element of $S_{\kappa}(K_0(\mathfrak{c}),\chi)$. It corresponds to a collection of classical Hilbert modular forms defined as functions on the product of upper-half planes. Note that, following {\it loc. cit.} our normalization on the weight is shifted by $2$ compared to the most common normalisation for Hilbert modular forms as functions on the upper half-plane. For example, we associate elliptic modular forms of weight $0$ to elliptic curves over $\mathbb{Q}$.

There is also a cohomological normalization, introduced in \S 5.5 of {\it loc. cit.},  given by \[
f\mapsto f^{-\iota}: x \mapsto f(x^{-\iota}) 
\]
for $x^{\iota}=\mathrm{det}(x)x^{-1}$ and $x^{-\iota}=\mathrm{det}(x)^{-1}x$. This changes $m$ to $-m$ (and $\chi$ to $\chi^{-1}$). This different normalization is more natural when we associate differential forms and cohomology classes to elements of $S_{\kappa}(K_0(\mathfrak{c}),\chi)$. 
Precisely, denote by $\chi_0$ the character of $K_0(\mathfrak{c})$, which at each finite place $v$ is defined by \[ \chi_v\left ( \begin{smallmatrix} a & b \\ c & d \end{smallmatrix} \right)= \begin{cases} \chi_v(d) & \text{ if } \frakp_v \mid \mathfrak{c} \\ 1 & \text{ otherwise,} \end{cases} \]
with $\frakp_v$ denoting the prime ideal associated with $v$. Then, for any subset $J \subset \Sigma_E$, we have a map   (Definition 6.8 of {\it loc. cit.})
\[S_{\kappa}(K_0(\mathfrak{c}),\chi) \to \Omega^{2d}(Y_G(K_0(\mathfrak{c})),\mathscr{V}^{\underline{k},{-m}}(
\chi^{-1}_0)_\C),\, f \mapsto  \omega_J(f^{-\iota}), \]
which defines a class  $[\omega_J(f^{-\iota})]$ in $H^{2d}( Y_G(K_0(\mathfrak{c})),\mathscr{V}^{\underline{k},{-m}}(
\chi^{-1}_0)_\C)$. The map is Hecke equivariant  in the sense of Proposition 7.13 of \textit{loc.cit.} (see Lemma \ref{TetaVSUp} below). The differential form is anti-holomorphic at places in $J$ and holomorphic at places outside $J$ and it is square integrable (Proposition 6.8 of {\it loc. cit.}), so that  $[\omega_J(f^{-\iota})]$ defines a class in $L^2$-cohomology of $Y_G(K_0(\mathfrak{c}))$. Note that as $f$ is a cusp form, the differential form $\omega_J(f^{-\iota})$ is rapidly decreasing, thus
\begin{align}\label{eq:RDdiffForm}
    [\omega_J(f^{-\iota})] \in H^{2d}_{{\rm rd}}(Y_G(K_0(\mathfrak{c})), \mathscr{V}^{\underline{k},{-m}}(
\chi^{-1}_0)_\C).
\end{align}

Recall that such a class can be represented by a compactly supported differential form:

\begin{proposition}[{\cite[Corollary 5.5]{borel2}}]\label{BorelRDtoC} Let $f \in S_{\kappa}(K_0(\mathfrak{c}),\chi)$. There exists a $\mathscr{V}^{\underline{k},{-m}}(
\chi^{-1}_0)_\C$-valued rapidly decreasing form $\Psi$ on $Y_G(K_0(\mathfrak{c}))$ such that $\omega_c : = \omega_J(f^{-\iota}) + d \Psi$ is compactly supported. In other words, 
\[[\omega_c]=[\omega_J(f^{-\iota})] \in H^{2d}_{{\rm rd}}(Y_G(K_0(\mathfrak{c})), \mathscr{V}^{\underline{k},{-m}}(
\chi^{-1}_0)_\C). \]
\end{proposition}

Similar normalization and definitions can be done for Hilbert modular forms for $F$.

\subsubsection{The Fourier expansion}\label{sec:Fourierexpansion}

We recall briefly the Fourier expansion of a form $f$ over $E$, following \cite[\S 5.9]{GetzGore}.
We define an additive character $e_E(x)$, for $x \in \mathbb{A}_E/E$. For a finite  place $v$ of $E$, lying above $q \in \mathbb{Z}$, we define
\[
\begin{array}{cccc}
  e_v:  &  E_v & \longrightarrow & \mathbb{C}^{\times} \\
     & x &\mapsto &  \exp(-2 \pi i \mr{tr}_{E_v}^{\mathbb{Q}_q}(x)) 
\end{array}
\]
where we identify $\mathbb{Q}/\mathbb{Z}$ with $\mathbb{Q}_q/\mathbb{Z}_q$.
If $v$ is an archimedean place, we set $e_v(x):=\exp(2 \pi i x) $.
We can finally define
\[
e_E(x):=\prod_{v}e_v(x_v), x = (x_v).
\]
If $f$ is of weight  $\kappa= (\underline{k},m) \in (\Z_{\geq 0})^{\Sigma_E} \times \Z$, for $\sigma \in \Sigma_E$  define
\[
\begin{array}{cccc}
  W_{\sigma}:  &  \mathbb{R}^{\times} & \rightarrow & \mathbb{C} \\
     & y &\mapsto &  |y|^{\frac{k_\sigma-m}{2}}e^{-2\pi|y|},
\end{array}
\]
and finally, for $x \in \mathbb{A}_E$ and $y \in \mathbb{A}_E^{\times}$
\[
q_{\kappa}(x,y):=e_E(x)\prod_{\sigma \in \Sigma_E}W_{\sigma}(y_{\sigma}).
\]
(Recall that our weight normalization is different from \cite{GetzGore}.)
We let $\mr{diff}_{E/\mathbb{Q}}$ be the different of $E$ over $\mathbb{Q}$. 
A  modular form $f \in S_{\kappa}(K_0(\mathfrak{c}),\chi)$ has a Fourier expansion of the form
\begin{equation}\label{eq:FourierExpansion}
    f\left( \begin{smallmatrix}
   y  & x \\
    0 & 1 
\end{smallmatrix} \right)= |y|_{\mathbb{A}_E}\sum_{\xi \in E^{\times}, \xi \gg 0} a (\xi y \mr{diff}_{E/\mathbb{Q}},f)q_{\kappa}(\xi x, \xi y). 
\end{equation}
When $\mathfrak{n}=( \xi y_f \mathcal{O}_E) \mr{diff}_{E/\mathbb{Q}}$ is an integral ideal, we write 
\[
a(\mathfrak{n},f)=a (\xi y \mr{diff}_{E/\mathbb{Q}},f).
\]
Similar results hold for forms over $F$. 

\subsubsection{The Hecke algebra}
We now define the Hecke algebra for $E$, following \cite[\S 1]{HidaFamilies89}. Denote \[K_{11}(\mathfrak{c}) : = \{ g \in K_0(\mathfrak{c})\,:\, g \equiv {\matrix{1}{\star}{0}{1}} \text{ mod } \mathfrak{c}\}.\]
\begin{definition}
 Fix a level $K \supset K_{11}(\mathfrak{c})$. For a prime ideal $\mathfrak{q}$ coprime with $\mathfrak{c}$ and a positive integer $l$ we define the Hecke operators $T_{\mathfrak{q}^l}, S_{\mathfrak{q}^l}$ as the double coset operators
 \[
 T_{\mathfrak{q}^l}:=\left[K \left(\begin{array}{cc}
    \varpi_{\mathfrak{q}^l}  & 0  \\
    0  & 1
 \end{array} \right)K \right], \quad  S_{\mathfrak{q}^l}:=\left[K \left(\begin{array}{cc}
     \varpi_{\mathfrak{q}^l} & 0 \\
    0  & \varpi_{\mathfrak{q}^l}
 \end{array} \right)K\right],
 \]
 for $\varpi_{\mathfrak{q}}$ a uniformizer at $\mathfrak{q}$. Note that as $\mathfrak{c}$ is coprime to $\mathfrak{q}$, our definition does not depend on the choice of $\varpi_{\mathfrak{q}}$. 
 
Let $\mathcal{R}$ be an extension of $E$ containing the Galois closure of $E$ and such that all the ideals of $E$ are principal in $\mathcal{R}$; let $\mathcal{V}$ its ring of integers. For any $\mathcal{V}$-algebra $R$ we define the spherical Hecke algebra $\mathcal{H}^{\rm sph}_{K,R}$ of level $K$ and coefficients in $R$ as the $R$-algebra generated by the operators $T_{\mathfrak{q}^l}$ and $S_{\mathfrak{q}^l}$ as $l$ and $\mathfrak{q} \nmid \mathfrak{c}$ vary. In $\mathcal{H}^{\rm sph}_{K,R}$, we have the relation
 \[
 T_{\mathfrak{q}}^2-T_{\mathfrak{q}^2}=\mathrm{N}_{E/\Q}(\mathfrak{q})S_\mathfrak{q}.
 \]
\end{definition}

We define similarly $U_{\mathfrak{q}^l}$ for ${\mathfrak{q}}$ dividing $\mathfrak{c}$; but note that it could depend on the choice of uniformizer $\varpi_{\mathfrak{q}}$, see the explicit formula in \cite[(2.2b)]{HidaFourier}. If $p$ is coprime with $\mathfrak{q}$, then there is no need of the normalization factor $\left\{ \phantom{e}\right\}^{-v}$ of \cite[p.~142]{HidaFamilies89}. We shall define later the proper normalization when $\mathfrak{q}$ divides $p$. We define by $\mathcal{H}_{K,R}$ the abstract Hecke algebra over $R$ generated by $U_{\mathfrak{q}^l}$, for all $\mathfrak{q}\mid \mathfrak{c}$ and $l$,  and $T_{\mathfrak{q}^l}$, $S_{\mathfrak{q}^l}$ for all $\mathfrak{q} \nmid \mathfrak{c}$ and any $l$.
 
 This Hecke algebra acts naturally on the space of Hilbert modular forms \cite[\S 2]{HidaFourier} or on cohomology of Hilbert modular varieties \cite[\S 7.6]{GetzGore} and these actions are compatible (in the sense of \cite[Proposition 7.13]{GetzGore}).

\begin{definition}\label{def_of_adjoint_character}
Let $\chi'$\index{$\chi'$} be another character of conductor $\mathfrak{c}$. We see it as a character of $K_0(\mathfrak{c})$ via 
\[ \chi'_v\left ( \begin{smallmatrix} a & b \\ c & d \end{smallmatrix} \right)= \begin{cases} \chi'_v(a^{-1}d) & \text{ if } \frakp_v \mid \mathfrak{c} \\ 1 & \text{ otherwise} \end{cases}. \]
We say that $f \in S_{\kappa}(K_0(\mathfrak{c}),\chi,\chi')$\index{$S_{\kappa}(K_0(\mathfrak{c}),\chi,\chi')$} if $f$ is a cusp form of level $K_{11}(\mathfrak{c})$ such that 
\[
f \vert \left ( \begin{smallmatrix} a & b \\ c & d \end{smallmatrix} \right) = \chi(d)\chi'(a^{-1}d) f
\]
for all $\left( \begin{smallmatrix} a & b \\ c & d \end{smallmatrix} \right) \in K_0(\mathfrak{c})$, where we use the convention of \cite{HidaFourier} for the slash operator $\vert$.
We shall call $\chi$ the Nebentypus of $f$ and $\chi'$ its adjoint character (because it factors via the adjoint group).
\end{definition}

\begin{definition}
 If $f \in S_{\kappa}(K_0(\mathfrak{c}),\chi,\chi')$ is an eigenform for all but finitely many Hecke operators, we denote by $\pi(f)$ the corresponding irreducible cuspidal automorphic representation of $\mr{GL}_2(\mathbb{A}_{E})$ of central character $\chi$, as in \cite[(5.4.3)]{GetzGore}. Note that $\pi(f)$ is not unitary in general, while its twist $\pi(f) \otimes | - |_{\mathbb{A}_{E}}^{m/2}$ is. 
\end{definition}

Similar definitions apply to Hilbert modular forms for $F$.

\subsubsection{Atkin--Lehner operators}
We recall here the Atkin--Lehner involution, following \cite[\S 7.4]{GetzGore}.

Choose a finite id\`ele $\tilde{c}$ such that it generates the ideal $\mathfrak{c}$. 
The Atkin--Lehner operator $W^
*_{\mathfrak{c}}$ on $S_{\kappa}(K_1(\mathfrak{c}))$ is induced by right multiplication by the matrix
\[
W_{\mathfrak{c}}=\matrix{0}{-1}{\tilde{c}}{0}.
\]
More precisely, it maps 
 $ f \in S_{(\underline{k},m)}(K_0(\mathfrak{c}),\chi)$ to  $ W^*_{\mathfrak{c}} (f) \in S_{(\underline{k},m)}(K_0(\mathfrak{c}),\psi)$,
where

\[
W^*_{\mathfrak{c}} (f) (\alpha)= \frac{|\chi(\det(\alpha))|}{\chi(\det(\alpha))}f(\alpha W_{\mathfrak{c}}),
\]

\noindent and $\psi=\chi^{-1}|\phantom{e}|^{-2m}$ (so at infinity it coincides with $\chi$, but the finite part is the inverse). Similarly, right multiplication by $W_{\mathfrak{c}}$ on $Y_G(K_0(\mathfrak{c}))$ induces an isomorphism of local systems $\mathscr{V}^{\underline{k},{-m}}(
\chi^{-1}_0) \simeq \mathscr{V}^{\underline{k},{-m}}(
\chi_0)$ which can be twisted to get a map on differential forms (\textit{cf}. \cite[Proposition 7.6]{GetzGore})
\[ W^*_{\mathfrak{c}} : \Omega^{2d}(Y_G(K_0(\mathfrak{c})),\mathscr{V}^{\underline{k},{-m}}(
\chi^{-1}_0)_\C) \to \Omega^{2d}(Y_G(K_0(\mathfrak{c})),\mathscr{V}^{\underline{k},{-m}}(
\chi_0)_\C), \] 
so that we have a commutative diagram 
\begin{eqnarray}\label{commdiagAL}
    \xymatrix{ S_{\kappa}(K_0(\mathfrak{c}),\chi) \ar[rr]^-{f \mapsto  \omega_J(f^{-\iota})} \ar[d]_{W_\mathfrak{c}^*} & & \Omega^{2d}(Y_G(K_0(\mathfrak{c})),\mathscr{V}^{\underline{k},{-m}}(
\chi^{-1}_0)_\C) \ar[d]^{W_\mathfrak{c}^*} \\ 
S_{\kappa}(K_0(\mathfrak{c}),\psi) \ar[rr]^-{f \mapsto  \omega_J(f^{-\iota})} & & \Omega^{2d}(Y_G(K_0(\mathfrak{c})),\mathscr{V}^{\underline{k},-m}(
\chi_0)_\C).
} 
\end{eqnarray}
Note that $W^*_{\mathfrak{c}}$ induces an isomorphism of the Hecke algebra of level $K_0(\mathfrak{c})$ by mapping every Hecke operator $T$ to its adjoint $T^*$ (see the proof of \cite[Lemma 9.3]{HidaFourier}), namely 
\[
T^*= W_{\mathfrak{c},*} \circ T \circ {W^*_{\mathfrak{c}}}.
\]
\begin{remark}
   When $T=T_{\mathfrak{q}}$ for $\mathfrak{q}$ coprime with $\mathfrak{c}$, then $T_{\mathfrak{q}}=T_{\mathfrak{q}}^*$. 
\end{remark}

\subsection{\texorpdfstring{$L$}{L}-functions and base change}

\subsubsection{The adjoint \texorpdfstring{$L$}{L}-function}\label{sec:AdLfunc}

Let $\mathfrak{c}_F$ be an ideal of $F$ and $f$ be an Hilbert cuspidal eigenform of level $K_{11}(\mathfrak{c}_F)$, weight $(\underline{k},m)$ and Nebentypus $\chi$. We assume that the form is minimal, {\it i.e.} that it has minimal conductor among all quadratic twists. For any Hecke character $\epsilon: F^\times \backslash \mathbb{A}^\times_F \to \C^\times$\index{$\epsilon$} of conductor $\mathfrak{f}(\epsilon)$, we define the (imprimitive) adjoint $L$-function of $f$ twisted by $\epsilon$ as 
\[
L(\mathrm{Ad}(f) \otimes \epsilon, s):=\prod_{ \mathfrak{q} \not\mid \mathfrak{f}(\epsilon)} {L_{\mathfrak{q}}(\mathrm{Ad}(f) \otimes \epsilon, s)}.
\]
If $\mathfrak{q}$ does not divide $\mathfrak{c}_F$ then ${L_{\mathfrak{q}}(\mathrm{Ad}(f) \otimes \epsilon, s)}^{-1}$ is defined as
\[
\left(1-\epsilon(\mathfrak{q})\frac{\alpha_{1,\mathfrak{q}}}{\alpha_{2,\mathfrak{q}}}\Norm_{F/\Q}(\mathfrak{q})^{-s}\right)
\left(1-\epsilon(\mathfrak{q})\Norm_{F/\Q}(\mathfrak{q})^{-s}\right) 
\left(1-\epsilon(\mathfrak{q})\frac{\alpha_{2,\mathfrak{q}}}{\alpha_{1,\mathfrak{q}}}\Norm_{F/\Q}(\mathfrak{q})^{-s}\right)
\]
for ${\alpha_{1,\mathfrak{q}}},{\alpha_{2,\mathfrak{q}}}$ the two Satake parameters of $\pi(f)_\mathfrak{q}$ defined via\footnote{Note that we normalize the Satake parameters as in \cite[(5.9.6)]{GetzGore}, where a shift $s\mapsto s+\frac{1}{2}$ from the usual normalisation is done when defining the standard $L$-function for $f$.}
  \[
1-a(\mathfrak{q},f)\Norm_{F/\Q}(\mathfrak{q})^{-s-1/2}+\chi(\mathfrak{q})\Norm_{F/\Q}(\mathfrak{q})^{-2s}=(1-{\alpha_{1,\mathfrak{q}}}\Norm_{F/\Q}(\mathfrak{q})^{-s})
(1-{\alpha_{2,\mathfrak{q}}}\Norm_{F/\Q}(\mathfrak{q})^{-s}).\]
Note that the adjoint $L$-function is independent of the variable $m$.

If $\mathfrak{q}$ divides $\mathfrak{c}_F$ but not $\mathfrak{f}(\epsilon)$, we consider the following cases. 
If the local automorphic representation $\pi(f)_\mathfrak{q}$ associated with $f$ is a ramified principal series,
then ${L_{\mathfrak{q}}(\mathrm{Ad}(f) \otimes \epsilon, s)}^{-1}$ is defined as
\[ 
(1-\epsilon(\mathfrak{q})\Norm_{F/\Q}(\mathfrak{q})^{-s}).
\]
If $\pi(f)_\mathfrak{q}$ is special, ({\it i.e.} Steinberg), then 
 ${L_{\mathfrak{q}}(\mathrm{Ad}(f) \otimes \epsilon, s)}^{-1}$ is defined as
\[ 
(1-\epsilon(\mathfrak{q})\Norm_{F/\Q}(\mathfrak{q})^{-1-s}).
\]
If $\pi(f)_\mathfrak{q}$ is supercuspidal, we set the local $L$-factor to be equal to $1$, even if the completed $L$-function could have a Euler factor different from $1$ at $\mathfrak{q}$.

For $\sigma \in \Sigma_F$ we define  
\[
{L_{\sigma}(\mathrm{Ad}(f) \otimes \epsilon, s)}:= {(2\pi)}^{-(s+k_\sigma+1)}\Gamma(s+k_\sigma+1)\pi^{-(s+1)/2}\Gamma((s+1)/2).
\]
For any other ideal $\mathfrak{c}'_F$ in $\mathcal{O}_F$ we define

\[
L^{\ast,\mathfrak{c}'_F }(\mathrm{Ad}(f) \otimes \epsilon, s):=\prod_{\sigma \in \Sigma_F}L_{\sigma}(\mathrm{Ad}(f) \otimes \epsilon, s) \prod_{\mathfrak{q} \not\mid \mathfrak{c}'_F \mathfrak{f}(\epsilon)} {L_{\mathfrak{q}}(\mathrm{Ad}(f) \otimes \epsilon, s)}.
\]

Finally, note that one can also define the (partial) symmetric square $L$-function for $f$ twisted by $\epsilon$: 
\[
L^{\mathfrak{c}\mathfrak{f}(\epsilon)} (\mathrm{Sym}^2(f) \otimes \epsilon, s):=\prod_{ \mathfrak{q} \not\mid \mathfrak{c}\mathfrak{f}(\epsilon)} {L_{\mathfrak{q}}(\mathrm{Sym}^2(f) \otimes \epsilon, s)},
\]
where ${L_{\mathfrak{q}}(\mathrm{Sym}^2(f) \otimes \epsilon, s)}^{-1}$ is given by 
\[
\left(1-\epsilon(\mathfrak{q})\alpha_{1,\mathfrak{q}}^2\Norm_{F/\Q}(\mathfrak{q})^{-s}\right)
\left(1-\epsilon(\mathfrak{q})\chi(\mathfrak{q})\Norm_{F/\Q}(\mathfrak{q})^{-s}\right) 
\left(1-\epsilon(\mathfrak{q})\alpha_{2,\mathfrak{q}}^2\Norm_{F/\Q}(\mathfrak{q})^{-s}\right).
\]
Note that at any prime ideal $\mathfrak{q}$, unramified for $f$ and $\chi$, we have \[ {L_{\mathfrak{q}}(\mathrm{Sym}^2(f) \otimes \chi^{-1}, s)} ={L_{\mathfrak{q}}(\mathrm{Ad}(f), s)}.\] 
 
\subsubsection{The Asai \texorpdfstring{$L$}{L}-function} 
 
Here, we follow \cite[\S 5.12.4]{GetzGore}. Let $\mathfrak{c}$ be an ideal of $E$, let $\chi_F$ be a Hecke character of $\A_F^\times$, and $\chi: = \chi_F \circ {\rm N}_{E/F}$ be the corresponding Hecke character of $\A_E^\times$. 
We also let $\eta_{E/F}$ be the only non-trivial (quadratic) character of the Galois group of $E$ over $F$, and, by abuse of notation, we use $\eta_{E/F}$ to denote the corresponding Hecke character of $\A^\times_F$, trivial at the archimedean places. Recall that there is a distinguished $F$-type for $E$ that we denoted by $\Phi$ in \S \ref{secNotation}. Then for a given weight $\kappa = (\underline{k},m)$ for $F$, we have the corresponding weight for $E$: $\hat{\kappa} = (\hat{\underline{k}},m)$, with $\hat{k}_\sigma =\hat{k}_{\sigma'}$ for every $\sigma \in \Phi$. 

To any Hilbert cuspidal eigenform $g$ of level $K_{11}(\mathfrak{c})$, weight $\hat{\kappa}$ and Nebentypus $\chi$ and to any quasi-character $\theta: E^\times \backslash \A_E^\times \to \C^\times$ with conductor $\mathfrak{f}(\theta)$\index{$\theta$}, we can associate the Asai $L$-function of $g$ twisted by $\theta$ as follows\footnote{We again follow the normalization of \cite{GetzGore} of the Satake parameters, which produces the shift $s\mapsto s+1$.} 
\begin{equation}\label{DefAsaiLF}
L(\mathrm{As}(g \otimes \theta) , s) := L^{\mathfrak{c}\mathfrak{f}(\theta) \cap \mathcal{O}_F }(\theta^2 \chi)|_F,2s) \sum_{\mathfrak{n} +\mathfrak{f}(\theta) = \mathcal{O}_F}  \theta(\mathfrak{n}\mathcal{O}_E) a(\mathfrak{n}\mathcal{O}_E,g) \Norm_{F/\Q}(\mathfrak{n})^{-1-s}.  
\end{equation}

The corresponding partial Asai $L$-function of $g$ admits an Euler product \[
L^{\mathfrak{c}\mathfrak{f}(\theta) \cap \mathcal{O}_F }(\mathrm{As}(g \otimes \theta) , s) :=\prod_{ \mathfrak{p} \not\mid \mathfrak{c} \cap \mathcal{O}_F } {L_{\mathfrak{p}}(\mathrm{As}(g \otimes \theta) , s)},
\]
where, if $\frakp \nmid \mathfrak{d}_{E/F}(\mathfrak{c}\mathfrak{f}(\theta) \cap \mathcal{O}_F) $, define 
\[ L_{\mathfrak{p}}(\mathrm{As}(g \otimes \theta) , s)^{-1} = \begin{cases} \prod_{i=1}^2\prod_{j=1}^2 (1 - \theta(\mathfrak{p}) \alpha_{i,\mathfrak{P}}\alpha_{j,\overline{\mathfrak{P}}} \Norm_{F/\Q}(\frakp)^{-s}) & \text{if } \frakp = \mathfrak{P}\overline{\mathfrak{P}} \text{ is split,} \\ 
(1 - \theta(\mathfrak{p})^2 \chi(\frakp) \Norm_{F/\Q}(\frakp)^{-2s})\prod_{i=1}^2 (1 - \theta(\mathfrak{p}) \alpha_{i,\frakp}\Norm_{F/\Q}(\frakp)^{-s}) & \text{if } \frakp \text{ is inert,} \end{cases}  \] 
with $\alpha_{1, \mathfrak{q}}, \alpha_{2, \mathfrak{q}}$ denoting the Satake parameters of $\pi(g)_\mathfrak{q}$ (\textit{cf}. \cite[\S 5.12.4]{GetzGore}).

\subsubsection{Base change forms and their  \texorpdfstring{$L$}{L}-function} The main reference for this section is \cite[Appendix E]{GetzGore}. If we let $G_1 = \GL_2 /F$ and $G_2= {\rm Res}_{E/F} \GL_{2}$, we have 
\begin{align*}
    {}^LG_1 &= \GL_2(\C) \times {\rm Gal}(E/F) \\
    {}^LG_2 &= \GL_2(\C)^{{\rm Gal}(E/F)} \rtimes {\rm Gal}(E/F),
\end{align*}
 where in the second ${\rm Gal}(E/F)$ acts on $\GL_2(\C)^{{\rm Gal}(E/F)}$ by permuting the factors. The base change lifting from $G_1$ to $G_2$ is induced by the map of $L$-groups 
 \[ b:{}^LG_1  \to {}^LG_2, \]
given by the diagonal embedding on the $\GL_2(\C)$ factor and by the identity on the ${\rm Gal}(E/F)$ factor. For any place $v$ of $F$, the map $ b:{}^LG_{1, F_v}  \to {}^LG_{2, F_v}$ of $L$-groups over $F_v$ is defined similarly. Suppose that $v$ is a finite place of $F$ which is unramified for $E/F$. Then $b$ induces a morphism of local spherical Hecke algebras \[b: \mathcal{H}(\GL_2(E_v) // \GL_2( \mathcal{O}_{E_v})) \longrightarrow \mathcal{H}(\GL_2(F_v) // \GL_2( \mathcal{O}_{F_v})), \] 
where we have denoted $E_v = E \otimes_F F_v$. If $\pi_v$ is an irreducible unramified admissible representation of $\GL_2(F_v)$, where $\mathcal{H}(\GL_2(F_v) // \GL_2( \mathcal{O}_{F_v}))$ acts via the character $\lambda_{\pi_v}$, then an irreducible admissible representation $\Pi_v$ of $\GL_2(E_v) $ is the base change of $\pi_v$ if it is unramified and \[ \mathcal{H}(\GL_2(E_v) // \GL_2( \mathcal{O}_{E_v})) \] acts through the character \[ \lambda_{\Pi_v} = \lambda_{\pi_v} \circ b. \]
Concretely, if $\pi_v$ is the principal series attached to unramified characters $\chi_1,\chi_2$ of $F_v^\times$, $\Pi_v$ is the base change of $\pi_v$ if it is the principal series of characters $\chi_1 \circ {\rm N}_{E_v/F_v},\chi_2 \circ {\rm N}_{E_v/F_v}$. Globally, we say that an automorphic representation $\Pi$ of $\GL_2(\A_E)$ is a weak base change of an automorphic representation $\pi$ of $\GL_2(\A_F)$ if $\Pi_v$ is the base change of $\pi_v$ almost everywhere. Note that, by strong multiplicity for $\GL_2$, if there exists a weak base change $\Pi$ of $\pi$, then it is unique. 
\begin{theorem}[{\cite[Theorem E.11]{GetzGore}}] Every automorphic representation $\pi$ of $\GL_2(\A_F)$ has a unique base change lift to an automorphic representation $\widehat{\pi}$ of $\GL_2(\A_E)$. Moreover, any cuspidal automorphic representation $\Pi$ of $\GL_2(\A_E)$ that is isomorphic to its conjugates under the action of ${\rm Gal}(E/F)$ is the base change lift of a cuspidal automorphic representation of $\GL_2(\A_F)$. The assignment $\pi \mapsto \widehat{\pi}$ satisfies the following : If $\pi$ and $\widehat{\pi}$ are both cuspidal, then \begin{enumerate}
    \item If $\widehat{\pi}$ is also the base change lift of another automorphic representation $\pi'$, then $\pi' \simeq \pi \otimes \eta_{E/F}$;
    \item For every place $v$ of $F$, then $L(\widehat{\pi}_v,s) = L(\pi_v,s)L(\pi_v \otimes \eta_{E_v/F_v},s)$.
\end{enumerate} 
\end{theorem}
 \noindent Fix a weight $\kappa = (\underline{k},m)$ for $F$ and consider the corresponding weight $\hat{\kappa} = (\hat{\underline{k}},m)$ for $E$. The theorem above can be used to prove the following Proposition, which we will use repeatedly in the course of the manuscript. 
 
\begin{proposition}\label{AsaiForBC} 
Suppose that a newform $g \in S_{\hat{\kappa}}(K_0(\mathfrak{c}), \chi)$ is an eigenform for all but finitely many Hecke operators for which the eigenvalues of $T_\mathfrak{q}$ and $T_{\mathfrak{q}^{\sigma}}$, with $\sigma$ generator of ${\rm Gal}(E/F)$, agree for almost all prime ideals $\mathfrak{q} \subset \mathcal{O}_E$. Then there exists an ideal $\mathfrak{c}_1 \subseteq \mathcal{O}_F$ and a newform $f \in  S_{\kappa}(K_0(\mathfrak{c}_1), \chi_F)$ such that $\widehat{\pi(f)} = \pi(g)$. Conversely, if $f \in  S_{\kappa}(K_0(\mathfrak{c}_F), \chi_F)$ is a newform, there exists a unique newform $f_E \in S_{\hat{\kappa}}(K_0(\mathfrak{c}), \chi)$ such that $\pi(f_E) = \widehat{\pi(f)}$. Moreover, if $\theta:=\theta^u|\phantom{ee}|_{\mathbb{A}_E}^{m/2}$ is a Hecke character of $\mathbb{A}^\times_E$, with $\theta^u$ unitary, such that its restriction to $\mathbb{A}^\times_F$ is $\chi_F \eta_{E/F}$, then
\[ L^{\mathfrak{b}'}(\mathrm{As}(f_E \otimes \theta^{-1}) , s) = \zeta_F^{\mathfrak{d}_{E/F}(\mathfrak{c} \mathfrak{f}(\theta) \cap \mathcal{O}_F )}(s) L^{\mathfrak{d}_{E/F} ( \mathfrak{f}(\theta) \cap \mathcal{O}_F )}(\mathrm{Ad}(f) \otimes \eta_{E/F}, s),  \]
with $\mathfrak{b}' := \prod_{\substack{ \mathfrak{p} | \mathfrak{d}_{E/F} \\ \mathfrak{p} \nmid \mathfrak{f}(\theta) \cap \mathcal{O}_F}} \mathfrak{p}$. Hence, 
\[{\rm Res}_{s=1}L^{\mathfrak{b}'}(\mathrm{As}(f_E \otimes \theta^{-1}) , s) = L^{\mathfrak{d}_{E/F} ( \mathfrak{f}(\theta) \cap \mathcal{O}_F )}(\mathrm{Ad}(f) \otimes \eta_{E/F}, 1) {\rm Res}_{s=1}\zeta_F^{\mathfrak{d}_{E/F}(\mathfrak{c} \mathfrak{f}(\theta) \cap \mathcal{O}_F )}(s). \]
\end{proposition}

 \begin{proof}
This follows from Corollary E.12 and Proposition 5.17 of \cite{GetzGore}. 
 \end{proof}

\subsection{A cohomological formula for adjoint \texorpdfstring{$L$}{L}-values}\label{sec:cohomologygroups}

Let $\mathfrak{c}_F$ be an ideal of $F$ and $f$ be a Hilbert cuspidal eigenform of level $K_{11}(\mathfrak{c}_F)$, weight $\kappa = (\underline{k},m)$ and Nebentypus $\chi_F$.
 Let $\mathfrak{c}$ be an ideal of $\mathcal{O}_E$, whose intersection with $\mathcal{O}_F$ is $\mathfrak{c}_F$. We consider a Hecke character $\theta:=\theta^u|\phantom{ee}|_{\mathbb{A}_E}^{m/2}$ of $\mathbb{A}^\times_E$ such that its restriction to $\mathbb{A}^\times_F$ is $\chi_F \eta_{E/F} $. Here $\theta^{u}$ denotes the unitary part of $\theta$ and $m$ is as in the weight $(\underline{k},m)$. Let $\mathfrak{b}$ be the conductor of  $\theta^{u}$  and $\mathfrak{c}'=\mathfrak{c}\mathfrak{b}^2$. We write $\mathfrak{b}_F, \mathfrak{c}'_F$ for the intersection of $\mathfrak{b}$ and $\mathfrak{c}'$ with $\mathcal{O}_F$.  We let $\chi : = \chi_F \circ {\rm N}_{E/F}$ and we let $f_E \in S_{\hat{\kappa}}(K_0(\mathfrak{c}),\chi)$ be the unique newform on $E$ which generates the automorphic representation of $G$ given by the base change of $\pi(f)$. Recall that $\hat{\kappa} = (\hat{\underline{k}},m)$, with $\hat{k}_\sigma =\hat{k}_{\sigma'}$ for every $\sigma \in \Phi$. We also have denoted by $\chi_0$ the character of $K_0(\mathfrak{c})$, which at each finite place $v$ is defined by \[ \chi_v\left ( \begin{smallmatrix} a & b \\ c & d \end{smallmatrix} \right)= \begin{cases} \chi_v(d) & \text{ if } \frakp_v \mid \mathfrak{c} \\ 1 & \text{ otherwise,} \end{cases} \]
with $\frakp_v$ denoting the prime ideal associated with $v$. Poincar\'e duality gives a perfect pairing 
\[ \langle \, , \, \rangle : H^{2d}_c(Y_G(K_0(\mathfrak{c})), \mathscr{V}_\C^{\hat{\underline{k}},-m}(\chi^{-1}_0)) \times H^{2d}(Y_G(K_0(\mathfrak{c})), \mathscr{V}_\C^{\hat{\underline{k}},m}(\chi_0)) \to \C, \]
where $H^{2d}_c(\bullet) $ stands for compactly supported cohomology. Similarly, we have an intersection pairing for the middle intersection cohomology $\langle \, , \, \rangle_{IH}$\index{$\langle \, , \, \rangle_{IH}$} (with respect to the Baily--Borel compactification of $Y_G(K_0(\mathfrak{c}))$)
\[ \langle \, , \, \rangle_{IH} : {IH}^{2d}(Y_G(K_0(\mathfrak{c})), \mathscr{V}_\C^{\hat{\underline{k}},-m}(\chi^{-1}_0)) \times {IH}^{2d}(Y_G(K_0(\mathfrak{c})), \mathscr{V}_\C^{\hat{\underline{k}},m}(\chi_0)) \to \C. \]
If we let $\alpha: H^{2d}_c(Y_G(K_0(\mathfrak{c})), \mathscr{V}_\C^{\hat{\underline{k}},-m}(\chi^{-1}_0)) \to {IH}^{2d}(Y_G(K_0(\mathfrak{c})), \mathscr{V}_\C^{\hat{\underline{k}},-m}(\chi^{-1}_0))$ be the composition of the natural morphism from $H^{2d}_c(\bullet)$ to $L^2$-cohomology $H^{2d}_{(2)}(\bullet)$ with the isomorphism $\mathcal{Z}$\index{$\mathcal{Z}$} of Zucker's conjecture  (\cite[Theorem 4.2]{GetzGore})
\[ \mathcal{Z}:{H}^{2d}_{(2)}(Y_G(K_0(\mathfrak{c})), \mathscr{V}_\C^{\hat{\underline{k}},-m}(\chi^{-1}_0)) \simeq  {IH}^{2d}(Y_G(K_0(\mathfrak{c})), \mathscr{V}_\C^{\hat{\underline{k}},-m}(\chi^{-1}_0)),\]
and $\beta$ the natural map ${IH}^{2d}(Y_G(K_0(\mathfrak{c})), \mathscr{V}_\C^{\hat{\underline{k}},m}(\chi_0)) \to {H}^{2d}(Y_G(K_0(\mathfrak{c})), \mathscr{V}_\C^{\hat{\underline{k}},m}(\chi_0))$, then we 
have (\textit{cf}. (3.4.7) of \emph{loc. cit.})

\[ \langle \omega_1 , \beta(\omega_2) \rangle =\langle \alpha(\omega_1) , \omega_2 \rangle_{IH}.  \]

For any type $J$ for $E/F$, the differential form  $\omega_J(f^{-\iota}_E)$ defines an intersection cohomology class via $\mathcal{Z}$. Let $f_E \otimes \theta^{-1}$ be the twist of $f_E$ by $\theta^{-1}$, see \cite[Lemma 5.14]{GetzGore}. 
By \cite[Theorem 10.1]{GetzGore}, there exists an intersection cohomology class \[[Z(\mathfrak{c})_{\theta}] \in {IH}^{2d}(Y_G(K_0(\mathfrak{c})), \mathscr{V}_\C^{\hat{\underline{k}},m}(\chi_0))  \] (made up of a cycle $Z(\mathfrak{c})_{\theta}$, given by $Y_H(K_0(\mathfrak{c}_F))$ twisted by $\theta$ and the $H$-invariant section $\Delta^{\hat{\underline{k}}}$ of the restricted local system  $\mathscr{V}^{\hat{\underline{k}},{m}}_\C$) such that   

\begin{align}\label{GGformulawithIH}
\langle \mathcal{Z}[ \omega_J(f^{-\iota}_E)], [Z(\mathfrak{c})_\theta] \rangle_{IH} &= C'_2 \int_{Y_H(K_0(\mathfrak{c}_F'))} \tilde{\iota}^{\hat{\underline{k}},\ast}( \omega_J(f^{-\iota}_E \otimes \theta^{-1})),
\end{align}
with\footnote{Note that our constant is slightly different from the $C_2$ of \cite[\S 9.4, 9.5]{GetzGore}; indeed on page 153 of {\it loc. cit.} they use a group-theoretic ${\pi_{1,\ast}}$, which sums over all representative of $K_0(\mathfrak{c}'_F) / K_{1}(\mathfrak{c}'_F)$. We use instead a geometric degeneracy map $\pi_1$, see the correspondence on the next page, and note that $\mathcal{O}_F^{\ast, +}$ acts trivially on the upper half-plane, hence the Galois group of our $\pi_1$ is then $K_0(\mathfrak{c}'_F)/\mathcal{O}_F^{\ast, +} K_{1}(\mathfrak{c}'_F)$.} $C'_2= [K_0(\mathfrak{c}'_F):\mathcal{O}_F^{\ast, +} K_{1}(\mathfrak{c}'_F)] [K_0(\mathfrak{c}'):\mathcal{O}_E^{\ast, +}K_{11}(\mathfrak{c}')]G(\theta^{u}) |N_{E/\Q}(D_{E/\Q})|^{-m/2}$ and where $\tilde{\iota}^{\hat{\underline{k}},\ast}$ is induced by the map in cohomology corresponding to the pullback to $Y_H$ and at the level of vector bundles to the $\chi^{-1}_0\theta_0^{-2}$-twisted branching map $\tilde{\iota}^{\hat{\underline{k}}} : \C \hookrightarrow \iota^{\ast} \mathscr{V}_\C^{\hat{\kappa}}(\chi^{-1}_0\theta_0^{-2})$ as in \S \ref{sec:cohomology}. Crucially, this integral recovers the algebraic part of the adjoint $L$-function of $f$ twisted by $\eta_{E/F}$.  

\begin{remark}
 The pairing can also be calculated with any Hilbert modular form $g$ for $E$, or with any subset $J$ of $\Sigma_E$, but the resulting integral will vanish if $g$ is not the base change of a Hilbert modular form from $F$ or if $J$ is not an $F$-type for $E$. Indeed, when the period integral in the right hand side of \eqref{GGformulawithIH} is non-zero then the residue at $s=1$ of the Asai $L$-function of $g$ twisted by $\eta_{E/F}$ is non-zero. The existence of this pole implies that $g$ is the base change of a Hilbert modular form from $F$ (\textit{cf.} \cite[Theorem 10.1]{GetzGore}).
\end{remark}

Note that the above integral formula can be obtained without using middle intersection cohomology by means of the Hirzebruch--Zagier cycles in the cohomology of the open Hilbert modular variety, as we now explain. Since we ultimately need the following construction for weights $\hat{\kappa}$ of the form $(\hat{\underline{k}},0)$, we assume that $m =0$.

Let $\mathbf{1}_{K_0(\mathfrak{c}_F)}$ be the fundamental class \[ [Y_H(K_0(\mathfrak{c}_F)) ] \in H^0(Y_H(K_0(\mathfrak{c}_F)), \Z),\]
corresponding to the Hilbert modular variety of level $K_0(\mathfrak{c}_F)$ and let $\iota^{\hat{\underline{k}}}:\C \hookrightarrow \mathscr{V}^{\hat{\underline{k}},0}_{\C}$ be the branching map introduced in \S \ref{sec:cohomology}.

\begin{definition}
 Let \[ \mathscr{Z}^{\hat{\kappa}}(\mathfrak{c}) : = \iota_\ast \circ \iota^{\hat{\underline{k}}} (\mathbf{1}_{K_0(\mathfrak{c}_F)}) \in {H}^{2d}(Y_G(K_0(\mathfrak{c})), \mathscr{V}_\C^{\hat{\kappa}}). \]
\end{definition}

Similarly to the definition above, we can define a cycle in the cohomology with coefficients $\mathscr{V}_\C^{\hat{\kappa}}(\chi_0)$. To do so, let 
\[K_{1}(\mathfrak{c}) : = \{ g \in K_0(\mathfrak{c})\,:\, g \equiv {\matrix{\star}{\star}{0}{1}} \text{ mod } \mathfrak{c}\},\]
and denote by \[P_{10,\ast}^{\chi_0}:{H}^{2d}(Y_G(K_1(\mathfrak{c})), \mathscr{V}_\C^{\hat{\kappa}}(\chi_0)) \to {H}^{2d}(Y_G(K_0(\mathfrak{c})), \mathscr{V}_\C^{\hat{\kappa}}(\chi_0)) \]
the $\chi_0$-twisted trace associated with the natural degeneracy map $P_{10}^{\mathfrak{c}}: Y_G(K_1(\mathfrak{c})) \to Y_G(K_0(\mathfrak{c}))$ (\textit{cf}. \cite[(9.2.5)]{GetzGore}). Note that $\mathscr{V}_\C^{\hat{\kappa}}(\chi_0) = \mathscr{V}_\C^{\hat{\kappa}}$ as local systems over $Y_G(K_1(\mathfrak{c}))$.

\begin{definition}\label{def:twistedcycle1}
 Let \[ \mathscr{Z}^{\hat{\kappa},\chi}(\mathfrak{c}) : =P_{10,\ast}^{\chi_0} \circ \iota_\ast \circ \iota^{\hat{\underline{k}}} (\mathbf{1}_{K_1(\mathfrak{c}_F)}) \in {H}^{2d}(Y_G(K_0(\mathfrak{c})), \mathscr{V}_\C^{\hat{\kappa}}(\chi_0)). \]\index{$\mathscr{Z}^{\hat{\kappa},\chi}(\mathfrak{c})$}
 \end{definition}

This cycle can be further twisted by $\theta$ by using the cohomological correspondence of \cite[(9.4.1)]{GetzGore}. Namely, choose a set $S = \{ t \}_{t \in \mathfrak{b}^{-1} \widehat{\mathcal{O}}_E}$ of representatives for the quotient $\mathfrak{b}^{-1} \widehat{\mathcal{O}}_E /  \widehat{\mathcal{O}}_E$. For each $t \in S$, let $u(t) \in G(\Af)$ be the matrix defined by \[ u(t)_v = \begin{cases} {\matrix{1}{0}{0}{1}} & \text{ if } v \nmid \mathfrak{b} \\
{\matrix{1}{t_v}{0}{1}} & \text{ if } v \mid \mathfrak{b}.
\end{cases} \]
Consider the correspondence 

\[ \xymatrix{ & Y_G(K_{11}(\mathfrak{c}')) \ar[ddl]_{\pi_1} \ar[r]^{\cdot u(t)} & Y_G(K_{11}(\mathfrak{c}'))\ar[ddr]^{\pi_2} & \\ & & & \\ Y_G(K_0(\mathfrak{c}')) & & &  Y_G(K_0(\mathfrak{c})),
}
\]
where $\pi_1,\pi_2$ are the natural degeneracy maps and $\mathfrak{c}'=\mathfrak{c}\mathfrak{b}^2$, and define $\mathcal{P}_{\mathfrak{b}} : = \sum_{t \in S} \pi_{2,\ast} u(t)_\ast \pi_1^\ast$. Thanks to \cite[Lemma 9.3]{GetzGore}, $u(t)$ induces a map of local systems $\pi_1^* \mathscr{V}_\C^{\hat{\kappa}}(\chi_0\theta_0^2) \to \pi_2^* \mathscr{V}_\C^{\hat{\kappa}}(\chi_0)$ and so a cohomological correspondence \[\mathcal{P}_{\mathfrak{b}} : {H}^{2d}(Y_G(K_0(\mathfrak{c}')), \mathscr{V}_\C^{\hat{\kappa}}(\chi_0\theta_0^2)) \to {H}^{2d}(Y_G(K_0(\mathfrak{c})), \mathscr{V}_\C^{\hat{\kappa}}(\chi_0)). \] Thus we can give the following (\textit{cf}. \cite[Proposition 9.5]{GetzGore}).

\begin{definition}\label{def:twistedcycle2}
 Let \[ \mathscr{Z}^{\hat{\kappa}}(\mathfrak{c})_\theta : = \mathcal{P}_{\mathfrak{b}}(\mathscr{Z}^{\hat{\kappa},\chi\theta^2}(\mathfrak{c'})) \in {H}^{2d}(Y_G(K_0(\mathfrak{c})), \mathscr{V}_\C^{\hat{\kappa}}(\chi_0)). \]\index{$\mathscr{Z}^{\hat{\kappa}}(\mathfrak{c})_\theta$}
\end{definition}

Let now $g \in S_{\hat{\kappa}}(K_0(\mathfrak{c}),\chi)$ be an eigenform for $E$ and let $J$ be an $F$-type for $E$. As $\omega_J(g^{-\iota})$ is rapidly decreasing and rapidly decreasing cohomology is isomorphic to the compactly supported one, see \cite[Theorem 5.2]{borel2}, we can consider the value

\[\langle [ \omega_J(g^{-\iota}) ] , \mathscr{Z}^{\hat{\kappa}}(\mathfrak{c})_\theta \rangle  = \langle [ \omega_c ] , \mathscr{Z}^{\hat{\kappa}}(\mathfrak{c})_\theta \rangle. \]
where $\omega_c$ is the compactly supported differential form given by Proposition \ref{BorelRDtoC}. 

\begin{lemma}\label{Pairingequalintegral}
Let $g \in S_{\hat{\kappa}}(K_0(\mathfrak{c}),\chi)$ be an eigenform for $E$ of weight $\hat{\kappa} = (\hat{\underline{k}},0)$ and let $J$ be an $F$-type for $E$, then \[ \langle [ \omega_J(g^{-\iota}) ] , \mathscr{Z}^{\hat{\kappa}}(\mathfrak{c}) \rangle = \int_{Y_H(K_0(\mathfrak{c}_F))} \iota^{\hat{\underline{k}},\ast}( \omega_J(g^{-\iota})). \]
Similarly, if $\theta$ is the Hecke character of $\A_E^\times$ introduced above, 
\[ \langle [ \omega_J(g^{-\iota}) ] , \mathscr{Z}^{\hat{\kappa}}(\mathfrak{c})_\theta \rangle = C_2
\int_{Y_H(K_0(\mathfrak{c}_F'))} \tilde{\iota}^{\hat{\underline{k}},\ast}( \omega_J(g^{-\iota} \otimes \theta^{-1})), \]
where  $C_2:= [K_0(\mathfrak{c}'_F):\mathcal{O}_F^{\ast, +} K_{1}(\mathfrak{c}'_F)] [K_0(\mathfrak{c}'):\mathcal{O}_E^{\ast, +}K_{11}(\mathfrak{c}')]G(\theta^{u})$ and \[\tilde{\iota}^{\hat{\underline{k}},\ast} : H^{2d}(Y_G(K_0(\mathfrak{c}')), \mathscr{V}_\C^{\hat{\kappa}}(\chi^{-1}_0\theta_0^{-2})) \to H^{2d}(Y_H(K_0(\mathfrak{c}_F')), \C)\]  
is the map induced by pullback along $\iota$ at the level of varieties and by the $\chi^{-1}_0\theta_0^{-2}$-twisted branching map $\tilde{\iota}^{\hat{\underline{k}}} : \C \hookrightarrow \iota^{\ast} \mathscr{V}_\C^{\hat{\kappa}}(\chi^{-1}_0\theta_0^{-2})$ as in \S \ref{sec:cohomology}.
\end{lemma}

\begin{proof}
For the first part of the statement, we argue exactly as in \cite[Proposition 4.8]{CLRG2}. Using Proposition \ref{BorelRDtoC}, we get that \begin{align*}
    \langle [ \omega_J(g^{-\iota}) ] , \mathscr{Z}^{\hat{\kappa}}(\mathfrak{c}) \rangle  &= \langle [ \omega_c ] , \mathscr{Z}^{\hat{\kappa}}(\mathfrak{c}) \rangle \\ &= \langle [ \omega_c ] , \iota_\ast \circ \iota^{\hat{\underline{k}}} (\mathbf{1}_{K_0(\mathfrak{c}_F)}) \rangle \\ &= \langle \iota^{\hat{\underline{k}},\ast}[ \omega_c] ,  \mathbf{1}_{K_0(\mathfrak{c}_F)} \rangle \\ 
    &= \int_{Y_H(K_0(\mathfrak{c}_F))}\!\!\!\! \iota^{\hat{\underline{k}},\ast}(\omega_c) \\
    &= \int_{Y_H(K_0(\mathfrak{c}_F))}\!\!\!\! \iota^{\hat{\underline{k}},\ast}(\omega_J(g^{-\iota}) ) +  \int_{Y_H(K_0(\mathfrak{c}_F))}\!\!\!\! \iota^{\hat{\underline{k}},\ast}(d \Psi),
\end{align*}
with $\Psi$ a degree $2d-1$ rapidly decreasing form on $Y_G(K_0(\mathfrak{c}))$. To conclude the proof, note that, as explained in \cite[\S 5.6]{borel2}, \[\int_{Y_H(K_0(\mathfrak{c}_F))}\!\!\!\! \iota^{\hat{\underline{k}},\ast}(d \Psi) = 0.\]
Regarding the second formula, notice \begin{align*}
     \langle [ \omega_J(g^{-\iota}) ] , \mathscr{Z}^{\hat{\kappa}}(\mathfrak{c})_\theta \rangle &=  \langle [ \omega_J(g^{-\iota}) ] , \mathcal{P}_{\mathfrak{b}}(\mathscr{Z}^{\hat{\kappa},\chi\theta^2}(\mathfrak{c'})) \rangle \\ &=  \langle \mathcal{P}_{\mathfrak{b}}'[ \omega_J(g^{-\iota}) ] , \mathscr{Z}^{\hat{\kappa},\chi\theta^2}(\mathfrak{c'}) \rangle \\ 
     &= C_2 \langle [ \omega_J(g^{-\iota} \otimes \theta^{-1}) ] , \mathscr{Z}^{\hat{\kappa},\chi\theta^2}(\mathfrak{c'}) \rangle,
\end{align*}
where $\mathcal{P}_{\mathfrak{b}}'$ denotes the adjoint to $\mathcal{P}_{\mathfrak{b}}$ and we have used that $\mathcal{P}_{\mathfrak{b}}'[ \omega_J(g^{-\iota}) ] = C_2 [ \omega_J(g^{-\iota} \otimes \theta^{-1}) ]$ (\textit{cf}. \cite[Proposition 9.4]{GetzGore}). The commutativity of the diagram 
\[\xymatrix{ Y_H(K_1(\mathfrak{c}_F')) \ar[rr]^-{\iota} \ar[d]_-{P_{10}^{\mathfrak{c}_F'}} & & Y_G(K_1(\mathfrak{c}')) \ar[d]^-{P_{10}^{\mathfrak{c}'}} \\ Y_H(K_0(\mathfrak{c}_F'))  \ar[rr]^-{\iota} & & Y_G(K_0(\mathfrak{c}'))  } \]
implies that \begin{align*}
    P_{10, \ast}^{\mathfrak{c}'} \circ \iota_\ast(\mathbf{1}_{K_1(\mathfrak{c}_F')}) & =   \iota_\ast \circ P_{10, \ast}^{\mathfrak{c}_F'}(\mathbf{1}_{K_1(\mathfrak{c}_F')}) \\ 
    &= [K_0(\mathfrak{c}'_F):\mathcal{O}_F^{\ast, +} K_{1}(\mathfrak{c}'_F)] \iota_\ast(\mathbf{1}_{K_0(\mathfrak{c}_F')}),
\end{align*}    
with $[K_0(\mathfrak{c}'_F):\mathcal{O}_F^{\ast, +} K_{1}(\mathfrak{c}'_F)]$ being the degree of the natural degeneracy map $P_{10}^{\mathfrak{c}_F'}$. This and the definition of $\tilde{\iota}^{\hat{\underline{k}},\ast} $ give that
\begin{align*}
      \langle [ \omega_J(g^{-\iota} \otimes \theta^{-1}) ] , \mathscr{Z}^{\hat{\kappa},\chi\theta^2}(\mathfrak{c'}) \rangle
     &=  \langle  [ \omega_J(g^{-\iota} \otimes \theta^{-1}) ],   P_{10,\ast}^{\chi_0\theta^2_0} \circ \iota_* \circ \iota^{\hat{\underline{k}}}(\mathbf{1}_{K_1(\mathfrak{c}_F')}) \rangle \\ 
      &=  [K_0(\mathfrak{c}'_F):\mathcal{O}_F^{\ast, +} K_{1}(\mathfrak{c}'_F)]  \langle \tilde{\iota}^{\hat{\underline{k}},\ast} [ \omega_J(g^{-\iota} \otimes \theta^{-1}) ],  \mathbf{1}_{K_0(\mathfrak{c}_F')} \rangle. 
\end{align*}
The result then follows as in the first part of the statement.
\end{proof}
 
\begin{remark}
Lemma \ref{Pairingequalintegral} implies in particular that 
\[ \langle [ \omega_J(g^{-\iota}) ] , \mathscr{Z}^{\hat{\kappa}}(\mathfrak{c})_\theta \rangle = \langle \mathcal{Z}[ \omega_J(g^{-\iota})], [Z(\mathfrak{c})_\theta] \rangle_{IH}. \]
\end{remark}

Write \[\mathfrak{b}' : = \prod \{ \frakp \,:\, \frakp \mid \mathfrak{d}_{E/F} \text{ and } \frakp \nmid \mathfrak{f}(\theta) \}, \]
with $\mathfrak{f}(\theta)$ the conductor of $\theta$. Finally, \cite[Theorem 10.1]{GetzGore} reads as follows: 

\begin{theorem}\label{thm:integral=Lfunction}
For $f$ as above, let $\chi_F$ be a Hecke character for $F$ corresponding to the Nebentypus of $f$, and $\theta$ a Hecke character for $E$ such that $\theta_{\vert_{{\mathbb{A}}_F^\times}}=\eta_{E/F} \chi_F $. Then we have 
\[
\langle \omega_J(f^{-\iota}_E), \mathscr{Z}^{\hat{\kappa}}(\mathfrak{c})_{\theta} \rangle=\frac{c_1}{2^{d}}L^{\ast,\mathfrak{d}_{E/F}\mathfrak{b}_F }(\mathrm{Ad}(f) \otimes \eta_{E/F}, 1) \prod_{\mathfrak{q} \mid \mathfrak{b}'} L_{\mathfrak{q}}(\mathrm{As}(f_E \otimes \theta^{-1}), 1),
\]
where $L_{\mathfrak{q}}(\mathrm{As}(f_E \otimes \theta^{-1}), 1)$ denotes the Euler factor at ${\mathfrak{q}}$ of the Asai $L$-function of $f_E$ twisted by $\theta^{-1}$ and
\begin{align*}
c_1:=\frac{C_2  \Norm_{F/\Q}({\mathfrak{c}'_F}^2 D_{E/F})D_{F}^3 \zeta_F^{\mathfrak{d}_{E/F}\mathfrak{c}'_F}(1)^* }{2^{(\sum_{\sigma \in \Sigma_F} k_{\sigma})-2} R_F  |(\mathcal{O}_F/\mathfrak{c}'_F)^{\times}|},\end{align*}
for $R_F$ the regulator of $F$ and $\zeta_F^{\mathfrak{d}_{E/F}\mathfrak{c}'_F}(1)^*$ the residue at $s=1$ of the partial Dedekind zeta function of $F$ with the Euler factors at $\mathfrak{d}_{E/F}\mathfrak{c}'_F$ removed.
\end{theorem}

The proof of the Theorem follows from comparing the left hand side to the residue at $s=0$ of a Rankin--Selberg integral $I(f_E,s)$ which calculates the Asai $L$-function of $f_E \otimes \theta^{-1}$ (\textit{cf}. \cite[Chapter 10]{GetzGore} and \S \ref{sec:RSintegrals} below). Note that the adjoint $L$-value appears because of the decomposition of Proposition \ref{AsaiForBC}.

\section{A \texorpdfstring{$p$}{p}-adic tower of level subgroups and its Iwasawa theory}\label{Towerplusmoments}

In what follows, we use the method described succinctly in Loeffler's survey \cite{LoefflerSphericalvarieties} and \cite{LRZ} to show that our cycles vary in $p$-adic families, in a suitable cohomological sense. The results in \emph{loc. cit.} do not apply directly to our setting exactly because of Milne's axiom SV5 forcing the centre to contain no $\R$-split torus which is not already $\Q$-split. This excludes Hilbert modular varieties most notably. We provide the necessary details to show that this hypothesis can be removed in our case. This relies on the fact that the open compact level subgroups in our $p$-adic tower all share the same intersection with the center of $G$.

Recall that $L$ denotes a number field over which $H$ and $G$ split. Throughout the rest of the manuscript, we fix a prime $\mathfrak{m}$ of $L$ above $p$ and denote by $\mathcal{O}_\mathfrak{m}$ the localization of the ring of integers $\mathcal{O}:=\mathcal{O}_L$ at $\mathfrak{m}$.

\subsection{A spherical variety}
 
Let $B_G = T_G \cdot N_G$ denote the upper-triangular Borel subgroup of $G$ with Levi subgroup the maximal diagonal torus $T_G$ and unipotent radical $N_G$. We also let $\overline{B}_G =  T_G \cdot \overline{N}_G$ be the opposite parabolic. Similarly, denote by $B_H = T_H \cdot N_H = B_G \cap H$ the upper-triangular Borel subgroup of $H$.
Consider the flag variety $\mathcal{F} := G / \overline{B}_G$ over ${\rm Spec}\,  \Z_p$. It has a natural left action by $G$ and, thus, a left action by $H=G_F$ via the embedding $\iota$. Recall that $H(\Z_p) = \prod_{\mathfrak{p}|p} \GL_2(\O_\mathfrak{p})$ over primes $\mathfrak{p}$ in $F$ with $\O_\mathfrak{p} := \O_{F,\mathfrak{p}}$, and $G(\Z_p) = \prod_{\mathfrak{P}|p} \GL_2(\O_\mathfrak{P})$ over primes $\mathfrak{P}$ in $E$ with $\O_\mathfrak{P} := \O_{E,\mathfrak{P}}$. Via these identifications $\iota: H(\Z_p) \to G(\Z_p)$ can be written as $(\iota_\mathfrak{p})_\mathfrak{p}$ where each embedding \[\iota_\mathfrak{p}: \GL_2(\O_{\mathfrak{p}}) \hookrightarrow \prod_{\mathfrak{P}|\mathfrak{p}} \GL_2(\O_{\mathfrak{P}}) \]
is the natural map induced by the inclusion of $\O_{\mathfrak{p}}$ in $\O_{\mathfrak{P}}$. If $\mathfrak{p}$ splits in $E$, then $\mathfrak{p}=\mathfrak{P} \overline{\mathfrak{P}}$ (thus $\O_{\mathfrak{p}} = \O_{\mathfrak{P}}=\O_{\overline{\mathfrak{P}}}$) and $\iota_\mathfrak{p}(g) = (g,g)$. If $\mathfrak{p}$ is inert in $E$, there is only one prime $\mathfrak{P}$ above it and $\iota_\mathfrak{p}$ identifies $\GL_2(\O_{\mathfrak{p}})$ with the $\O_{\mathfrak{P}}$-valued matrices fixed by ${\rm Gal}(E_\mathfrak{P}/F_\mathfrak{p})$.

Let $\Upsilon_{F,p}$\index{$\Upsilon_{F,p}$} denote the set of primes of $F$ above $p$.  This can be divided into two subsets $\Upsilon_{F,p} = \Upsilon^{\rm s}_{F,p} \cup \Upsilon^{\rm i}_{F,p} $, where $\Upsilon^{\rm s}_{F,p}$ (resp. $\Upsilon^{\rm i}_{F,p}$) is the set of primes above $p$ which are split (resp. inert) in $E$. Then we may write \[ G(\Z_p)= \prod_{\mathfrak{p} \in \Upsilon^{\rm s}_{F,p}} \GL_2(\O_{\mathfrak{p}})^2 \times \prod_{\mathfrak{p} \in \Upsilon^{\rm i}_{F,p}} \GL_2(\O_{\mathfrak{P}}). \]

\begin{definition}\label{def:representativeorbit}
Let $u \in G(\Z_p)$\index{$u$} denote the element $(u_\mathfrak{p})_{\mathfrak{p} \in \Upsilon_{F,p}}$ given by \[u_\mathfrak{p} = \begin{cases} \left ( {\matrix {1} {1} {  } {1}}  , {\matrix {1} {} { } {1}} \right ) & \text{ if } \mathfrak{p} \in \Upsilon^{\rm s}_{F,p}, \\ {\matrix {1} {\delta_\mathfrak{p}} {}  {1}} & \text{ if } \mathfrak{p} \in \Upsilon^{\rm i}_{F,p},  \end{cases} \]
with $\delta_\mathfrak{p} \in \O_{\mathfrak{P}} \setminus \O_{\mathfrak{p}} $ such that $\O_{\mathfrak{P}} = \O_{\mathfrak{p}} \oplus \O_{\mathfrak{p}} \delta_\mathfrak{p}$, and $\delta_\mathfrak{p}^2 \in \O_{\mathfrak{p}}$.
\end{definition}

Finally, we let $L_H$ denote the torus of $H(\Z_p)$ given by 
\[ L_H = \prod_{\mathfrak{p} \in \Upsilon^{\rm s}_{F,p}}T_{(1,0)}  \times \prod_{\mathfrak{p} \in \Upsilon^{\rm i}_{F,p}} T_{\delta_\mathfrak{p}}, \]
where $T_{(1,0)}$ denotes the conjugate\footnote{Note that $T_{(1,0)} = y T_{\GL_2} y^{-1}$, with $y = {\matrix {1 } { } {1} {1 }}$.} of the maximal split torus of of $\GL_2(\O_{\mathfrak{p}})$ given by  \[\{ \left(\begin{smallmatrix}a &  \\a-d  & d\end{smallmatrix}\right)\in \GL_2(\O_{\mathfrak{p}})\}, \] while $T_{\delta_\mathfrak{p}}$ denotes the maximal non-split torus of $\GL_2(\O_{\mathfrak{p}})$ given by \[\{ \left(\begin{smallmatrix}a & c\delta_\mathfrak{p}^2 \\c  & a\end{smallmatrix}\right)\,:\, (a,c) \in \O_{\mathfrak{p}}^2 \text{ such that } a+c\delta_\mathfrak{p} \in \O_{\mathfrak{P}}^\times \}. \]

\begin{lemma}\label{openorbitlemma} \leavevmode
\begin{enumerate}
    \item  The $H$-orbit of $u$ in $\mathcal{F} / {\rm Spec}\,\Z_p$ is Zariski open,
    \item ${\rm Stab}_{H} (u \overline{B}_G)(\Z_p)=  \iota( L_H ) .$
\end{enumerate}
\end{lemma}

\begin{proof} 

We start with calculating the stabilizer of $u \overline{B}_G$, which is  $$ {\rm Stab}_{H} (u \overline{B}_G) =  \iota( H) \cap u \overline{B}_G u^{-1}.$$
Let $\mathfrak{p}  \in \Upsilon^{\rm s}_{F,p}$, then $\GL_2(\O_{\mathfrak{p}}) \ni g={\matrix {a} { b} {c} {d}}$ fixes $[u_{\mathfrak{p}}]$ if $g$ and $  {\matrix {1 } { -1} {} {1 }} \cdot g \cdot {\matrix {1 } {1} { } { 1}}$ are lower-triangular. As \[ u_{\mathfrak{p}}^{-1} \cdot g \cdot u_{\mathfrak{p}} = \left({\matrix {a-c} {a+b-c-d} {c} {d+c}}, {\matrix {a} {b} {c} {d}} \right)
,\]
we deduce that $g ={\matrix {a} {} { a-d} {d}} $. Therefore any element in $ u_{\mathfrak{p}}^{-1} \GL_2(\O_{\mathfrak{p}}) u_{\mathfrak{p}} \cap \overline{B}_{\GL_2}(\O_{\mathfrak{p}})^2$ is of the form $\left(  {\matrix {d} { } { a-d} {a}}, {\matrix {a} { } {a-d} {d}}\right)$. In other words, $\GL_2(\O_{\mathfrak{p}}) \cap u_{\mathfrak{p}} \overline{B}_{\GL_2}(\O_{\mathfrak{p}})^2 u_{\mathfrak{p}}^{-1}  = \iota_\mathfrak{p}(T_{(1,0)})$.

Let $\mathfrak{p}  \in \Upsilon^{\rm i}_{F,p}$. Keeping the notation as above,  
\[  u_{\mathfrak{p}}^{-1}  \cdot {\matrix {a} {b } {c} {d}} \cdot u_{\mathfrak{p}} = {\matrix { a - \delta_\mathfrak{p} c} {\delta_\mathfrak{p}(a-d) + b - \delta_\mathfrak{p}^2 c} {c} {d + \delta_\mathfrak{p} c}} \in \overline{B}_{\GL_2}(\O_{\mathfrak{P}}) \] 
implies that $b= \delta_\mathfrak{p}^2 c $ and that $d= a$. Thus, in this case $g \in T_{\delta_\mathfrak{p}}$. Hence $  \GL_2(\O_{\mathfrak{p}}) \cap u_{\mathfrak{p}} \overline{B}_{\GL_2}(\O_{\mathfrak{P}}) u_{\mathfrak{p}}^{-1}= \iota_\mathfrak{p}( T_{\delta_\mathfrak{p}} )  $. 

Summing up, we deduce that  ${\rm Stab}_{H} (u \overline{B}_G)(\Z_p) =   \iota (L_H )$. Finally, notice that the $H$-orbit of $u$ is open for the Zariski topology as its dimension equals the dimension of $\mathcal{F}$.
\end{proof} 

\begin{definition}
We let $L_G$ be the subtorus of $T_G$ given by $L_G = \{ {\matrix {\alpha} { } {} {\tau(\alpha)}}\,:\, \alpha \in {\rm Res}_{\mathcal{O}_E/\Z} \mathbf{G}_{\rm m} \}$ and denote $Q_G:=L_G \cdot N_G$, $\overline{Q}_G:=L_G \cdot \overline{N}_G$. 
\end{definition}

\begin{corollary}\label{StupRMKonLG}
We have an inclusion $u^{-1} \iota (L_H) u \subset \overline{Q}_G(\Z_p) \subset \overline{B}_G(\Z_p)$.
\end{corollary}

\begin{proof}
This follows directly from the proof of the Lemma above. For instance, at an inert prime $\mathfrak{p}$, if $g=\left(\begin{smallmatrix} a & c \delta_\mathfrak{p}^2 \\ c & a \end{smallmatrix}\right) \in \iota_\mathfrak{p}( T_{\delta_\mathfrak{p}} )$, then  
\[u_{\mathfrak{p}}^{-1} g u_{\mathfrak{p}} =  {\matrix { a - \delta_\mathfrak{p} c} {} {c} {a + \delta_\mathfrak{p} c}}, \]
which is lower-triangular, with diagonal elements conjugate to each other as desired.
\end{proof}

In the course of the proof of Theorem \ref{BIGCLASS}, we will need an explicit description of the action of $u^{-1}$ on the invariant vector $\Delta^{\underline{k}} \in V^{\underline{k},0}_{\mathcal{O}_\mathfrak{m}}$ introduced in Formula \eqref{InvVectus}. We record it here. The element $u$ or its inverse act on the $G(\mathcal{O}_\mathfrak{m})$-lattice $V^{\underline{k},0}_{\mathcal{O}_\mathfrak{m}}$ via the embedding   \[j: G(\Z_p) = \GL_2(  \mathcal{O}_E \otimes_\Z \Z_p) \hookrightarrow G(\mathcal{O}_\mathfrak{m}) = \GL_2(\mathcal{O}_E \otimes_\Z \mathcal{O}_\mathfrak{m}) \simeq \GL_2(\mathcal{O}_\mathfrak{m})^{2d}, \]
which we now recall. Write $\Z_p \otimes_\Z \mathcal{O}_E = \prod_{ \mathfrak{p} \in \Upsilon^{\rm s}_{F,p}} \mathcal{O}^2_\mathfrak{p} \times   \prod_{\mathfrak{p} \in \Upsilon^{\rm i}_{F,p}} \mathcal{O}_\mathfrak{P}$ as above, so that \[ G(\Z_p)= \prod_{\mathfrak{p} \in \Upsilon^{\rm s}_{F,p}} \GL_2(\O_{\mathfrak{p}})^2 \times \prod_{\mathfrak{p} \in \Upsilon^{\rm i}_{F,p}} \GL_2(\O_{\mathfrak{P}}). \]
At each $\mathfrak{p}  \in \Upsilon^{\rm s}_{F,p}$, then $j$ is induced by the natural embedding of $\O_{\mathfrak{p}} \hookrightarrow \O_{\mathfrak{p}} \otimes_{\Z_p} \O_\mathfrak{m} \simeq \O_\mathfrak{m}^{d_\mathfrak{p}}$. At each $\mathfrak{p}  \in \Upsilon^{\rm i}_{F,p}$, $j$ is also induced by the embedding $\GL_2(\O_{\mathfrak{P}}) \hookrightarrow \GL_2(\O_{\mathfrak{P}} \otimes_{\Z_p} \O_\mathfrak{m})$ together with the isomorphism $\O_{\mathfrak{P}} \otimes_{\O_{\mathfrak{p}}} (\O_{\mathfrak{p}} \otimes_{\Z_p} \O_\mathfrak{m}) \simeq (\O_{\mathfrak{p}} \otimes_{\Z_p} \O_\mathfrak{m})^2$ given by sending $\alpha \otimes \beta \mapsto (\beta i_\mathfrak{P}(\alpha), \beta i_\mathfrak{P}(\bar{\alpha}))$, with $i_\mathfrak{P} :\O_{\mathfrak{P}}  \hookrightarrow \O_{\mathfrak{p}} \otimes_{\Z_p} \O_\mathfrak{m}$\index{$i_\mathfrak{P}$}, and the isomorphism $\O_{\mathfrak{p}} \otimes_{\Z_p} \O_\mathfrak{m} \simeq \O_\mathfrak{m}^{d_\mathfrak{p}}$. Note that the sum of the $d_\mathfrak{p}$'s is $d$, and when $p$ is totally split in $F$ then $d_\mathfrak{p} =1$ for every $\mathfrak{p} \in \Upsilon_{F,p}$.

\begin{definition}\label{Definition_partition_infinity_to_padic}
The fixed isomorphism $i_p: \C \simeq \overline{\Q}_p$ yields a partition of $\Sigma_F = \bigsqcup_{\frakp \in \Upsilon_{F,p}} \Sigma_\frakp$, where, for every prime ideal $\frakp \in \Upsilon_{F,p}$, we let $\Sigma_\frakp$\index{$\Sigma_\frakp$} be the subset of $\Sigma_F$ of ($p$-adic) places $\sigma$ sending $\frakp$ to the maximal ideal of the valuation ring of $F_\frakp$. Similarly, this induces a partition $\Phi = \bigsqcup_{\frakp \in \Upsilon_{F,p}} \Phi_\frakp$.
\end{definition}
\noindent Note that, for each $\frakp \in \Upsilon_{F,p}$, $\Sigma_\frakp$ and $\Phi_\frakp$ consist of $d_\frakp$ elements.

At each $\mathfrak{p}  \in \Upsilon^{\rm s}_{F,p}$, $u^{-1}_\frakp$ acts on the components in $\Phi_\frakp$ of $ \Delta^{\underline{k}}$ via the action of $\left ( {\matrix {1} {-1} {  } {1}}, {\matrix {1} {} { } {1}} \right )$, which sends $\Delta_\sigma=(e_{1,\sigma}e_{2,\sigma'}-e_{2,\sigma}e_{1,\sigma'})$ to $(e_{1,\sigma}e_{1,\sigma'})+ \Delta_\sigma$ and so, taking the $k_\sigma$ tensor product, it sends $\Delta_\sigma^{k_\sigma}=(e_{1,\sigma}e_{2,\sigma'}-e_{2,\sigma}e_{1,\sigma'})^{\otimes {k_\sigma}}$ to $(e_{1,\sigma}e_{1,\sigma'})^{\otimes {k_\sigma}} + R_\sigma^{k_\sigma}$, where $R_\sigma^{k_\sigma}$ has zero projection to the highest weight eigenspace. 

At each $\mathfrak{p}  \in \Upsilon^{\rm i}_{F,p}$, recall we have fixed $\delta_\mathfrak{p} \in \O_{\mathfrak{P}}$ such that $\O_{\mathfrak{P}} = \O_{\mathfrak{p}} \oplus \O_{\mathfrak{p}} \delta_\mathfrak{p}$ and $\delta_\mathfrak{p}^2 \in \O_{\mathfrak{p}}$. Then $u^{-1}_\frakp$ acts on the components in $\Phi_\frakp$ of $ \Delta^{\underline{k}}$ via the action of $\left ( {\matrix {1} {-i_\mathfrak{P}(\delta_\mathfrak{p})} {  } {1}}, {\matrix {1} {i_\mathfrak{P}(\delta_\mathfrak{p})} { } {1}} \right )$, which sends $\Delta_\sigma$ to  $(2 i_\mathfrak{P}(\delta_\mathfrak{p}))(e_{1,\sigma}e_{1,\sigma'}) + \Delta_\sigma$. Summing up, we have shown the following : 

\begin{lemma}\label{rmk:uflag}
We have that \[u^{-1} \cdot \Delta^{\underline{k}} = a_{\underline{k}} \cdot v_{\underline{k}}^{\rm hw} +  R^{\underline{k}}, \] 
where $v_{\underline{k}}^{\rm hw}$ is the highest weight vector in $V^{\underline{k},0}_{\mathcal{O}_\mathfrak{m}}$, $a_{\underline{k}} = \prod_{\frakp \in \Upsilon^{\rm i}_{F,p}} (2 i_\mathfrak{P}(\delta_\mathfrak{p}))^{\sum_{\sigma \in \Phi_\frakp} k_\sigma} \in \O_\mathfrak{m}^\times$, and $R^{\underline{k}} \in V^{\underline{k},0}_{\mathcal{O}_\mathfrak{m}}$ has no $v_{\underline{k}}^{\rm hw}$-component.
\end{lemma}

\subsection{Level subgroups and Hecke operators}

Let $\eta$\index{$\eta$} denote the diagonal matrix ${\matrix {p} { } {} {1}} \in G(\Q_p)$ and let \[ \Gamma_0(p^r) : = \{g \in G(\Z_p)\,:\,  g\,[\text{mod }p^r] \in B_G(\Z/p^r\Z)\} \]
be the Iwahori subgroup of $G(\Z_p)$. Similarly, $\Gamma^0(p^r)$ denotes the opposite Iwahori subgroup of level $p^r$ of $G(\Z_p)$, while $\Gamma^0_0(p^r,p^s)$ denotes the intersection of $\Gamma_0(p^r)$ with $\Gamma^0(p^s)$. Following \cite{LoefflerSphericalvarieties}, we define the following level subgroups. 

\begin{definition}\label{def:levelgroups} Fix $K^{(p)} \subseteq G(  { \hat{\Z} ^{(p)} } )$ a neat open compact subgroup and for $r\geq 0$ let \begin{align*}
K_r &: = K^{(p)}\cdot \{ g \in \Gamma_0(p)\,:\, \eta^{-r} g \eta^r \in G (\Z_p),\; g\,[\text{mod }p^r] \in \overline{Q}_G(\Z/p^r\Z)  \}, \\
K_r' &: = K^{(p)}\cdot \{ g \in \Gamma_0(p)\,:\, \eta^{-r-1} g \eta^{r+1} \in G (\Z_p),\; g\,[\text{mod }p^r] \in \overline{Q}_G(\Z/p^r\Z)   \}, \\ 
V_r &:= \eta^{-r} K_r \eta^r, \\ 
V_r' &:= \eta^{-r-1} K_r' \eta^{r+1}.
\end{align*} 
\end{definition}

\begin{remark}\label{rmk:factsabouttower}\leavevmode
\begin{itemize}
\item Explicitly, $K_0 = K^{(p)}\cdot \Gamma_0(p)$, $K_0' = K^{(p)}\cdot \Gamma_0^0(p,p)$, and  \[K_1 = K^{(p)}\cdot \{ g \in G(\Z_p)\,:\,  g\,[\text{mod }p] \in L_G(\Z/p\Z)  \}.\]
    \item Notice that $V_0=K_0$, $V_0' = K^{(p)}\cdot \eta^{-1} \Gamma_0^0(p,p) \eta =  K^{(p)}\cdot \Gamma_0(p^2)$. In general, \[V_r = K^{(p)}\cdot \{ g \in \Gamma_0(p^{r+1})\,:\,  g\,[\text{mod }p^r] \in Q_G(\Z/p^r\Z)   \}. \]
 The tower of subgroups $(V_r)_r$ will appear when we define our Iwasawa cohomology groups.
 \item  $K_r'$ is the natural subgroup used to define the Hecke correspondence for the ``$U_p$''-operator associated with $\eta^{-1}$.
\end{itemize}    
\end{remark}

 Note that $ K_{r+1} \subseteq K_r' = K_r \cap \eta K_r \eta^{-1} \subseteq K_r$. This implies that right multiplication by $\eta$ gives a morphism \[ [\eta]_{r+1,r}: \xymatrix{    Y_G(K_{r+1}) \ar[r]^{{\rm pr}} & Y_G(K_r') \ar[r]^-{\eta} &  Y_G(\eta^{-1}K_r' \eta) = Y_G(\eta^{-1}K_r \eta \cap K_r) \ar[r]^-{{\rm pr}} &  Y_G(K_r),}  \]
where ${\rm pr}$ denote natural projections.

We note that we can refine Definition \ref{def:levelgroups} by defining towers of level subgroups for each subset of $ \Upsilon_{F,p}$, as follows.

For each $\frakp \in \Upsilon_{F,p}$, let $\eta_\frakp : = \matrix{\pi_\frakp}{}{}{1} \in \GL_2(F_\frakp)$\index{$\eta_\frakp$} and denote in the same way the corresponding element of $H(\Q_p)$ with trivial component at every other place in $\Upsilon_{F,p}$. Moreover, for any $r \geq 0$ set $K_{\frakp,r}$ to be either \[\left\{(g_1,g_2) \in {\GL}_2(\mathcal{O}_\frakp)^2\,:\,  (g_1,g_2) \equiv \left(\left(\begin{smallmatrix}
    a & \\ x_1 & d
\end{smallmatrix} \right),\left(\begin{smallmatrix}
    d & \\ x_2 & a
\end{smallmatrix} \right)\right)\,[\text{mod }\frakp^r] \right\} \cap \Gamma_0(\pi_\frakp)^2\]
if $\frakp \in \Upsilon^{\rm s}_{F,p}$, or 
\[\left\{g \in \GL_2(\mathcal{O}_\mathfrak{P})\,:\,  g \equiv  \left(\begin{smallmatrix}
    \alpha & \\ x & \tau(\alpha)
\end{smallmatrix} \right)\,[\text{mod }\mathfrak{P}^r] \right\} \cap \Gamma_0(\pi_\frakp)\]
if $\frakp \in \Upsilon^{\rm i}_{F,p}$, with $\mathfrak{P}$ denoting the prime of $E$ above $\frakp$. We then let $K_{\frakp,r}' := K_{\frakp,r} \cap \eta_\frakp K_{\frakp,r} \eta_\frakp^{-1}$, $V_{\frakp,r}:= \eta_\frakp^{-r} K_{\frakp,r} \eta_\frakp^r$, and $V_{\frakp,r}':= \eta_\frakp^{-r-1} K_{\frakp,r} \eta_\frakp^{r+1}$.

\begin{definition}\label{def:multilevels}
    For any multi-index $\pmb{r} = (r_\frakp)_\frakp \in \Z_{\geq 0}^{|\Upsilon_{F,p}|}$, we define 
    \begin{align*}
       K_{\pmb{r}}&:= K^{(p)} \cdot \prod_{\frakp \in \Upsilon_{F,p}} K_{\frakp,r_\frakp}, \\
       V_{\pmb{r}}&:= K^{(p)} \cdot \prod_{\frakp \in \Upsilon_{F,p}} V_{\frakp,r_\frakp}, 
    \end{align*}
    and similarly for $K_{\pmb{r}}'$ and $V_{\pmb{r}}'$.
\end{definition}
\noindent When $\pmb{r} = (r)_\frakp$, then these level groups agree with those of Definition \ref{def:levelgroups}.

Later on, if
$\mathfrak{S}$ is any subset of $ \Upsilon_{F,p}$, we also consider 
\begin{align}\label{eqdefintermediatelevelsigma}
    K_0(\mathfrak{S}) : = K^{(p)} \cdot (  \prod_{\frakp \in \mathfrak{S}} \Gamma^0_0(\pi_\frakp,\pi_\frakp) \cdot  \prod_{\frakp \in \Upsilon_{F,p} \smallsetminus \mathfrak{S}} \Gamma_0( \pi_\frakp)).
\end{align}
We have \[K_{\pmb{0}}' = K_0' \subseteq  K_0(\mathfrak{S}) \subseteq K_{\pmb{0}} = K_0,\]
where the symbol $\pmb{0}$\index{$\pmb{0}$} denotes the multi-index $(0,\dots,0)$.

We now define certain Hecke operators that play a crucial role in the manuscript.

\begin{lemma}\label{Lem:traceeta}
 There is a well-defined action of $p^{\sum_\sigma{(k_{\sigma}+m)/2}} \cdot \eta^{-1}$ on $V_{\mathcal{O}_\mathfrak{m}}^{\underline{k},m}$ which induces a morphism of local systems \[ \eta_\flat : \mathscr{V}_{\mathcal{O}_\mathfrak{m}}^{\underline{k},m} \to \eta^*\mathscr{V}_{\mathcal{O}_\mathfrak{m}}^{\underline{k},m}. \]
 In particular we have a map
 \[ {\rm Tr}_{\eta} : H^i(Y_G(K_{\pmb{r}}'),\mathscr{V}_{\mathcal{O}_\mathfrak{m}}^{\underline{k},m} ) \longrightarrow H^i(Y_G(K_{\pmb{r}}),\mathscr{V}_{\mathcal{O}_\mathfrak{m}}^{\underline{k},m} )  \]
 defined by composing the map in cohomology induced by $\eta_\flat$ with the trace of ${\rm pr} \circ \eta: Y_G(K_{\pmb{r}}') \to Y_G(K_{\pmb{r}})$  in Betti cohomology.
\end{lemma}

\begin{proof}
Note that the torus $\{ \left(\begin{smallmatrix} x & \\ & 1\end{smallmatrix} \right)\}$ acts on each weight space of ${\rm Sym}^{k_\sigma} (V_\sigma) \otimes \mathrm{det}_\sigma^{\frac{m-k_{\sigma}}{2}}$ by the character $\left(\begin{smallmatrix} x & \\ & 1 \end{smallmatrix} \right) \mapsto x^{\frac{m+k_{\sigma} - 2 \ell_\sigma}{2}}$, where $\ell_\sigma$ is an integer ranging between $0$ to $k_\sigma$. Hence, the matrix $\left(\begin{smallmatrix} p^{-1} & \\ & 1\end{smallmatrix} \right)$ acts on each weight space of ${\rm Sym}^{k_\sigma} (V_\sigma) \otimes \mathrm{det}_\sigma^{\frac{m-k_{\sigma}}{2}}$ by $p^{\frac{2 \ell_\sigma -k_{\sigma}-m}{2}}$. This implies that the action of $\eta^{-1}$ on each weight space of $V^{\underline{k},m}$ will be given by $p^{\sum_\sigma{(2 \ell_\sigma - k_{\sigma}-m)/2}}$, for some $\underline{\ell}=(\ell_\sigma)_\sigma$ such that $0 \leq \ell_\sigma \leq k_\sigma$. In particular, $\eta^{-1}$ acts on the highest weight subspace $L_\mathfrak{m} \cdot v_{\underline{k},m}^{\rm hw}$ by $p^{-\sum_\sigma{(k_{\sigma}+m)/2}}$.
As every admissible lattice is the direct sum of its intersections with the weight spaces of $V^{\underline{k},m}$, we get an action of $p^{\sum_\sigma{(k_{\sigma}+m)/2}} \cdot \eta^{-1}$ on $V^{\underline{k},m}_{\mathcal{O}_\mathfrak{m}}$ such that it is trivial on the highest weight subspace $\mathcal{O}_\mathfrak{m} \cdot v_{\underline{k},m}^{\rm hw}$ and divisible by a positive power of $p$ on every other weight subspace. By letting $ \eta_\flat : \mathscr{V}_{\mathcal{O}_\mathfrak{m}}^{\underline{k},m} \to \eta^*\mathscr{V}_{\mathcal{O}_\mathfrak{m}}^{\underline{k},m}$ be the corresponding morphism of local systems, we get the result.
\end{proof}

\begin{definition}\label{Def:traceetar1r}
Let the symbol $\pmb{1}$\index{$\pmb{1}$} denote the multi-index with all entries equal to $1$. We define \[{\rm Tr}_{[\eta]_{\pmb{r}+\pmb{1},\pmb{r}}}: H^i(Y_G(K_{\pmb{r}+\pmb{1}}),\mathscr{V}_{\mathcal{O}_\mathfrak{m}}^{\underline{k},m} ) \longrightarrow H^i(Y_G(K_{\pmb{r}}),\mathscr{V}_{\mathcal{O}_\mathfrak{m}}^{\underline{k},m} ) \]\index{${\rm Tr}_{[\eta]_{\pmb{r}+\pmb{1},\pmb{r}}}$} as the composition of the trace of the natural projection $Y_G(K_{\pmb{r}+\pmb{1}}) \to Y_G(K_{\pmb{r}}')$ with ${\rm Tr}_\eta$.
\end{definition}

\begin{remark}\label{rem:normalizationTp}
It is clear from the proof of the Lemma above that the trace ${\rm Tr}_\eta$ is optimally normalized so that $p^{\sum_\sigma{(k_{\sigma}+m)/2}} \cdot \eta^{-1}$ acts as $1$ on the highest weight subspace and divisible by $p$ elsewhere.
\end{remark}

\begin{definition}\label{definitionTeta} 
We let $T_\eta$\index{$T_\eta$} be the Hecke correspondence acting on cohomology given by ${\rm Tr}_\eta \circ (\pi_{\pmb{r}}^{G})^*$, where $\pi_{\pmb{r}}^{G}:Y_G(K_{\pmb{r}}') \to Y_G(K_{\pmb{r}})$ is the natural degeneracy map and ${\rm Tr}_\eta$ is the map defined in Lemma \ref{Lem:traceeta}. 
\end{definition}

\noindent Note that $T_\eta$ is the correspondence associated with the Hecke operator $[K_{\pmb{r}} \eta^{-1} K_{\pmb{r}}]$ in the Hecke algebra $\mathcal{H}(K_{\pmb{r}} \backslash G(\A_f) / K_{\pmb{r}})_{\mathcal{O}_\mathfrak{m}}$ of $K_{\pmb{r}}$-bi-invariant smooth compactly supported $v$-valued functions on $G(\A_f)$.

\begin{remark}\label{rmk:normofalletap}
    In the same analogous way, attached to each $\eta_\frakp$, we have the Hecke operator $T_{\eta_\frakp}$, which we normalize so that \[ T_{\eta} = \prod_{\frakp \in \Upsilon_{F,p}} T_{\eta_\frakp}.\]\index{$T_{\eta_\frakp}$}
Concretely, we normalize $T_{\eta_\frakp}$ as follows.
Recall that, in Definition \ref{Definition_partition_infinity_to_padic}, we have fixed a partition of $\Sigma_F = \bigsqcup_{\frakp \in \Upsilon_{F,p}} \Sigma_\frakp$. Then, we multiply the trace of $\eta_{\frakp}$ by $p^{\sum_{\sigma \in \Sigma_{\frakp}}  (k_{\sigma} + k_{\sigma'} + 2 m )/2}$, where we identify $\Sigma_{\frakp}$ with the subset $\Phi_\frakp$ of the $F$-type $\Phi$.
Note that the (normalized) trace ${\rm Tr}_{[\eta]_{\pmb{r}+\pmb{1},\pmb{r}}}$ is equal to the composition of the (normalized) traces \[{\rm Tr}_{[\eta_\frakp]_{\pmb{r}+\pmb{1},\pmb{r}'}}: H^i(Y_G(K_{\pmb{r}+\pmb{1}}),\mathscr{V}_{\mathcal{O}_\mathfrak{m}}^{\underline{k},m} ) \longrightarrow H^i(Y_G(K_{\pmb{r}'}),\mathscr{V}_{\mathcal{O}_\mathfrak{m}}^{\underline{k},m} ),\]\index{${\rm Tr}_{[\eta_\frakp]_{\pmb{r}+\pmb{1},\pmb{r}'}}$} with $\pmb{r}'$ equal to $\pmb{r}+\pmb{1}$ away from $\frakp$ and equal to $\pmb{r}$ at $\frakp$, as $\frakp$ runs over $\Upsilon_{F,p}$.
\end{remark}

\subsection{A few Cartesian diagrams}

Let $u \in G(\Z_p)$ denote the element of Definition \ref{def:representativeorbit}. For any multi-index $\pmb{r} = (r_\frakp)_\frakp \in \Z_{\geq 0}^{|\Upsilon_{F,p}|}$, we have the morphism
\[ \xymatrix{    u \circ \iota_{\pmb{r}}: Y_H(u K_{\pmb{r}}u^{-1} \cap H)  \ar[r]^-{\iota_{\pmb{r}}} & Y_G(u K_{\pmb{r}}u^{-1}) \ar[r]^-{u} &  Y_G(K_{\pmb{r}}).}  \]

\begin{lemma}\label{closedimm}
Suppose that there exists a neat open compact subgroup $U$ of $G(\A_f)$ containing $K_{\pmb{r}}$ and its conjugate $K_{\pmb{r}}^\tau$ via the non-trivial element $\tau \in {\rm Gal}(E/F)$. Then the map $u \circ \iota_{\pmb{r}}$ is a closed immersion.
\end{lemma}
\begin{proof}
Notice that it is enough to check it at the level of $\C$-points. As right multiplication by $u$ is clearly a bijection, the statement follows from showing that $\iota_{\pmb{r}}$ is injective. To ease the notation, let $V$ denote $u K_{\pmb{r}}u^{-1}$. Recall that, if we denote by $Y_H$ and $Y_G$ the varieties at infinite level, we have using \cite[(1.8.2)]{DeligneBourbakiShimura} and that the center acts trivially
$$ Y_H(\C) = \tfrac{H(\Q)}{Z_H(\Q)} \backslash \left( X_H \times H(\A_f) / \overline{Z_H(\Q)} \right),$$
where $\overline{Z_H(\Q)}$ denotes the closure of $Z_H(\Q)$ in $Z_H(\A_f)$ and where $X_H$ is isomorphic to $d$ copies of the union $\mathfrak{H} \cup \mathfrak{H}^{-}$ of upper and lower half planes. Similarly,
$$ Y_G(\C) = \tfrac{G(\Q)}{Z_G(\Q)} \backslash \left( X_G \times G(\A_f) / \overline{Z_G(\Q)} \right),$$ with $X_G$ isomorphic to $2d$ copies of $\mathfrak{H} \cup \mathfrak{H}^{-}$. Note that $\iota: H \to G$ induces an injection $Y_H(\C) \hookrightarrow Y_G(\C)$ since, in this case, $$ \overline{Z_G(\Q)} \cap H(\A_f) =  \overline{Z_H(\Q)}.$$
The result thus follows from showing that $Y_H(\C) \cap Y_H(\C)v = \emptyset$ for any $v \in V \smallsetminus V \cap H(\A_f)$. We show this fact as follows. 
Firstly, notice that the non-trivial element $\tau \in {\rm Gal}(E/F)$ induces a map $ \tau : Y_G(\C) \to Y_G(\C)$, which fixes $Y_H(\C)$ point-wise.  
Suppose now that $y,y v \in Y_H(\C)$ for $v \in V$. Then $v^{-1} v^\tau$ fixes $y$. Since $v^{-1} v^\tau$ belongs to $u U u^{-1}$ and, by the neatness assumption, we deduce $v^{-1} v^\tau=1$, so $v = v^\tau$, which implies that $v \in H(\A_f)$, as desired.
\end{proof}

\begin{remark}
The hypothesis of Lemma \ref{closedimm} is far from being necessary, but as it will always be the case in our applications, we are content with it. 
\end{remark}

Observe that, for instance, the hypothesis of Lemma \ref{closedimm} is satisfied whenever $K^{(p)}$ is stable under the non-trivial element $\tau \in {\rm Gal}(E/F)$. From now on, we then assume that this is the case.

\begin{proposition} \label{Cartesian}
Assume that $K^{(p)}$ is stable under $\tau \in {\rm Gal}(E/F)$ and let $\pmb{r} = (r_\frakp)_\frakp \in \Z_{\geq 1}^{|\Upsilon_{F,p}|}$. The commutative diagram 
\begin{eqnarray*}  \xymatrix{  & & & Y_G(K_{\pmb{r} + \pmb{1}}) \ar[d]^{{\rm pr}} \ar[ddr]^-{[\eta]_{\pmb{r} + \pmb{1},\pmb{r}}} & \\ 
Y_H(uK_{\pmb{r} + \pmb{1}}u^{-1} \cap H)  \ar[d]_{\pi_{\pmb{r}}^H  } \ar[urrr]^-{u \circ \iota_{\pmb{r} + \pmb{1}}}  \ar[rrr]^-{{\rm pr} \circ u \circ \iota_{\pmb{r} + \pmb{1}}}   & & &  Y_G(K_{\pmb{r}}') \ar[d]_{  \pi_{\pmb{r}}^G} \ar[dr]^\eta &\\
Y_H(u K_{\pmb{r}}u^{-1} \cap H)  \ar[rrr]^-{ u \circ \iota_{\pmb{r}}} & & &  Y_G(K_{\pmb{r}}) \ar@{-->}[r]^{T_\eta}  & Y_G(K_{\pmb{r}}) }  \end{eqnarray*}
has Cartesian bottom square. Here $\pi_{\pmb{r}}^H, \pi_{\pmb{r}}^G$ denote the natural degeneracy maps.
\end{proposition}

\begin{proof}
To show that the diagram is Cartesian, it is enough to check it for the complex points of the varieties. Then the result follows from showing that \begin{enumerate}
    \item  $u \circ \iota_{\pmb{r}}$ and ${\rm pr} \circ u \circ \iota_{{\pmb{r}}+\pmb{1}}$ are closed immersions,
    \item The maps $\pi_{\pmb{r}}^H$ and $\pi_{\pmb{r}}^G$ are covering maps of the same degree.
\end{enumerate} 

By Lemma \ref{closedimm}, $u \circ \iota_{\pmb{s}}$ is a closed immersion for any ${\pmb{s}}$, therefore (1) follows from showing that  ${\rm pr} \circ u \circ \iota_{{\pmb{r}}+\pmb{1}}$ is a closed immersion, i.e. 
\[ u K'_{\pmb{r}} u^{-1} \cap H =  u K_{{\pmb{r}}+\pmb{1}} u^{-1} \cap H. \]
This follows as in the proof of \cite[Lemma 4.4.1(i)]{LoefflerSphericalvarieties} from the calculation of Lemma \ref{openorbitlemma}(2) applied modulo $p^{{\pmb{r}}+{\pmb{1}}}$. Here, for any multi-index \pmb{r}, we denote $p^{{\pmb{r}}}:= \prod_{\frakp}\frakp^{r_\frakp}$.\index{$p^{{\pmb{r}}}$} 

Notice that $K_{\pmb{r}}' \cap Z_G(\Q) = K_{\pmb{r}} \cap Z_G(\Q)$. Indeed, if $\gamma \in K_{\pmb{r}}' \cap Z_G(\Q)$, then at $p$,  $\eta^{-\pmb{r}-\pmb{1}}\gamma \eta^{\pmb{r}+\pmb{1}} = \gamma \in G(\Z_p)$, with $\eta^{\pmb{s}}:=\prod_\frakp \eta_\frakp^{s_\frakp}$, and $ \gamma  \,[\text{mod }p^{\pmb{r}}] \in \overline{Q}_G$. These are exactly the conditions at $p$ characterising the intersection $K_{\pmb{r}} \cap Z_G(\Q)$. Hence, by Lemma \ref{verticaldeg}, the degrees of the maps $\pi_{\pmb{r}}^H$ and $\pi_{\pmb{r}}^G$ equal to $[u K_{\pmb{r}} u^{-1} \cap H : u K_{\pmb{r}}' u^{-1} \cap H ]$ and $[K_{\pmb{r}}: K_{\pmb{r}}']$ respectively. We are thus left to show that \[[u K_{\pmb{r}} u^{-1} \cap H : u K_{\pmb{r}}' u^{-1} \cap H ] = [K_{\pmb{r}}: K_{\pmb{r}}']. \]
This follows as in the proof of \cite[Lemma 4.4.1(ii)]{LoefflerSphericalvarieties} from the fact that the $H$-orbit of $u \overline{B}_G$ is open as a $\Z_p$-subscheme of $\mathcal{F}$ (\textit{cf}. Lemma \ref{openorbitlemma}(1)).
\end{proof}

Proposition \ref{Cartesian} does not follow directly from \cite{LoefflerSphericalvarieties} as our Shimura data do not satisfy the Axiom SV5. However, to circumvent this obstruction it suffices to use the following:

\begin{lemma}\label{verticaldeg}
 Let $K,K' \subset G(\A_f)$ be neat compact open subgroups such that $K' \subset K$  and $K \cap Z_G(\Q) = K' \cap Z_G(\Q)$. Then, the map ${\rm pr}_{K',K}: Y_G(K')(\C) \twoheadrightarrow Y_G(K)(\C)$ is a covering map of smooth complex manifolds of degree $[K:K']$.
\end{lemma}
\begin{proof}
See for instance \cite[Lemma 2.7.1]{WaqarShahzetaelts}. 
\end{proof}

\begin{remark}\label{rmkonsingleCartDiag}
    The Cartesian diagram of Proposition \ref{Cartesian} is the composition of Cartesian diagrams involving each $T_{\eta_\frakp}$ separately as $\frakp$ varies in $\Upsilon_{F,p}$. For each $\frakp$, the corresponding Cartesian diagram is
    \begin{eqnarray*}  \xymatrix{ 
Y_H(uK_{\pmb{r}'}u^{-1} \cap H)  \ar[d]  \ar[rrr]^-{{\rm pr}' \circ u \circ \iota_{\pmb{r}'}}   & & &  Y_G(K_{\pmb{r}}^{(\frakp)} \cdot K_{\frakp,r_\frakp}') \ar[d]  \\
Y_H(u K_{\pmb{r}}u^{-1} \cap H)  \ar[rrr]^-{ u \circ \iota_{\pmb{r}}} & & &  Y_G(K_{\pmb{r}}),}  \end{eqnarray*}
where $\pmb{r}'$ is given by $r_\mathfrak{q}' = r_\mathfrak{q}$ for $\mathfrak{q} \ne \frakp$ and $r_\frakp' = r_\frakp +1$, ${\rm pr}'$ denotes the natural projection $Y_G(K_{\pmb{r}'}) \to Y_G(K_{\pmb{r}}^{(\frakp)} \cdot K_{\frakp,r_\frakp}')$, and vertical arrows are the natural degeneracy maps. 
\end{remark}

 Assume now that $\pmb{r}=\pmb{0}$. As the level subgroups $K_{\pmb{0}}, K_{\pmb{0}}', $ and $K_{\pmb{1}}$ agree with $K_{0}, K_{0}', $ and $K_{1}$, we adopt the latter notation. Note that in this case the bottom square of the diagram of Proposition \ref{Cartesian} is not Cartesian as the degree of $\pi_{0}^H:Y_H(uK_{1}u^{-1} \cap H) \to Y_H(K_{0} \cap H)$ is strictly smaller than $\pi_{0}^G: Y_G(K_{0}') \to Y_G(K_{0})$, as the next lemma shows.

\begin{lemma}\label{degreesofmapszerolevel}
 The map $\pi_0^H:Y_H(uK_{1}u^{-1} \cap H) \to Y_H(K_{0} \cap H)$ has degree $\prod_{\frakp \in \Upsilon_{F,p}} (q_\frakp^2-q_\frakp)  $ and $\pi_0^G : Y_G(K_{0}') \to Y_G(K_{0})$ has degree $\prod_{\frakp \in \Upsilon_{F,p}} q_\frakp^2$, where $q_\frakp$\index{$q_\frakp$} is the cardinality of the residue field $\mathbb{F}_\frakp$ of $\frakp$.
\end{lemma}

\begin{proof}
We first note that, as in the proof of Proposition \ref{Cartesian}, we have
\[ u K'_0 u^{-1} \cap H =  u K_{1} u^{-1} \cap H. \]
Thus, by Lemma \ref{verticaldeg}, the degree of $\pi_0^H$ equals to $[ K_0 \cap H :  u K'_0 u^{-1} \cap H]$. To calculate the latter, since $u$ is trivial outside $p$, it is enough to do the calculation prime by prime above $p$.

If $\frakp$ is a prime of $F$ above $p$ which is split in $E$, the component at $\frakp$ of $K_0 \cap H$ is given by  $\Gamma_{\GL_2,0}(\pi_\frakp) = \{ g={\matrix {a} { b} {c} {d}} \in \GL_2(\mathcal{O}_{\frakp})\,:\,\pi_\frakp \mid c \}$, while the one of $ u K'_0 u^{-1} \cap H$ is given by matrices $g={\matrix {a} { b} {c} {d}} \in \GL_2(\mathcal{O}_{\frakp})$ with $\pi_\frakp \mid c$, $\pi_\frakp \mid b$, and  $\pi_\frakp \mid a-d$ (\textit{cf}. proof of Lemma \ref{openorbitlemma}). The latter subgroup has index $(q_\frakp-1)q_\frakp$ in $\Gamma_{\GL_2,0}(\pi_\frakp)$.

Similarly, when $v$ is inert in $E$  the image of $u K_0' u^{-1} \cap H$ in $\GL_2(\mathbb{F}_{q_\frakp})$ is the center, implying that the $\frakp$-component of $u K_0' u^{-1} \cap H$ has index in $\Gamma_{\GL_2,0}(\pi_\frakp)$ equal to $q_\frakp^2-q_\frakp$. 
Thus we have that $\pi_0^H$ has degree $\prod_{\frakp \in \Upsilon_{F,p}} (q_\frakp^2-q_\frakp)$.

 The degree of $\pi_0^G$, by Lemma \ref{verticaldeg}, equals to $[K_0 : K_0']$. Since the $p$-component of $K_0'$ is $\Gamma_0^0(p,p)$, its $\frakp$-component has index in the Iwahori subgroup of $\GL_2(\mathcal{O}_{\frakp})^2$ for $\frakp$ split, resp.  $\GL_2(\mathcal{O}_{\mathfrak{P}})$ for $\frakp$ inert, equal to $q_\frakp^2$. Thus, we have that  $[K_0 : K_0'] = \prod_{\frakp \in \Upsilon_{F,p}} q_\frakp^2 $.
\end{proof}

Note the difference of the two degrees at each prime $\frakp$ above $p$ is $q_\frakp$, which is exactly the index of the Hecke correspondence $U_{\pi_\frakp}$ at $\frakp$ associated with $[\Gamma^0_H(\pi_\frakp){\matrix{\pi_\frakp}{}{}{1}}\Gamma^0_H(\pi_\frakp) ]$, with $\pi_\frakp$ the uniformizer in $\mathcal{O}_\frakp$. A partial explanation of this phenomenon is given by the following proposition. Before that, we need to introduce some notation.

Let $\mathfrak{S}$ be a subset of $\Upsilon_{F,p}$. Let $u_\mathfrak{S} \in G(\Z_p)$ denote the element $(\gamma_\frakp)_{\frakp \in \Upsilon_{F,p}}$ given by \[\gamma_\frakp = \begin{cases} u_{\mathfrak{p}} & \text{ if } \frakp \in \mathfrak{S}, \\ 1 & \text{ if } \frakp \not\in \mathfrak{S}, \end{cases} \]
with $u_{\mathfrak{p}}$ given in Definition \ref{def:representativeorbit}.
Moreover, define
\[K^0_H(\mathfrak{S}):= (K^{(p)} \cap H(\hat{\Z}^{(p)}))\cdot( u_\mathfrak{S}\Gamma^0_0(p,p)u_\mathfrak{S}^{-1} \cap H(\Z_p)),\] where recall that $\Gamma^0_0(p,p)$ is the $p$-component of $K_0'$. 
Finally, denote by $\tilde{\iota}_{\mathfrak{S}}: Y_H(K^0_H(\mathfrak{S})) \to Y_G(K_0')$ the closed immersion induced by $u_\mathfrak{S} \circ \iota$. A few remarks are in order.

\begin{remark}\label{somermksonthissigmatowers}\leavevmode \begin{enumerate}
    \item For every $\frakp \not \in \mathfrak{S}$, the $\frakp$-component of $K^0_H(\mathfrak{S})$ equals \[\Gamma^0_{\GL_2,0}(\frakp,\frakp)= \{ h \in \GL_2(\mathcal{O}_\frakp)\,:\, h \in T_{\GL_2} \, [\text{mod }\frakp] \}.\]
    \item If $\mathfrak{S}=\Upsilon_{F,p}$, we have $K^0_H(\mathfrak{S}) = u K_0'u^{-1} \cap H = u K_1 u^{-1} \cap H $.
    \item As the index of $\Gamma_{\GL_2,0}^0(\frakp,\frakp)$ in $\Gamma_{\GL_2,0}(\frakp)$ is $q_\frakp$, an immediate generalization of Lemma \ref{degreesofmapszerolevel} shows that the natural projection ${\rm pr}_\mathfrak{S}^H: Y_H(K^0_H(\mathfrak{S})) \to Y_H(K_0 \cap H)$ has degree equal to \[\prod_{\frakp \not \in \mathfrak{S}} q_\frakp \cdot \prod_{\frakp \in \mathfrak{S}} (q_\frakp^2-q_\frakp).\] 
    \end{enumerate}
\end{remark}

The following Lemma is needed in the proof of Proposition \ref{Cartesianwhenr=0} and makes explicit the open and closed $H$-orbits on the flag $\mathcal{F}$ given in Lemma \ref{openorbitlemma}.

\begin{lemma}\label{lem:open/closed_orbit}\leavevmode
\begin{enumerate}
    \item  Let $\mathfrak{p}  \in \Upsilon^{\rm s}_{F,p}$ and consider the map
 \[ t_\frakp:\GL_2(\mathbb{F}_\frakp) \to  \GL_2(\mathbb{F}_\frakp) \times  \GL_2(\mathbb{F}_\frakp),\, 
g= {\matrix{a}{b}{c}{d}}   \mapsto  (g,g)u_{\mathfrak{p}}=\left({\matrix{a}{a+b}{c}{c+d}},{\matrix{a}{b}{c}{d}} \right). \]
 Then the open orbit of $\GL_2(\mathbb{F}_\frakp)$ acting on $\mathbb{P}^1(\mathbb{F}_\frakp)\times \mathbb{P}^1(\mathbb{F}_\frakp)$ consists of points $(z_1,z_2)$ with $z_1 \neq z_2$ and is realized as $t_\frakp(\GL_2(\mathbb{F}_\frakp))\cdot ([0:1]^T,[0:1]^T)$, while the closed one consists of points $(z,z)$ and equals $\iota_\mathfrak{p}(\GL_2(\mathbb{F}_\frakp))\cdot ([0:1]^T,[0:1]^T)$.
 \item Let $\mathfrak{p}  \in \Upsilon^{\rm i}_{F,p}$ and let $\mathfrak{P}$ be the only prime in $E$ above $\frakp$. Consider the map
 \[
 t_\frakp:
\GL_2(\mathbb{F}_\frakp) \to \GL_2(\mathbb{F}_{\mathfrak{P}}),\, 
g= {\matrix{a}{b}{c}{d}} \mapsto g u_{\mathfrak{p}}= {\matrix{a}{a\delta_\mathfrak{p} +b}{c}{c\delta_\mathfrak{p} + d}}.
 \]
 Then the open orbit of $\GL_2(\mathbb{F}_\frakp)$ acting on $\mathbb{P}^1(\mathbb{F}_{\mathfrak{P}})$ consists of points $z \in \mathbb{P}^1(\mathbb{F}_{\mathfrak{P}}) \backslash \mathbb{P}^1(\mathbb{F}_{{\frakp}})$ and is realized as $t_\frakp(\GL_2(\mathbb{F}_\frakp))\cdot [0:1]^T$ while the closed one consists of $z \in \mathbb{P}^1(\mathbb{F}_{{\frakp}})$ and equals $\iota_\mathfrak{p}(\GL_2(\mathbb{F}_\frakp))\cdot [0:1]^T$.
 \end{enumerate}
\end{lemma}
\begin{proof}
We start with the split case.
Recall that $\mathbb{P}^1=\GL_2/\overline{B}_{\GL_2}$. The action of $\GL_2$ is given by left multiplication:
\[
{\matrix{a}{b}{c}{d}} [x:y]^T=[ax+by:cx+dy]^T.
\]
So we have that the open orbit is 
\[
\left({\matrix{a}{a+b}{c}{c+d}},{\matrix{a}{b}{c}{d}} \right)([0:1]^T,[0:1]^T)=([a+b:c+d]^T,[b:d]^T)
\]
and the closed one is
\[
\left( {\matrix{a}{b}{c}{d}},{\matrix{a}{b}{c}{d}} \right)([0:1]^T,[0:1]^T)=([b:d]^T,[b:d]^T)
\]
In the inert case we have the open orbit
\[
{\matrix{a}{a\delta_\mathfrak{p} +b}{c}{c\delta_\mathfrak{p} + d}}[0:1]^T=[a\delta_\mathfrak{p} +b: c\delta_\mathfrak{p} + d]^T.
\]
One sees that the $\delta_\mathfrak{p}$-coefficient of the ratio $a\delta_\mathfrak{p} +b/ (c\delta_\mathfrak{p} + d)$ is $0$ if and only if $ad-bc=0$, which is never the case.
The closed orbit is instead
\[
{\matrix{a}{b}{c}{d}}[0:1]^T=[b: d]^T.
\]
\end{proof}

We can finally state the equivalent of Proposition \ref{Cartesian} for when $r=0$.

\begin{proposition}\label{Cartesianwhenr=0}
Assume that $K^{(p)}$ is stable under $\tau \in {\rm Gal}(E/F)$. The commutative diagram 
 \begin{eqnarray*}  \xymatrix{  
\bigsqcup_{\mathfrak{S}} Y_H(K^0_H(\mathfrak{S}))  \ar[d]_{({\rm pr}_\mathfrak{S}^H)_\mathfrak{S}}  \ar[rrr]^-{(\tilde{\iota}_\mathfrak{S})_\mathfrak{S}}   & & &  Y_G(K_0') \ar[d]^{\pi_0^G} \\
Y_H( K_0 \cap H)  \ar[rrr]^-{ \iota_{0}} & & &  Y_G(K_0)    }  \end{eqnarray*}
is Cartesian. Moreover, it can be decomposed as a series of Cartesian diagrams indexed by the primes in $\Upsilon_{F,p}$: if $\mathfrak{S}$ is a (possibly empty) subset of $ \Upsilon_{F,p}$ and $\mathfrak{p} \not \in \mathfrak{S}$, then, denoting $\mathfrak{S}' : = \mathfrak{S} \cup \{ \mathfrak{p}\}$,  \begin{eqnarray*}  \xymatrix{  
 Y_H(u_\mathfrak{S}K_0(\mathfrak{S}')u_\mathfrak{S}^{-1} \cap H) \sqcup Y_H(u_{\mathfrak{S}'}K_0(\mathfrak{S}')u_{\mathfrak{S}'}^{-1} \cap H)  \ar[d]_{({\rm pr}_{\mathfrak{S},\emptyset}^H, {\rm pr}_{\mathfrak{S},\mathfrak{p}}^H )}  \ar[rrr]^-{(\tilde{\iota}_\mathfrak{S}, u_\mathfrak{p} \circ \tilde{\iota}_\mathfrak{S} )}   & & &  Y_G(K_0(\mathfrak{S}')) \ar[d]^{\pi_{\mathfrak{p}}^G} \\
Y_H( u_\mathfrak{S}K_0(\mathfrak{S})u_\mathfrak{S}^{-1} \cap H)  \ar[rrr]^-{ \tilde{\iota}_\mathfrak{S}} & & &  Y_G(K_0(\mathfrak{S})),    }  \end{eqnarray*}
is Cartesian, where $K_0(\mathfrak{S})$ is the level subgroup introduced in \eqref{eqdefintermediatelevelsigma}. 
\end{proposition}
\begin{proof}
We mimic the proof of Proposition \ref{Cartesian}. By Remark \ref{somermksonthissigmatowers}(3), the degree of the map $({\rm pr}_\mathfrak{S}^H)_\mathfrak{S}$ equals to  \[\sum_{\mathfrak{S}}  [ \prod_{\frakp \not \in \mathfrak{S}} q_\frakp \cdot \prod_{\frakp \in \mathfrak{S}} (q_\frakp^2-q_\frakp)] = \prod_{\frakp \in \Upsilon_{F,p}}  q_\frakp^2,   \]
where the last equality follows from the binomial theorem. This implies that the degree of $({\rm pr}_\mathfrak{S}^H)_\mathfrak{S}$ equals to the degree of $\pi_0^G$. To prove the result we are left to show that \[\xymatrix{  
\bigsqcup_{\mathfrak{S}} Y_H(K^0_H(\mathfrak{S}))   \ar[rrr]^-{(\tilde{\iota}_\mathfrak{S})_\mathfrak{S}}   & & &  Y_G(K_0') }  \]
is a closed immersion, where recall that $\tilde{\iota}_\mathfrak{S} = u_\mathfrak{S} \circ \iota$. This follows from the fact that each $\tilde{\iota}_\mathfrak{S}$ is a closed immersion and that the images of any two  sub-varieties do not intersect in $Y_G(K_0')$.
The first fact is clear (\textit{cf}. Lemma \ref{closedimm}), so we are left to show that the images don't intersect. This follows roughly from the fact that the number of orbits of $H$ in the flag $\mathcal{F}=G/\overline{B}_G$ is in bijection with the subsets $\mathfrak{S}$ of $\Upsilon_{F,p}$, with the open orbit corresponding to the choice of $\mathfrak{S}=\Upsilon_{F,p}$. 
Precisely, note that the fiber of the natural degeneracy map $\pi_{\mathfrak{p}}^G$ on the right can be identified with
\[
\Gamma_0(p)/\Gamma_0^0(p,p) \cong B_G(\mathbb{F}_p) / T_G(\mathbb{F}_p) \cong N_G(\mathbb{F}_p) \cong \prod_{\frakp \in \Upsilon^{\rm s}_{F,p}} \left(N_{\GL_2}(\mathbb{F}_\frakp)\times N_{\GL_2}(\mathbb{F}_\frakp) \right)\times \prod_{\frakp \in \Upsilon^{\rm i}_{F,p}}  N_{\GL_2}(\mathbb{F}_{\mathfrak{P}}),
\]
for $\mathfrak{P}$ the only prime in $E$ above $\frakp$. Moreover, as $N_G$ can be realized as an open affine of $\mathcal{F}$, we can view the right hand side explicitly as the subspace of $\prod_{\frakp \in \Upsilon^{\rm s}_{F,p}} \left(\mathbb{P}^1(\mathbb{F}_\frakp)\times \mathbb{P}^1(\mathbb{F}_\frakp) \right)\times \prod_{\frakp \in \Upsilon^{\rm i}_{F,p}}  \mathbb{P}^1(\mathbb{F}_{\mathfrak{P}})$ given by \[\prod_{\frakp \in \Upsilon^{\rm s}_{F,p}} \{ \left ([a:1],[a':1] \right ),\, a,a' \in \mathbb{F}_\frakp \} \times \prod_{\frakp \in \Upsilon^{\rm i}_{F,p}}  \{ [b:1],\, b \in \mathbb{F}_\mathfrak{P} \}\, (=: \prod_{\frakp \in \Upsilon^{\rm s}_{F,p}} N_\frakp^{\rm s} \times \prod_{\frakp \in \Upsilon^{\rm i}_{F,p}}  N_\frakp^{\rm i}). \]

 We are left to show that the $\frakp$-component of any element in the image of $Y_H(K_0^H(\mathfrak{S}))$ is characterized by landing into the closed, resp. open, orbit for $\GL_2(\mathbb{F}_\frakp)$ when $\frakp \not \in \mathfrak{S}$, resp. $\frakp \in \mathfrak{S}$. This follows from our definition of $u_\mathfrak{S}$ and the explicit calculation of Lemma \ref{lem:open/closed_orbit}: indeed, when $\frakp \in \mathfrak{S}$, the image of the $\frakp$-component of $\Gamma_{H,0}(p)$  into the $\frakp$-component of $\Gamma_0(p)/\Gamma_0^0(p,p)$ equals to $t_\frakp(B_{\GL_2}(\mathbb{F}_\frakp))$, which realizes the intersection of the open orbit with $N_\frakp^{\rm s}$, resp. $N_\frakp^{\rm i}$, if $\frakp$ is split, resp. inert. When $\frakp \not \in \mathfrak{S}$, it equals to  $\iota_\mathfrak{p}(B_{\GL_2}(\mathbb{F}_\frakp))$, which gives the intersection of the closed orbit with $N_\frakp^{\rm s}$, resp. $N_\frakp^{\rm i}$, if $\frakp$ is split, resp. inert.
Therefore, if $\mathfrak{S},\mathfrak{S}'$ are two distinct subsets, then the images of $Y_H(K^0_H(\mathfrak{S}))$ and $Y_H(K^0_H(\mathfrak{S}'))$ do not intersect, which proves our first claim. The second part of the statement follows in the same way. 
\end{proof}

The interested reader might wonder what happens if we build our tower of level subgroups, starting from spherical level at $p$ instead of $\Gamma_0(p)$, as done in \cite{LoefflerSphericalvarieties}. The resulting Cartesian diagrams, which are almost identical to those given above, can be used to relate the cycles with spherical level at $p$ with the twisted cycles in the $p$-tower of level subgroups. While the diagram of Proposition \ref{Cartesian} translates verbatim, a few small modifications are required in the $r=0$ case. We conclude this section by illustrating it. Let us denote \begin{itemize}
    \item $K_{{\rm un},0} := K^{(p)}\cdot G(\Z_p)$, $K_{{\rm un},0}' := K^{(p)}\cdot \Gamma^0(p)$;
    \item for any subset $\mathfrak{S}$ of $\Upsilon_{F,p}$, 
\[K^0_{{\rm un},H}(\mathfrak{S}):= (K^{(p)} \cap H(\hat{\Z}^{(p)}))\cdot( u_\mathfrak{S}\Gamma^0(p)u_\mathfrak{S}^{-1} \cap H(\Z_p)).\]
\end{itemize}

\begin{lemma}\label{degreesofmapszerosphericallevel}
 The natural degeneracy map ${\rm pr}_{{\rm un},\mathfrak{S}}^H:Y_H(K^0_{{\rm un},H}(\mathfrak{S})) \to Y_H(K_{un,0} \cap H)$ has degree \[\prod_{\frakp \not \in \mathfrak{S}} (q_\frakp + 1) \cdot \prod_{\frakp \in \mathfrak{S} \cap \Upsilon^{\rm s}_{F,p}} (q_\frakp^2+q_\frakp) \cdot \prod_{\frakp \in \mathfrak{S} \cap \Upsilon^{\rm i}_{F,p}} q_\frakp(q_\frakp-1)\] and the natural projection $\pi_{{\rm un},0}^G:Y_G(K_{{\rm un},0}') \to Y_G(K_{{\rm un},0})$ has degree $\prod_{\frakp \in \Upsilon^{\rm s}_{F,p}} (q_\frakp+1)^2  \times \prod_{\frakp \in \Upsilon^{\rm i}_{F,p}} (q^2_\frakp+1)$
\end{lemma}

\begin{proof}
By Lemma \ref{verticaldeg}, the degree of ${\rm pr}_{{\rm un},\mathfrak{S}}^H$ equals to $[ K_{{\rm un},0} \cap H :  K^0_{{\rm un},H}(\mathfrak{S})]$. To calculate the latter, since $u_\mathfrak{S}$ is trivial outside $p$, it is enough to do the calculation prime by prime above $p$. 
\begin{itemize}
    \item If $\frakp$ is a prime of $F$ above $p$ not in $\mathfrak{S}$, the $\frakp$-component of $ K^0_{{\rm un},H}(\mathfrak{S})$ equals to $\Gamma^0_{\GL_2}(\frakp)$, which has index $q_\frakp + 1$ in $\GL_2(\mathcal{O}_\frakp)$.
    \item If $\frakp \in \mathfrak{S}$ is split in $E$, the $\frakp$-component of $ K^0_{{\rm un},H}(\mathfrak{S})$ is given by matrices  $g={\matrix {a} { b} {c} {d}} \in \GL_2(\mathcal{O}_{\frakp})$ with $\pi_\frakp \mid b$ and $\pi_\frakp \mid a-d-c$ (\textit{cf}. proof of Lemma \ref{openorbitlemma}). This subgroup has index $(q_\frakp+1)q_\frakp$ in $\GL_2(\mathcal{O}_{\frakp})$.
    \item If $\frakp \in \mathfrak{S}$ is inert in $E$, then the image of the $\frakp$-component of $ K^0_{{\rm un},H}(\mathfrak{S})$ in $\GL_2(\mathbb{F}_{q_\frakp})$ is the non-split torus $\mathbb{F}_{q^2_\frakp}^\times$ which has index $q_\frakp^2-q_\frakp$ in $\GL_2(\mathbb{F}_{q_\frakp})$.
\end{itemize}

The degree of $\pi_0^G$, by Lemma \ref{verticaldeg}, equals to $[K_{{\rm un},0} : K_{{\rm un},0}']$. Since the $p$-component of a matrix in $K_{{\rm un},0}'$ is characterized by having lower left entry divisible by $p$ and the index of the Iwahori subgroup of $\GL_2(\mathcal{O}_{\frakp})^2$ for $\frakp$ split, resp. $\GL_2(\mathcal{O}_{\mathfrak{P}})$ for $\frakp$ inert, is $(q_\frakp+1)^2$, resp. $q_\frakp^2+1$, we have that  $[K_{{\rm un},0} : K_{{\rm un},0}'] = \prod_{\frakp \in \Upsilon^{\rm s}_{F,p}} (q_\frakp+1)^2  \times \prod_{\frakp \in \Upsilon^{\rm i}_{F,p}} (q^2_\frakp+1)$.
\end{proof}

\begin{remark}
 If $\mathfrak{S} = \Upsilon_{F,p}$, the difference of the two degrees at each prime $\frakp$ above $p$ is $q_\frakp+1$, which is exactly the index of the correspondence for $T_{\pi_\frakp}$, which is the Hecke operator at $v$ given by $[H(\mathcal{O}_\frakp){\matrix{\pi_\frakp}{}{}{1}}H(\mathcal{O}_\frakp) ]$, with $\pi_\frakp$ the uniformizer in $\mathcal{O}_\frakp$. This is coherent with what we found above in Lemma \ref{degreesofmapszerolevel}, where the difference was instead given by the degree of the correspondence $U_{\pi_\frakp}$.
\end{remark} 

As before, denote by $\tilde{\iota}_{{\rm un},\mathfrak{S}}: Y_H(K^0_{{\rm un},H}(\mathfrak{S})) \to Y_G(K_{{\rm un}, 0}')$ the closed immersion induced by $u_\mathfrak{S} \circ \iota$.

\begin{proposition}\label{Cartesianwhenr=0forspherical}
Assume that $K^{(p)}$ is stable under $\tau \in {\rm Gal}(E/F)$. The commutative diagram 
 \begin{eqnarray*}  \xymatrix{  
\bigsqcup_{\mathfrak{S}}  Y_H(K^0_{{\rm un},H}(\mathfrak{S}))  \ar[d]_{({\rm pr}_{{\rm un},\mathfrak{S}}^H)_\mathfrak{S}}  \ar[rrr]^-{(\tilde{\iota}_{{\rm un},\mathfrak{S}})_\mathfrak{S}}   & & &  Y_G(K_{{\rm un}, 0}') \ar[d]^{\pi_{{\rm un},0}^G} \\
Y_H( K_{{\rm un}, 0} \cap H)  \ar[rrr]^-{ \iota_{{\rm un},0}} & & &  Y_G(K_{{\rm un}, 0})    }  \end{eqnarray*}
is Cartesian. 
\end{proposition}
\begin{proof}
We closely follow the proof of Proposition \ref{Cartesianwhenr=0}. By Lemma \ref{degreesofmapszerosphericallevel}, using the binomial theorem we see that the degree of the map $({\rm pr}_{{\rm un},\mathfrak{S}}^H)_\mathfrak{S}$ equals to  \[\sum_{\mathfrak{S}}   [\prod_{\frakp \not \in \mathfrak{S}} (q_\frakp + 1) \cdot \prod_{\frakp \in \mathfrak{S} \cap \Upsilon^{\rm s}_{F,p}} (q_\frakp^2+q_\frakp) \cdot \prod_{\frakp \in \mathfrak{S} \cap \Upsilon^{\rm i}_{F,p}} q_\frakp(q_\frakp-1)  ] = \prod_{\frakp \in \Upsilon^{\rm s}_{F,p}} (q_\frakp+1)^2  \times \prod_{\frakp \in \Upsilon^{\rm i}_{F,p}} (q^2_\frakp+1).   \]
This implies that the degree of the two vertical maps coincide. We are left to show that the upper horizontal map is a closed immersion. Since each $\tilde{\iota}_{{\rm un},\mathfrak{S}}$ is, we are left to show that the images of any the two immersions do not intersect in $Y_G(K_{{\rm un},0}')$.

Note that the fiber of the right vertical map can be identified with
\[
G(\Z_p)/\Gamma^0(p) \cong G(\mathbb{F}_p) / \overline{B}_G(\mathbb{F}_p) \cong \prod_{\frakp \in \Upsilon^{\rm s}_{F,p}} \left( \mathbb{P}^1(\mathbb{F}_\frakp)\times ( \mathbb{P}^1(\mathbb{F}_\frakp) \right)\times \prod_{\frakp \in \Upsilon^{\rm i}_{F,p}}   \mathbb{P}^1(\mathbb{F}_{\mathfrak{P}}),
\]
for $\mathfrak{P}$ the only prime in $E$ above $\frakp$. 
 As in Proposition \ref{Cartesianwhenr=0}, we are thus left to show that the $\frakp$-component of any element in the image of $Y_H(K^0_{{\rm un},H}(\mathfrak{S}))$ is characterized by landing into the closed, resp. open, orbit for $\GL_2(\mathbb{F}_\frakp)$ when $\frakp \not \in \mathfrak{S}$, resp. $\frakp \in \mathfrak{S}$. This follows again from Lemma \ref{lem:open/closed_orbit}: when $\frakp \in \mathfrak{S}$, the image of the $\frakp$-component of $H(\Z_p)$ into the $\frakp$-component of $G(\Z_p)/\Gamma^0(p)$ equals to $t_\frakp(\GL_2(\mathbb{F}_\frakp))$, which realizes the  open orbit. When $\frakp \not \in \mathfrak{S}$, it equals to  $\iota_\mathfrak{p}({\GL_2}(\mathbb{F}_\frakp))$, which is the closed orbit.
Therefore, if $\mathfrak{S},\mathfrak{S}'$ are two distinct subsets, then the images of $\tilde{\iota}_{{\rm un},\mathfrak{S}}$  and $\tilde{\iota}_{{\rm un},\mathfrak{S}'}$ are disjoint.
\end{proof}

\subsection{Moment maps and specialisation}

For $r \geq 1$, let $V_r=K^{(p)} \cdot V_{p,r}$ denote the level subgroup of Definition \ref{def:levelgroups}. Consider a big enough set of primes $S$  containing $p$ and every prime $\ell$ such that $K^{(p)}_\ell$ is not hyperspecial. By abuse of notation, we denote by $Y_G(V_r)$ the integral model of the Shimura variety for $G $ over $\Z[1/S]$. Recall we have fixed a prime $\mathfrak{m}$ of the coefficient field $L$ above $p$. Consider, for any integer $i$, the groups
   \[ H^i_{\rm Iw}(Y_G(V_\infty), \mathcal{O}_{\mathfrak{m}}) := \varprojlim_{r \geq 1} H^i(Y_G(V_r),\mathcal{O}_{\mathfrak{m}}), \] where the inverse limit is taken  with respect to the trace maps associated with the natural degeneracy maps ${\rm pr}_{r+1,r}:Y_G(V_{r+1}) \to Y_G(V_r)$.  
   \begin{remark}\label{remarkonpassingbetweentowers}
   Recall that we have defined $[\eta]_{r+1,r}:Y_G(K_{r+1}) \to Y_G(K_r)$ in Definition \ref{Def:traceetar1r}. Moreover, we have a normalized trace map $\eta^r_\ast$ of $\eta^r: Y_G(K_{r}) \to Y_G(V_r)$ defined as in Lemma \ref{Lem:traceeta} as the composition of $\eta_\flat^r$ with the pushforward of $\eta^r$.  Note that we have an equality \[ \eta^r_\ast \circ {\rm Tr}_{[\eta]_{r+1,r}} = {\rm pr}_{r+1,r, \ast} \circ \eta^{r+1}_\ast,    \]
   induced by the commutative diagram \[ \xymatrix{Y_G(K_{r+1}) \ar[rr]^-{\eta^{r+1}} \ar[d]_-{[\eta]_{r+1,r}} & &  Y_G(V_{r+1}) \ar[d]^-{{\rm pr}_{r+1,r}} \\ Y_G(K_r)\ar[rr]^-{\eta^{r}} & & Y_G(V_r).  }   \]     
   \end{remark}

We let $\overline{\mathcal{E}}$ be the closure of the totally positive units in $T_G(\Z_p)$, and ${\mathcal{E}}(p^r)$ the image  of the totally positive units in $T_G(\Z_p/p^r)$. We write \begin{align}\label{eqdefIwalg}
\Lambda_{G/H}:=\mathcal{O}_{\mathfrak{m}}\llbracket T_G(\Z_p)/\overline{\mathcal{E}}L_G(\Z_p)\rrbracket \simeq \varprojlim_r \Lambda_{G/H,r}\, \left (:=\mathcal{O}_{\mathfrak{m}}/\mathfrak{m}^r[T_G(\Z_p/p^r)/{\mathcal{E}}(p^r) L_G(\Z_p/p^r)] \right),\end{align} where $L_G = \{ {\matrix {\alpha} { } {} {\tau(\alpha)}}\,:\, \alpha \in {\rm Res}_{\mathcal{O}_E/\Z} \mathbf{G}_{\rm m} \}$. \\

The Iwasawa cohomology groups defined above have naturally a continuous action of the quotient $T_G(\Z_p)/ \overline{\mathcal{E}} L_G(\Z_p)$, which makes them into Iwasawa modules over the Iwasawa algebra $\Lambda_{G/H}$. 

Notice that the kernel of the map $T_H(\Z_p) \rightarrow T_G(\Z_p)/L_G(\Z_p)$ is exactly the center of $H(\Z_p)$ and thus we have an action of the Iwasawa algebra $\Lambda_{G/\mathcal{O}_\mathfrak{m}\llbracket T_H(\Z_p)/Z_H(\Z_p)\rrbracket}$ on these cohomology groups. The  Hecke eigensystems that appear in $H^i_{\rm Iw}(Y_G(V_\infty), \mathcal{O}_{\mathfrak{m}}) $ correspond to automorphic representations with trivial central character, ({\it i.e.} the Nebentypus is trivial and $m=0$).

We need to construct specialization maps from $H^i_{\rm Iw}(Y_G(V_\infty), \mathcal{O}_{\mathfrak{m}})$ to $H^i(Y_G(V_r),\mathscr{V}^{\underline{k},0}_{\mathcal{O}_{\mathfrak{m}}})$, which will be used to relate the universal Hirzebruch--Zagier cycle with (the ordinary projection of) the cycle studied by Getz--Goresky.
Following \cite{LRZ}, the key input is a vector in $\mathscr{V}^{\underline{k},0}_{\mathcal{O}_{\mathfrak{m}}}$ which is invariant under $V_r$ modulo $\mathfrak{m}^r$, and that has the same ordinary projection as the $u$-translate of the $H(\mathcal{O}_{\mathfrak{m}})$-invariant vector $\Delta^{\underline{k}}$. This motivates the switch from the tower of level subgroups $(K_r)_r$ to $(V_r)_r$ instead.

Let $v_{\underline{k}}^{\rm hw}$ be the highest weight vector for $V^{\underline{k},0}$. As $V_{\mathcal{O}_\mathfrak{m}}^{\underline{k},0}$ is admissible, its intersection with the highest weight subspace of $V^{\underline{k},0}$ is $ \mathcal{O}_\mathfrak{m} \cdot v_{\underline{k}}^{\rm hw}$. For any $r \geq 1$, denote by $v_{\underline{k},r}^{\rm hw}$ the reduction modulo $\mathfrak{m}^r$.

\begin{lemma} \label{lemmahwvector} Let $\underline{k} \in \Z^{\Sigma_E}$ be such that $k_\sigma=k_{\sigma'}$ for every $\sigma \in \Phi$. For each $r \geq 1$, $v_{\underline{k},r}^{\rm hw} \in H^0(Y_G(V_r), \mathscr{V}^{\underline{k},0}_{\mathcal{O}_\mathfrak{m}/\mathfrak{m}^r})$. Moreover, we have that \[v_{\underline{k},r+1}^{\rm hw} \, [{\rm mod }\,\, p^r] = ({\rm pr}_{r+1,r})^\ast v_{\underline{k},r}^{\rm hw} \in H^0(Y_G(V_{r+1}), \mathscr{V}^{\underline{k},0}_{\mathcal{O}_\mathfrak{m}/\mathfrak{m}^r}).\]
\end{lemma}

\begin{proof}
By definition, the image of $V_r$ in $G(\Z/p^r\Z)$ sits inside $L_G(\Z/p^r\Z) \cdot N_G(\Z/p^r\Z)$. Being $v_{\underline{k},r}^{\rm hw}$ the highest weight vector, $N_G(\Z/p^r\Z)$ acts trivially on it. Moreover, $L_G(\Z/p^r\Z)$ acts on the highest weight vector $v_{\underline{k},r}^{\rm hw}$ of $\mathscr{V}^{\underline{k},m}_{\mathcal{O}_\mathfrak{m}/\mathfrak{m}^r}$ exactly via $\mathrm{N}^m_{E/\Q}$ [mod $\mathfrak{m}^r$]. Thus, it defines an element of $H^0(Y_G(V_r), \mathscr{V}^{\underline{k},0}_{\mathcal{O}_\mathfrak{m}/\mathfrak{m}^r})$.

For the second statement, this simply follows from reducing modulo $\mathfrak{m}^r$ the vector $v_{\underline{k},r+1}^{\rm hw}$, which is naturally $({\rm pr}_{r+1,r})^\ast v_{\underline{k},r}^{\rm hw}$, where $({\rm pr}_{r+1,r})^\ast$ is just the restriction of the $V_r$-representation $V^{\underline{k},0}_{\mathcal{O}_\mathfrak{m}/\mathfrak{m}^r}$ to $V_{r+1}$.
\end{proof}

Thanks to the lemma, we can make the following (\textit{cf}. \cite[Definition 5.1.1]{LRZ}).

\begin{definition}\label{def:mom} Let $\underline{k} \in \Z^{\Sigma_E}$ be such that $k_\sigma=k_{\sigma'}$ for every $\sigma \in \Phi$. For $\pmb{r} = (r_\frakp)_\frakp \in \Z_{\geq 1}^{|\Upsilon_{F,p}|}$, we define the moment map \[  \mathrm{mom}^{\underline{k}}_{\pmb{r}}: H^i_{\rm Iw}(Y_G(V_\infty), \mathcal{O}_{\mathfrak{m}}) \rightarrow H^i(Y_G(V_{\pmb{r}}),\mathscr{V}^{\underline{k},0}_{\mathcal{O}_{\mathfrak{m}}}) \]\index{$\mathrm{mom}^{\underline{k}}_{\pmb{r}}$} in the following way:
 \begin{align*} 
  \mathrm{mom}^{\underline{k}}_{\pmb{r}}: H^i_{\rm Iw}(Y_G(V_\infty), \mathcal{O}_{\mathfrak{m}}) &\xrightarrow{} \varprojlim_{s } H^i(Y_G(V_s),\mathcal{O}_\mathfrak{m}/{\mathfrak{m}^s}) \\
  &\to \varprojlim_{s} H^i(Y_G(V_s),\mathscr{V}^{\underline{k},0}_{\mathcal{O}_\mathfrak{m}/\mathfrak{m}^s}) \\
  &\to \varprojlim_{s} H^i(Y_G(V_{\pmb{r}}),\mathscr{V}^{\underline{k},0}_{\mathcal{O}_\mathfrak{m}/\mathfrak{m}^s}) \\
   &\xrightarrow{\sim} H^i(Y_G(V_{\pmb{r}}),\mathscr{V}^{\underline{k},0}_{\mathcal{O}_{\mathfrak{m}}})
 \end{align*}
 
\noindent where the first map is the reduction modulo $\mathfrak{m}^s$  in the projective system, the second map is cup product  $\bullet \cup v_{\underline{k},s}^{\rm hw}$ $\in H^0(Y_G(V_s), \mathscr{V}^{\underline{k},0}_{\mathcal{O}_\mathfrak{m}/\mathfrak{m}^s})$, the third one is the projection to level $\pmb{r}$, and the last one is an isomorphism.
\end{definition}
 \noindent When $\pmb{r}=(r, \dots, r)$, we denote the moment map by $\mathrm{mom}^{\underline{k}}_{r}$\index{$\mathrm{mom}^{\underline{k}}_{r}$}.

\subsection{Control theorem for Iwasawa cohomology}

We show that the ordinary part of our Iwasawa cohomology is of finite type over the Iwasawa algebra $\Lambda_{G/H}$, and a control theorem. We note that similar results have been obtained in \cite{ShethControl}.

\begin{lemma}\label{lem:Tplowerlevel}
For $r \geq 1$, we have a commutative diagram \[ \xymatrix{ H^i(Y_G(V_{r+1}),\mathscr{V}^{\underline{k},0}_{\mathcal{O}_\mathfrak{m}}) \ar[rr]^{T_\eta} \ar[d]_{{\rm pr}_{r+1,r, \ast}} & & H^i(Y_G(V_{r+1}),\mathscr{V}^{\underline{k},0}_{\mathcal{O}_\mathfrak{m}}) \ar[d]^{{\rm pr}_{r+1,r, \ast}}  \\ \! H^i(Y_G(V_r),\mathscr{V}^{\underline{k},0}_{\mathcal{O}_\mathfrak{m}}) \ar[rr]^{T_\eta} \ar[urr] & & H^i(Y_G(V_r),\mathscr{V}^{\underline{k},0}_{\mathcal{O}_\mathfrak{m}}),  }\]
where the diagonal map is the natural restriction map induced by the pullback of the projection $Y_G(V_{r+1}) \xrightarrow{{\rm pr}} Y_G(\eta^{-r-1} K_{r} \eta^{r+1}) \xrightarrow{\eta^{-1}}  Y_G(V_{r})$, which is normalized integrally as in Lemma \ref{Lem:traceeta}. 
\end{lemma}
\begin{proof}
This lemma is the well known fact (already known to Shimura) that $T_p$ lowers the level at $p$. One follows essentially the proof of \cite[(8.8)]{HidaAnnals88}: one calculates the explicit coset decomposition of $T_\eta$ and the explicit description of the trace of ${\rm pr}_{r+1,r}$ to check they are the same.
\end{proof}

\begin{lemma}\label{lem:ordinaryIwasawaCoho}
 There exists an ordinary projector $e^{{\rm n.ord}}$ on $H^i(Y_G(U), \mathscr{V}^{\underline{k},0}_{\mathcal{O}_\mathfrak{m}})$
such that

\[
e^{\rm n.ord} H^i_{\rm Iw}(Y_G(V_\infty), \mathcal{O}_{\mathfrak{m}})=\varprojlim_r e^{\rm n.ord}H^i(Y_G(V_r),\mathcal{O}_{\mathfrak{m}}).
\]
\end{lemma}\begin{proof}
The existence of the idempotent on $H^i(Y_G(V_r), \mathscr{V}^{\underline{k},0}_{\mathcal{O}_\mathfrak{m}})$ is well-known, see for example \cite[\S 1.11]{Hida1993Duke} that is, one considers the direct limit of $T_{\eta}^{n!}$, for $n$ going to infinity. The key fact is that $T_{\eta}$ is optimally normalized to preserve the integral structure, see Remark \ref{rem:normalizationTp}.

The idempotent on $H^i_{\rm Iw}(Y_G(V_\infty), \mathcal{O}_{\mathfrak{m}})$ is defined as the inverse limit over $r$ of the idempotent at each level. 
\end{proof} 

\noindent Let $K_0(p^r)$ be the congruence subgroup given by $K^{(p)} \cdot \Gamma_0(p^r)$. Observe that, by Remark \ref{rmk:factsabouttower}, $K_0(p^2) = V_0'$.  Consider
the sheaf on $Y_G(K_0(p^2))$ given by $D_r := s_{r,\ast}\mathcal{O}_{\mathfrak{m}}$, where $s_r: Y_G(V_r) \to Y_G(K_0(p^2))$ denotes the natural projection map. Observe that $D_r$ splits in a unipotent part $\Lambda_{\overline{N}_{G,r}}$ and a toric part $\Lambda_{G/H,r}$; note that this decomposition is $K_0(p^{r+1})$-equivariant but not $K_0(p^2)$-equivariant. Here, $\overline{N}_{G,r}$ denotes the quotient of $K_0(p^{2}) \cap \overline{N}_G(\zp)$ with $K_0(p^{r+1}) \cap \overline{N}_G(\zp)$.  

We then define \[ H^{i}(Y_G(K_0(p^2)), D) := \varprojlim_r  H^i(Y_G(K_0(p^2)),D_{r}).\]
We have an isomorphism $D \simeq \mathcal{O}_\mathfrak{m}\llbracket \overline{N}_G\rrbracket \hat{\otimes}_{\mathcal{O}_\mathfrak{m}}   \Lambda_{G/H}$, where  $\mathcal{O}_\mathfrak{m}\llbracket \overline{N}_G\rrbracket = \varprojlim_r \Lambda_{\overline{N}_{G,r}}$.

\begin{proposition}\label{prop:IwasawaOrdinarycoho}
 We have an isomorphism of $\Lambda_{G/H}$-modules
 \[
 e^{\rm n.ord} H^i_{\rm Iw}(Y_G(V_\infty), \mathcal{O}_{\mathfrak{m}})  \cong e^{\rm n.ord} H^{i}(Y_G(K_0(p^2)),D) =: e^{\rm n.ord} H^{i}(Y_G(K_0(p^2)),\Lambda_{G/H}).
 \]
\end{proposition}
\begin{remark}
    The reason why we call the right hand side $e^{\rm n.ord} H^{i}(Y_G(K_0(p^2)),\Lambda_{G/H})$ is because the ordinary projector $e^{\rm n.ord}$ kills the contribution of the lower unipotent part $\mathcal{O}_\mathfrak{m}\llbracket \overline{N}_G\rrbracket$, see \cite[(2.5a), p.280 ]{Hida1993Duke}.
\end{remark}
\begin{proof}
Recall we  have $H^i_{\rm Iw}(Y_G(V_\infty), \mathcal{O}_{\mathfrak{m}})=\varprojlim_{r \geq 1} H^i(Y_G(V_r),\mathcal{O}_{\mathfrak{m}})$ where the transition maps are essentially given by $T_{\eta}$ by Lemma \ref{lem:Tplowerlevel}. Moreover,

\[
e^{\rm n.ord}H^i(Y_G(V_r),\mathcal{O}_{\mathfrak{m}})=\varprojlim_n e^{\rm n.ord}H^i(Y_G(V_r),\mathcal{O}_{\mathfrak{m}}/\mathfrak{m}^n).\]

Note that, by Remark \ref{rmk:factsabouttower}, $ K_0(p^{r+1})$ contains $V_r$ as a normal subgroup and the quotient is isomorphic to $T_G(\Z_p/p^r)/L_G(\Z_p/p^r)$. In order to use Shapiro lemma, first we note that the cohomology of $H^i(Y_G(V_r),\mathcal{O}_{\mathfrak{m}})$ can be calculated in terms of group cohomology of several congruence subgroups $\overline{\Gamma_i}$ of ${\rm PGL}_2(\mathcal{O}_E)$, see \cite[(6.4)]{HidaAnnals88} (assuming the $\overline{\Gamma_i}$ are torsion free). As ${\mathcal{E}}(p^r)$ acts trivially on the corresponding locally symmetric space, by Shapiro lemma we have
\[
e^{\rm n.ord}H^i(Y_G(V_r),\mathcal{O}_{\mathfrak{m}}/\mathfrak{m}^n)\cong e^{\rm n.ord}H^i(Y_G(K_0(p^{r+1})),\mathcal{O}_{\mathfrak{m}}/\mathfrak{m}^n[T_G(\Z_p/p^r)/{\mathcal{E}}(p^r)L_G(\Z_p/p^r)]).
\]

As applying the ordinary projector kills the contribution by $\Lambda_{\overline{N}_{G,r}}$ \cite[(2.5a)]{Hida1993Duke}, 
the latter in turn is isomorphic to 
\[
e^{\rm n.ord} H^i(Y_G(K_0(p^{r+1})),D_r \otimes_{\mathcal{O}_{\mathfrak{m}}} \mathcal{O}_{\mathfrak{m}}/\mathfrak{m}^n).
\]
By a variation of Lemma \ref{lem:Tplowerlevel} to lower the level at $p$, this is also isomorphic to 
\[
e^{\rm n.ord} H^i(Y_G(K_0(p^2)),D_r \otimes_{\mathcal{O}_{\mathfrak{m}}} \mathcal{O}_{\mathfrak{m}}/\mathfrak{m}^n).
\]
The result follows from this and the definition of the module
$ e^{\rm n.ord} H^{i}(Y_G(K_0(p^2)),\Lambda_{G/H})$.
\end{proof}

We let $\mathcal{H}_{\rm Iw }$ be the (big) Hecke algebra acting on $H^i_{\rm Iw}(Y_G(V_\infty),  \mathcal{O}_{\mathfrak{m}})$, generated by the (spherical) Hecke operators $T_v$, $S_v$ at each place $v$ where $V_r$ is hyperspecial, and by $T_{\eta}$. 
 
\begin{theorem} \label{ControlTheorem}
Let $\underline{k} \in \Z^{\Sigma_E}$ be such that $k_\sigma = k_{\sigma'}$ for all $\sigma \in \Phi$ and let $\mathfrak{m}_f$ be a maximal ideal of the Hecke algebra of $e^{\rm n.ord} H^{2d}(Y_G(V_r),\mathscr{V}^{\underline{k},0}_{\mathcal{O}_{\mathfrak{m}}})$ and suppose that it lifts\footnote{This always holds if $\mathfrak{m}_f$ is a maximal ideal of residual characteristic $0$ supported only in parabolic cohomology - see \cite[Theorem 2.7]{BDJ22}} to an ideal of $\mathcal{H}_{\rm Iw}$. If 
\[
e^{\rm n.ord} H^\bullet (Y_G(V_r),\mathscr{V}^{\underline{k},0}_{\mathcal{O}_{\mathfrak{m}}})_{\mathfrak{m}_f} \cong e^{\rm n.ord} H^{2d}(Y_G(V_r),\mathscr{V}^{\underline{k},0}_{\mathcal{O}_{\mathfrak{m}}})_{\mathfrak{m}_f},
\]
and the cohomology is torsion free, then the specialisation map \[
 e^{\rm n.ord} H^{2d}_{\rm Iw}(Y_G(V_\infty), \mathcal{O}_{\mathfrak{m}}) \otimes_{\Lambda_{G/H}} \kappa_{f} \rightarrow e^{\rm n.ord} H^{2d}(Y_G(V_r),\mathscr{V}^{\underline{k},0}_{\kappa_{f}}),
 \]
induced by the moment map in Definition \ref{def:mom}, is an isomorphism after localisation at $\mathfrak{m}_f$, where $\kappa_{f}$ is the residual field of $\mathfrak{m}_f$.
\end{theorem}
\begin{proof}

First, the moment map of Definition \ref{def:mom} induces a specialisation map 
\[
e^{\rm n.ord} H^{2d}_{\rm Iw}(Y_G(V_\infty), \mathcal{O}_{\mathfrak{m}}) \otimes_{\Lambda_{G/H}} \kappa_{f} \rightarrow e^{\rm n.ord} H^{2d}(Y_G(V_r),\mathscr{V}^{\underline{k},0}_{\kappa_{f}}).
 \]
 
Now, note that we can use Proposition \ref{prop:IwasawaOrdinarycoho} to  
get
\[
 e^{\rm n.ord} H^i_{\rm Iw}(Y_G(V_\infty), \mathcal{O}_{\mathfrak{m}}) \otimes_{\Lambda_{G/H}} \kappa_{f} \cong e^{\rm n.ord} H^{i}(Y_G(K_0(p^2)),\Lambda_{G/H}) \otimes_{\Lambda_{G/H}} \kappa_{f},
 \]
so we are left to study the tensor product on the right hand side. The default of the right hand side being isomorphic to $e^{\rm n.ord} H^{2d}(Y_G(V_r),\mathscr{V}^{\underline{k},0}_{\kappa_{f}})$ is controlled by a spectral sequence as in \cite[Cor. 2.7.11]{LRZ}: after localization at $\mathfrak{m}_f$, the hypothesis that the $\mathfrak{m}_f$-part of the cohomology is concentrated in a single degree ensures that all the terms in {\it loc. cit.} vanish for $i \neq 0$, giving that the above map is an isomorphism. 

\end{proof}

\begin{remark}
    To apply this theorem in characteristic $p$, one needs to verify that the cohomology is concentrated in one degree. This is not easy and in general not true, for example if there are Eisenstein congruences. We refer to \cite{CarTam} for the state-of-the-art result.
\end{remark}

\section{\texorpdfstring{$p$}{p}-adic families of twisted Hirzebruch--Zagier cycles}

\subsection{Cycles at finite level in cohomology}

 Throughout the section, we assume that $K^{(p)}$ is a neat open compact subgroup of $G(\A_f)$, which is stable under $\tau \in {\rm Gal}(E/F)$. Recall that we have the morphism
\[ \xymatrix{    u \circ \iota_{\pmb{r}}: Y_H(u K_{\pmb{r}}u^{-1} \cap H)  \ar[r]^-{\iota_{\pmb{r}}} & Y_G(u K_{\pmb{r}}u^{-1}) \ar[r]^-{u} &  Y_G(K_{\pmb{r}}),}  \]
where $u \in G(\Z_p)$ denotes the element of Definition \ref{def:representativeorbit}. Thanks to Lemma \ref{closedimm}, we can define the pushforward map $u_\star \circ {\iota_{\pmb{r}}}_\star $ in cohomology. For a given local system $\mathscr{V}$ on $Y_G(K_{\pmb{r}})$, we have

\[u_\star \circ {\iota_{\pmb{r}}}_\star: H^\bullet ( Y_H(u K_{\pmb{r}}u^{-1} \cap H), \mathscr{V}_{|_H}) \longrightarrow H^{\bullet + 2d} ( Y_G( K_{\pmb{r}}), \mathscr{V}). \]

Recall that, in \S \ref{branchinglaws}, we have shown that there exists an $H$-equivariant map $\iota^{\underline{k}}_{\mathcal{O}_{\mathfrak{m}}} : \mathcal{O}_{\mathfrak{m}} \hookrightarrow V^{\underline{k},0}_{\mathcal{O}_{\mathfrak{m}}}$
with $\underline{k} \in \Z^{\Sigma_E}$ such that $k_\sigma = k_{\sigma'}$ for all $\sigma \in \Phi$. In view of Lemma \ref{rmk:uflag}, we normalize it by letting  
\[ \iota^{\dagger,\underline{k}}_{\mathcal{O}_{\mathfrak{m}}} : \mathcal{O}_{\mathfrak{m}} \hookrightarrow V^{\underline{k},0}_{\mathcal{O}_{\mathfrak{m}}}\]
be the map which sends $1 \in \mathcal{O}_{\mathfrak{m}}$ to $\Delta^{\dagger, \underline{k}}:=a_{\underline{k}}^{-1} \cdot \Delta^{\underline{k}}$, where $a_{\underline{k}} \in \O_\mathfrak{m}^\times$ is the constant introduced in Lemma \ref{rmk:uflag}. 
The induced map in cohomology is \[ \iota^{\dagger ,\underline{k}}_{\mathcal{O}_{\mathfrak{m}}} : H^0 ( Y_H(U), \mathcal{O}_{\mathfrak{m}}) \longrightarrow H^0 ( Y_H(U), \mathscr{V}^{\underline{k},0}_{\mathcal{O}_{\mathfrak{m}}}|_H), \]
 where $\mathscr{V}^{\underline{k},0}$ denotes the local system associated with $V^{\underline{k},0}$.
Note that the image of the characteristic class $\mathbf{1}_{ Y_H(U)}$ via $\iota^{\dagger, \underline{k}}_{\mathcal{O}_{\mathfrak{m}}}$ can be explicitly described by using the invariant vector $\Delta^{\dagger, \underline{k}}$  (\textit{cf}. \cite[Proposition 9.1]{GetzGore}).

\begin{definition}\label{defcyclesfinitelevel}
 Let $\underline{k} \in \Z^{\Sigma_E}$ be such that $k_\sigma = k_{\sigma'}$ for all $\sigma \in \Phi$. The Hirzebruch--Zagier cycle $z_{\pmb{r}}^{\underline{k}}$ of weight $\underline{k}$ and level $K_{\pmb{r}}$, with $\pmb{r} \in \Z_{\geq 0}^{|\Upsilon_{F,p}|}$, is defined as \[z_{\pmb{r}}^{\underline{k}}:= u_\star \circ {\iota_{\pmb{r}}}_\star \circ  \iota^{\dagger ,\underline{k}}_{\mathcal{O}_{\mathfrak{m}}}(\mathbf{1}_{Y_H(u K_{\pmb{r}}u^{-1} \cap H)}) \in H^{2d} ( Y_G( K_{\pmb{r}}), \mathscr{V}^{\underline{k},0}_{\mathcal{O}_{\mathfrak{m}}}).\]\index{$z_{\pmb{r}}^{\underline{k}}$}
 We also define the twisted Hirzebruch--Zagier cycle $\mathscr{Z}_{\pmb{r}}^{\underline{k}}$ of weight $\underline{k}$ and level $V_{\pmb{r}}$ as \[\mathscr{Z}_{\pmb{r}}^{\underline{k}}:=\eta^{\pmb{r}}_\ast\, z_{\pmb{r}}^{\underline{k}} \in H^{2d} ( Y_G( V_{\pmb{r}}), \mathscr{V}^{\underline{k},0}_{\mathcal{O}_{\mathfrak{m}}}),\] \index{$\mathscr{Z}_{\pmb{r}}^{\underline{k}}$}
 where $\eta^{\pmb{r}}:= \prod_\frakp \eta_\frakp^{r_\frakp}$\index{$\eta^{\pmb{r}}$} and $\eta^{\pmb{r}}_\ast$ denotes the composition of all the traces of $\eta_\frakp^{r_\frakp}$ normalized as in Remark \ref{rmk:normofalletap}.
\end{definition}
\noindent When $\pmb{r}=(r, \dots, r)$, the level subgroups are the $K_r$ and $V_r$ of Definition \ref{def:levelgroups} and we denote the corresponding classes by $\mathscr{Z}_r^{\underline{k}}$, $z_r^{\underline{k}}$. 
 \begin{remark}
  When $\pmb{r}=\pmb{0}$, there is an overlap in the notation used as the levels $K_0,V_0$ and classes $\mathscr{Z}_0^{\underline{k}}$, $z_0^{\underline{k}}$ both coincide.
\end{remark}

\subsection{Universal cycles}

Following \cite{LoefflerSphericalvarieties}, we can use the results of \S \ref{Towerplusmoments} to construct a universal family of Hirzebruch--Zagier cycles.

\begin{proposition}\label{Step1inNR}
For all $\pmb{r} \in \Z_{\geq 1}^{|\Upsilon_{F,p}|}$, we have \[{\rm pr}_{\pmb{r}+\pmb{1},\pmb{r}, \ast} \left( \mathscr{Z}_{\pmb{r}+1}^{\underline{k}} \right)= T_\eta \cdot \mathscr{Z}_{\pmb{r}}^{\underline{k}}, \]
where recall that ${\rm pr}_{\pmb{r}+\pmb{1},\pmb{r}}$ denotes the natural degeneracy map $Y_G(V_{\pmb{r}+\pmb{1}}) \to Y_G(V_{\pmb{r}})$.
\end{proposition}

\begin{proof}
This is \cite[Proposition 4.5.1]{LoefflerSphericalvarieties}: by the Cartesian diagram of Proposition \ref{Cartesian}, we have that 
\begin{align*}  {\rm pr}_\ast z_{\pmb{r}+\pmb{1}}^{\underline{k}}=(\pi_{\pmb{r}}^G)^\ast z_{\pmb{r}}^{\underline{k}} .  
\end{align*}
Applying the trace of $\eta$ to both sides, we get 
\begin{align} \label{normrelforclassuntwisted}  {\rm Tr}_{[\eta]_{\pmb{r}+\pmb{1},\pmb{r}}}( z_{\pmb{r}+\pmb{1}}^{\underline{k}})=T_\eta \cdot z_{\pmb{r}}^{\underline{k}}.  
\end{align}
Note that the (integrally normalized) trace $\eta^{\pmb{r}}_\ast$ satisfies the equality 
\[ \eta^{\pmb{r}}_\ast \circ {\rm Tr}_{[\eta]_{{\pmb{r}}+{\pmb{1}},{\pmb{r}}}} = {\rm pr}_{{\pmb{r}}+{\pmb{1}},{\pmb{r}}, \ast} \circ \eta^{{\pmb{r}}+{\pmb{1}}}_\ast,    \]
induced by the commutative diagram \[ \xymatrix{Y_G(K_{\pmb{r}+\pmb{1}}) \ar[rr]^-{\eta^{{\pmb{r}}+{\pmb{1}}}} \ar[d]_-{[\eta]_{{\pmb{r}}+{\pmb{1}},{\pmb{r}}}} & &  Y_G(V_{{\pmb{r}}+{\pmb{1}}}) \ar[d]^-{{\rm pr}_{{\pmb{r}}+{\pmb{1}},{\pmb{r}}}} \\ Y_G(K_{\pmb{r}})\ar[rr]^-{\eta^{{\pmb{r}}}} & & Y_G(V_{\pmb{r}}).  }   \]     
This implies that, after applying $\eta^{{\pmb{r}}}_\ast$ to \eqref{normrelforclassuntwisted}, we get \[ {\rm pr}_{{\pmb{r}}+{\pmb{1}},{\pmb{r}}, \ast}\left( \mathscr{Z}_{{\pmb{r}}+{\pmb{1}}}^{\underline{k}} \right)= \eta^{\pmb{r}}_\ast\left( T_\eta\cdot z_{\pmb{r}}^{\underline{k}}\right) = T_\eta \cdot  \mathscr{Z}_{\pmb{r}}^{\underline{k}}.\] 
\end{proof}

Recall that we have the $H$-equivariant vector $\Delta^{\dagger, \underline{k}}=a_{\underline{k}}^{-1} \cdot \Delta^{\underline{k}}$ in $ V_{\mathcal{O}_{\mathfrak{m}}}^{\underline{k},0}$,
with $\underline{k} \in \Z^{\Sigma_E}$ such that $k_\sigma = k_{\sigma'}$ for all $\sigma \in \Phi$ (\textit{cf}. Lemma \ref{invariantvectorintegral}). For any $r \geq 1$ denote by $\Delta^{\dagger,\underline{k}}_r$ its reduction modulo $\mathfrak{m}^r$. 
\linebreak 

We have gathered all the ingredients to define the desired Hirzebruch--Zagier class:

\begin{theorem}\label{BIGCLASS}
There is a Big Hirzebruch--Zagier class \[ \mathscr{Z}_{\infty} := \varprojlim_r T_{\eta}^{-r} e^{\nord} \mathscr{Z}_r^{\underline{0}} \in  e^{\nord} H^{2d}_{\rm Iw}(Y_G(V_\infty), \mathcal{O}_{\mathfrak{m}}). \]
Let $\underline{k} \in \Z^{\Sigma_E}$ be such that $k_\sigma = k_{\sigma'}$ for all $\sigma \in \Phi$. Then, for every $\pmb{r} \in \Z_{\geq 1}^{|\Upsilon_{F,p}|}$, we have
\[
\mathrm{mom}^{\underline{k}}_{\pmb{r}}\, \mathscr{Z}_{\infty}= T_{\eta}^{-\pmb{r}} e^{\nord} \,\mathscr{Z}_{\pmb{r}}^{\underline{k}},
\]
with $T_{\eta}^{-\pmb{r}} = \prod_\frakp T_{\eta_\frakp}^{-r_\frakp}$.
\end{theorem}
\begin{proof}

After applying the ordinary idempotent, by Proposition \ref{Step1inNR} the classes $T_{\eta}^{-r} e^{\nord}  \mathscr{Z}_r^{\underline{0}}$ are compatible via the pushforward ${\rm pr}_{r+1,r, \ast}$ for $r \geq 1$ and thus they define an element (\textit{cf}. \cite[Proposition 4.5.2]{LoefflerSphericalvarieties}) \[ \mathscr{Z}_{\infty} \in  e^{\nord} H^{2d}_{\rm Iw}(Y_G(V_\infty), \mathcal{O}_{\mathfrak{m}}). \]

We now evaluate the moment map on $\mathscr{Z}_{\infty}$ as in \cite[Theorem 5.2.1]{LRZ}. For $\pmb{r} \in \Z_{\geq 1}^{|\Upsilon_{F,p}|}$, let $r = {\rm max}_\frakp \{ r_\frakp\}$ and note that $\mathrm{mom}^{\underline{k}}_{\pmb{r}}$ factors as the composition of $\mathrm{mom}^{\underline{k}}_{r}$ and the trace of the natural projection ${\rm pr}_{r, \pmb{r}}:Y_G(V_r) \to Y_G(V_{\pmb{r}})$. We start with showing that 
\begin{align}\label{intermediateintegralmom}
    \mathrm{mom}^{\underline{k}}_{r}\, \mathscr{Z}_{\infty}= T_{\eta}^{-r} e^{\nord} \,\mathscr{Z}_{r}^{\underline{k}}.
\end{align}
Unfolding Definition \ref{def:mom}, proving this follows from showing that \[\mathscr{Z}_r^{\underline{k}} \, \text{mod }\mathfrak{m}^r = ( \mathscr{Z}_r^{\underline{0}} \, \text{mod }\mathfrak{m}^r ) \cup v_{\underline{k},r}^{\rm hw}\]
as elements in $H^{2d}(Y_G(V_r),\mathscr{V}^{\underline{k},0}_{\mathcal{O}_\mathfrak{m}/\mathfrak{m}^r})$. Now, recall that $\mathscr{Z}_r^{\underline{k}} = \eta^r_\ast( z_r^{\underline{k}})$, where $\eta^r_\ast$ is the composition of the pushforward by $\eta^r$ and the morphism of local systems $\eta^r_\flat: \mathscr{V}^{\underline{k},0}_{\mathcal{O}_\mathfrak{m}/\mathfrak{m}^r} \to (\eta^r)^\ast \mathscr{V}^{\underline{k},0}_{\mathcal{O}_\mathfrak{m}/\mathfrak{m}^r}$. Recall that $\eta^r_\flat$ is associated with a map $V^{\underline{k},0}_{\mathcal{O}_\mathfrak{m}/\mathfrak{m}^r} \to (\eta^r)^\ast V^{\underline{k},0}_{\mathcal{O}_\mathfrak{m}/\mathfrak{m}^r}$ which is the identity on the highest weight subspace and acts by powers of $p^r$ (hence zero) elsewhere (\textit{cf}. Remark \ref{rem:normalizationTp}). Thus, it suffices to show 
\[z_r^{\underline{k}} \; \text{mod }\mathfrak{m}^r = ( z_r^{\underline{0}} \; \text{mod }\mathfrak{m}^r ) \cup v_{\underline{k},r}^{\rm hw}\; \text{mod ker}\,\eta^r_\flat.\]

\noindent Using the projection formula, we can write the left hand side as \[ ( z_r^{\underline{0}} \; \text{mod }\mathfrak{m}^r ) \cup u^{-1} \cdot \Delta^{\dagger,\underline{k}}_r,\] thus the result follows from checking that $u^{-1} \cdot \Delta^{\dagger,\underline{k}}_r$ and $v_{\underline{k},r}^{\rm hw}$ agree modulo $ \text{ker}\,\eta^r_\flat$. As $\eta^r_\flat$ kills everything outside the highest weight space, this follows from checking that the projection of $u^{-1} \cdot \Delta^{\underline{k}}$ to the highest weight space of $V^{\underline{k},0}_{\mathcal{O}_\mathfrak{m}}$ agrees with $v_{\underline{k}}^{\rm hw}$, up to the constant $a_{\underline{k}}$. This is immediate from the definition of $u$ and the explicit realization of $\Delta^{\underline{k}}$ given in Lemma \ref{lemma:branching}, as shown in Lemma \ref{rmk:uflag}. This proves \eqref{intermediateintegralmom} and the result follows from showing that, if $\pmb{s} = (r)_\frakp -\pmb{r}$, we have \[{\rm pr}_{r,\pmb{r},\star}\,\mathscr{Z}_{r}^{\underline{k}} = T_\eta^{\pmb{s}} \mathscr{Z}_{\pmb{r}}^{\underline{k}}. \]
This follows, as in Proposition \ref{Step1inNR}, by using the Cartesian diagrams of Remark \ref{rmkonsingleCartDiag} recursively for each $\frakp$ with $\pmb{s}_\frakp \ne 0$.
\end{proof}

We note that, after allowing $p$-adic square roots of the determinant, one could also define another $p$-adic family of cycles interpolating Hirzebruch--Zagier cycles at odd weights and non-trivial central characters.

\subsection{Euler Factor on the First Floor}
In what follows we relate the universal cycle to the Hirzebruch--Zagier cycles of level $V_0$. 

Recall that, for each $\frakp \in \Upsilon_{F,p}$, $\eta_\frakp : = \matrix{\pi_\frakp}{}{}{1} \in \GL_2(F_\frakp)$ and we denote in the same way the corresponding element of $H(\Q_p)$.  Note that the morphism $\eta_{K_0',K_0}: Y_G(K_0') \to Y_G(K_0)$ can be written as the composition of 
\begin{small}
\[ Y_G(K_0') \xrightarrow[]{{\rm pr}_{\frakp_1} \circ \eta_{\frakp_1}} Y_G(K_0(\Upsilon_{F,p} \smallsetminus \{ \frakp_1 \})) \xrightarrow[]{{\rm pr}_{\frakp_2} \circ \eta_{\frakp_2}} Y_G(K_0(\Upsilon_{F,p} \smallsetminus \{ \frakp_1,\frakp_2 \})) \cdots Y_G(K_0(\{ \frakp_r \})) \xrightarrow[]{{\rm pr}_{\frakp_r} \circ \eta_{\frakp_r}} Y_G(K_0), \]
\end{small}

\noindent where we denoted $\Upsilon_{F,p} = \{\frakp_1, \dots, \frakp_r \}$\index{$\Upsilon_{F,p}$} and where, for any $\mathfrak{S}$ subset of $ \Upsilon_{F,p}$, we have set in \eqref{eqdefintermediatelevelsigma}
\[ K_0(\mathfrak{S}) = K^{(p)} \cdot (  \prod_{\frakp \in \mathfrak{S}} \Gamma^0_0(\pi_\frakp,\pi_\frakp) \cdot  \prod_{\frakp \in \Upsilon_{F,p} \smallsetminus \mathfrak{S}} \Gamma_0( \pi_\frakp)). \]  
Recall (Remark \ref{rmk:normofalletap}) that we have \[ T_{\eta} = \prod_{\frakp \in \Upsilon_{F,p}} T_{\eta_\frakp},\] where $T_{\eta_\frakp}$ is the Hecke operator acting on cohomology with coefficients $\mathscr{V}^{\underline{k},0}$, normalized by multiplying the trace by $\eta_{\frakp}$ by $p^{\sum_{\sigma \in \Sigma_{\frakp}}  k_{\sigma}}$ when $k_\sigma=k_{\sigma'}$ for every $\sigma \in \Phi$. 

\begin{proposition}\label{proponK1toK0}
 Let $\underline{k} \in \Z^{\Sigma_E}$ be such that $k_\sigma=k_{\sigma'}$ for every $\sigma \in \Phi$. We have \[ {\rm Tr}_{[\eta]_{1,0}} (z_1^{\underline{k}}) = \prod_{\frakp \in \Upsilon_{F,p}} \left( T_{\eta_{\frakp}} - q_{\frakp} \cdot p^{\sum_{\sigma \in \Sigma_{\frakp}}  k_{\sigma}} \right) \cdot 
z_0^{\underline{k}} \in  H^{2d} ( Y_G( K_{0}), \mathscr{V}^{\underline{k},0}_{\mathcal{O}_{\mathfrak{m}}}) . \]
\end{proposition}

\begin{proof} 
Let $\frakp$ be in $\Upsilon_{F,p}$. By Proposition \ref{Cartesianwhenr=0}, denoting $\mathfrak{S} : = \Upsilon_{F,p} \smallsetminus \{\frakp\}$
we have the Cartesian diagram 
\begin{eqnarray*}  \xymatrix{  
 Y_H(u_\mathfrak{S}K_0'u_\mathfrak{S}^{-1} \cap H) \sqcup Y_H(u K_0'u^{-1} \cap H)  \ar[d]_{({\rm pr}_{\mathfrak{S},\emptyset}^H, {\rm pr}_{\mathfrak{S},\frakp}^H )}  \ar[rrr]^-{(\tilde{\iota}_\mathfrak{S},  u \circ \iota )}   & & &  Y_G(K_0') \ar[d]^{\pi_{\frakp}^G} \\
Y_H( u_\mathfrak{S}K_0(\mathfrak{S})u_\mathfrak{S}^{-1} \cap H)  \ar[rrr]^-{ \tilde{\iota}_\mathfrak{S}} & & &  Y_G(K_0(\mathfrak{S})),    }  \end{eqnarray*}
 where, recall that $K_0(\mathfrak{S})$, resp. $K_0'$, has $\frakp$-component $\Gamma_0( \pi_\frakp)$, resp. $\Gamma^0_0(\pi_\frakp,\pi_\frakp)$. From the Cartesian diagram, we obtain \begin{align}\label{firststepCartesian}
     {\pi_{\frakp}^{G,*}} \circ  \tilde{\iota}_{\mathfrak{S},*} = \tilde{\iota}_{\mathfrak{S},*} \circ {\rm pr}_{\mathfrak{S},\emptyset}^{H,*} + u_* \circ \iota_* \circ {\rm pr}_{\mathfrak{S},\frakp}^{H,*}.
 \end{align}

Apply \eqref{firststepCartesian} to $\mathbf{1}_{\mathfrak{S}} : = \mathbf{1}_{u_\mathfrak{S}K_0(\mathfrak{S})u_\mathfrak{S}^{-1} \cap H}$:  \[  u_* \circ \iota_* \circ \iota^{\dagger,\underline{k}}_{\mathcal{O}_\mathfrak{m}} \circ {\rm pr}_{\mathfrak{S},\frakp}^{H,*}(\mathbf{1}_{\mathfrak{S}}) = {\pi_{\frakp}^{G,*}} \circ  \tilde{\iota}_{\mathfrak{S},*} \circ \iota^{\dagger,\underline{k}}_{\mathcal{O}_\mathfrak{m}} ( \mathbf{1}_{\mathfrak{S}} ) - \tilde{\iota}_{\mathfrak{S},*} \circ \iota^{\dagger,\underline{k}}_{\mathcal{O}_\mathfrak{m}} \circ {\rm pr}_{\mathfrak{S},\emptyset}^{H,*} (\mathbf{1}_{\mathfrak{S}}) . \]
Recall that $ u K_0'u^{-1} \cap H = u K_1 u^{-1} \cap H $, hence, similarly to Proposition \ref{Step1inNR}, the equation above reads as
\[ {\rm pr}_{K_1,K_0',\ast} (z_1^{\underline{k}}) =  {\pi_{\frakp}^{G,*}} \circ  \tilde{\iota}_{\mathfrak{S},*} \circ \iota^{\dagger,\underline{k}}_{\mathcal{O}_\mathfrak{m}} ( \mathbf{1}_{\mathfrak{S}} ) - \tilde{\iota}_{\mathfrak{S},*} \circ \iota^{\dagger,\underline{k}}_{\mathcal{O}_\mathfrak{m}} \circ {\rm pr}_{\mathfrak{S},\emptyset}^{H,*} (\mathbf{1}_{\mathfrak{S}}).\]
We now apply the trace of $Y_G(K_0') \xrightarrow[]{{\rm pr}_\frakp \circ \eta_{\frakp}} Y_G(K_0(\mathfrak{S}))$, which we normalize, similarly to Lemma  \ref{Lem:traceeta}, by multiplying it with $p^{\sum_{\sigma \in \Sigma_{\frakp}}  k_{\sigma}}$. We get 
\begin{align*}
  ({\rm pr}_\frakp \circ \eta_{\frakp})_* \circ {\rm pr}_{K_1,K_0',\ast} (z_1^{\underline{k}}) &=  ({\rm pr}_\frakp \circ \eta_{\frakp})_* \left ( {\pi_{\frakp}^{G,*}} \circ  \tilde{\iota}_{\mathfrak{S},*} \circ \iota^{\dagger,\underline{k}}_{\mathcal{O}_\mathfrak{m}} ( \mathbf{1}_{\mathfrak{S}} ) - \tilde{\iota}_{\mathfrak{S},*} \circ {\rm pr}_{\mathfrak{S},\emptyset}^{H,*} \circ \iota^{\dagger,\underline{k}}_{\mathcal{O}_\mathfrak{m}} (\mathbf{1}_{\mathfrak{S}})\right) \\
 &=   T_{\eta_\frakp} \cdot \left (   \tilde{\iota}_{\mathfrak{S},*} \circ \iota^{\dagger,\underline{k}}_{\mathcal{O}_\mathfrak{m}} ( \mathbf{1}_{\mathfrak{S}} ) \right) - ({\rm pr}_\frakp \circ \eta_{\frakp})_* \left ( \tilde{\iota}_{\mathfrak{S},*} \circ {\rm pr}_{\mathfrak{S},\emptyset}^{H,*} \circ \iota^{\dagger,\underline{k}}_{\mathcal{O}_\mathfrak{m}} (\mathbf{1}_{\mathfrak{S}}) \right).
\end{align*}
We now examine the last term of the right hand side. Since $\eta_\frakp$ defines an element of $H(\Q_p)$ and $u_\mathfrak{S}K_0(\mathfrak{S})u_\mathfrak{S}^{-1} \cap H$ has $\frakp$-component equal to $\Gamma_{H,0}(\pi_\frakp)$, we have a commutative diagram 

\begin{eqnarray*}  \xymatrix{  
 Y_H(u_\mathfrak{S}K_0'u_\mathfrak{S}^{-1} \cap H)  \ar[d]_{{\rm pr}_\frakp^H \circ \eta_{\frakp}}  \ar[rr]^-{\tilde{\iota}_\mathfrak{S}}    & &  Y_G(K_0') \ar[d]^{{\rm pr}_\frakp \circ \eta_{\frakp}} \\
Y_H( u_\mathfrak{S}K_0(\mathfrak{S})u_\mathfrak{S}^{-1} \cap H)  \ar[rr]^-{ \tilde{\iota}_\mathfrak{S}}  & &  Y_G(K_0(\mathfrak{S})).    }  \end{eqnarray*}
We thus have \[ ({\rm pr}_\frakp \circ \eta_{\frakp})_* \circ \tilde{\iota}_{\mathfrak{S},*} = \tilde{\iota}_{\mathfrak{S},*} \circ ({\rm pr}_\frakp^H \circ \eta_{\frakp})_* .\]

This implies that 
\begin{align*}
     ({\rm pr}_\frakp \circ \eta_{\frakp})_* \left ( \tilde{\iota}_{\mathfrak{S},*} \circ {\rm pr}_{\mathfrak{S},\emptyset}^{H,*} \circ \iota^{\dagger,\underline{k}}_{\mathcal{O}_\mathfrak{m}} (\mathbf{1}_{\mathfrak{S}})  \right) &=\tilde{\iota}_{\mathfrak{S},*} \circ ({\rm pr}_\frakp^H \circ \eta_{\frakp})_* \circ {\rm pr}_{\mathfrak{S},\emptyset}^{H,*} \circ \iota^{\dagger,\underline{k}}_{\mathcal{O}_\mathfrak{m}} (\mathbf{1}_{\mathfrak{S}}) .
\end{align*}
 The Hecke correspondence $({\rm pr}_\frakp^H \circ \eta_{\frakp})_* \circ {\rm pr}_{\mathfrak{S},\emptyset}^{H,*}$ commutes with $\iota^{\dagger,\underline{k}}_{\mathcal{O}_\mathfrak{m}}$ and acts on \[\mathbf{1}_{\mathfrak{S}} \in H^0 ( Y_H(u_\mathfrak{S}K_0(\mathfrak{S})u_\mathfrak{S}^{-1} \cap H), \mathcal{O}_{\mathfrak{m}})\] 
by $q_\frakp \cdot p^{\sum_{\sigma \in \Sigma_{\frakp}}  k_{\sigma}}$, with $q_\frakp$ being the degree of ${\rm pr}_\frakp^H$. Therefore we get \begin{align*}
  ({\rm pr}_\frakp \circ \eta_{\frakp})_* \circ {\rm pr}_{K_1,K_0',\ast} (z_1^{\underline{k}}) &=  \left( T_{\eta_\frakp} - q_\frakp \cdot p^{\sum_{\sigma \in \Sigma_{\frakp}}  k_{\sigma}} \right) \cdot \left (   \tilde{\iota}_{\mathfrak{S},*} \circ \iota^{\dagger,\underline{k}}_{\mathcal{O}_\mathfrak{m}} ( \mathbf{1}_{\mathfrak{S}} ) \right).
\end{align*}    

We now reiterate the same proof for each prime in $\Upsilon_{F,p}$. Let ${\frakp'} \in \mathfrak{S}$ and denote $\mathfrak{S}' : = \mathfrak{S} \smallsetminus \{ {\frakp'} \}$. By Proposition \ref{Cartesianwhenr=0}, we have the Cartesian diagram 
\begin{eqnarray*}  \xymatrix{  
 Y_H(u_\mathfrak{S'}K_0(\mathfrak{S})u_\mathfrak{S'}^{-1} \cap H) \sqcup Y_H(u_\mathfrak{S}K_0(\mathfrak{S})u_\mathfrak{S}^{-1} \cap H)  \ar[d]_{({\rm pr}_{\mathfrak{S}',\emptyset}^H, {\rm pr}_{\mathfrak{S}',{\frakp'}}^H )}  \ar[rrr]^-{(\tilde{\iota}_{\mathfrak{S}'}, \tilde{\iota}_\mathfrak{S} )}   & & &  Y_G(K_0(\mathfrak{S})) \ar[d]^{\pi_{{\frakp'}}^G} \\
Y_H( u_{\mathfrak{S}'}K_0(\mathfrak{S}')u_{\mathfrak{S}'}^{-1} \cap H)  \ar[rrr]^-{ \tilde{\iota}_{\mathfrak{S}'}} & & &  Y_G(K_0(\mathfrak{S}')).    }  \end{eqnarray*}

This Cartesian diagram allows us to write 
\begin{align*}
      \tilde{\iota}_{\mathfrak{S},*} \circ \iota^{\dagger,\underline{k}}_{\mathcal{O}_\mathfrak{m}} ( \mathbf{1}_{\mathfrak{S}} ) = {\pi_{{\frakp'}}^{G,*}} \circ  \tilde{\iota}_{\mathfrak{S}',*} \circ \iota^{\dagger,\underline{k}}_{\mathcal{O}_\mathfrak{m}} ( \mathbf{1}_{\mathfrak{S}'} ) - \tilde{\iota}_{\mathfrak{S}',*} \circ {\rm pr}_{\mathfrak{S}',\emptyset}^{H,*} \circ \iota^{\dagger,\underline{k}}_{\mathcal{O}_\mathfrak{m}}  (\mathbf{1}_{\mathfrak{S}'}),
\end{align*}
where we have again denoted  $\mathbf{1}_{\mathfrak{S}'} : = \mathbf{1}_{u_{\mathfrak{S}'}K_0(\mathfrak{S}')u_{\mathfrak{S}'}^{-1} \cap H}$. Exactly as above, after applying the normalized trace of  \[Y_G(K_0(\mathfrak{S})) \xrightarrow[]{{\rm pr}_{{\frakp'}} \circ \eta_{{\frakp'}}} Y_G(K_0(\mathfrak{S}')),\]
we get 
\begin{align*}
  ({\rm pr}_{{\frakp'}} \circ \eta_{{\frakp'}})_* \circ \tilde{\iota}_{\mathfrak{S},*} \circ \iota^{\dagger,\underline{k}}_{\mathcal{O}_\mathfrak{m}} ( \mathbf{1}_{\mathfrak{S}} )  &=  \left( T_{\eta_{{\frakp'}}} - q_{{\frakp'}} \cdot p^{\sum_{\sigma \in \Sigma_{{\frakp'}}}  k_{\sigma}} \right) \cdot \left (   \tilde{\iota}_{\mathfrak{S}',*} \circ \iota^{\dagger,\underline{k}}_{\mathcal{O}_\mathfrak{m}} ( \mathbf{1}_{\mathfrak{S}'} ) \right).
\end{align*}   
As the trace $({\rm pr}_{{\frakp'}} \circ \eta_{{\frakp'}})_*$ commutes with the correspondence $\left( T_{\eta_\frakp} - q_\frakp \cdot p^{\sum_{\sigma \in \Sigma_{\frakp}}  k_{\sigma}} \right)$, we thus have 
\begin{align*}
  ({\rm pr}_{{\frakp'}} \circ \eta_{{\frakp'}})_* \circ ({\rm pr}_\frakp \circ \eta_{\frakp})_* \circ {\rm pr}_{K_1,K_0',\ast} (z_1^{\underline{k}}) &=  \left( T_{\eta_\frakp} - q_\frakp \cdot p^{\sum_{\sigma \in \Sigma_{\frakp}}  k_{\sigma}} \right)\left( T_{\eta_{{\frakp'}}} - q_{{\frakp'}} \cdot p^{\sum_{\sigma \in \Sigma_{{\frakp'}}}  k_{\sigma}} \right) \cdot \left (   \tilde{\iota}_{\mathfrak{S}',*} \circ \iota^{\dagger,\underline{k}}_{\mathcal{O}_\mathfrak{m}} ( \mathbf{1}_{\mathfrak{S}'} ) \right).
\end{align*}    
We now reiterate the process for each remaining prime in $\Upsilon_{F,p}$ to get \[ \eta_{K_0',K_0,\ast} \circ {\rm pr}_{K_1,K_0',\ast} (z_1^{\underline{k}}) = \prod_{\frakp \in \Upsilon_{F,p}} \left( T_{\eta_\frakp} - q_\frakp \cdot p^{\sum_{\sigma \in \Sigma_{\frakp}}  k_{\sigma}} \right) \cdot \left (   \iota_{0,*} \circ \iota^{\dagger,\underline{k}}_{\mathcal{O}_\mathfrak{m}} ( \mathbf{1}_{K_0 \cap H} ) \right), \]
which reads as 
\[ {\rm Tr}_{[\eta]_{1,0}} (z_1^{\underline{k}}) = \prod_{\frakp \in \Upsilon_{F,p}} \left( T_{\eta_\frakp} - q_\frakp \cdot p^{\sum_{\sigma \in \Sigma_{\frakp}}  k_{\sigma}} \right) \cdot 
z_0^{\underline{k}} . \]
\end{proof}

Recall that $\mathscr{Z}^{\underline{k}}_1 : = \eta_{\ast} z^{\underline{k}}_1 $, with $\eta_{\ast}$ the normalized trace associated with multiplication by $\eta$ from $Y_G(K_1)$ to $Y_G(V_1)$. Proposition \ref{proponK1toK0} translates into the following :

\begin{theorem}\label{proponV1toV0}
Let $\underline{k} \in \Z^{\Sigma_E}$ be such that $k_\sigma=k_{\sigma'}$ for every $\sigma \in \Phi$. We have \[ {\rm pr}_{1,0, \ast} (\mathscr{Z}_1^{\underline{k}}) = \prod_{\frakp \in \Upsilon_{F,p}} \left( T_{\eta_{\frakp}} - q_{\frakp} \cdot p^{\sum_{\sigma \in \Sigma_{\frakp}}  k_{\sigma}} \right) \cdot 
\mathscr{Z}_0^{\underline{k}} \in  H^{2d} ( Y_G( V_{0}), \mathscr{V}^{\underline{k},0}_{\mathcal{O}_{\mathfrak{m}}}) . \]
\end{theorem}

\begin{proof}
It follows from Proposition \ref{proponK1toK0} and the commutative diagram  \[ \xymatrix{Y_G(K_{1}) \ar[rr]^-{\eta} \ar[d]_-{[\eta]_{1,0}} & &  Y_G(V_{1}) \ar[d]^-{{\rm pr}_{1,0}} \\ Y_G(K_0)\ar@{=}[rr] & & Y_G(V_0).  }   \] 
\end{proof}

Recall that we have the moment map \[  \mathrm{mom}^{\underline{k}}_{1}: H^i_{\rm Iw}(Y_G(V_\infty), \mathcal{O}_{\mathfrak{m}}) \rightarrow H^i(Y_G(V_1),\mathscr{V}^{\underline{k},0}_{\mathcal{O}_{\mathfrak{m}}}). \]
Denote by $\mathrm{mom}^{\underline{k}}_{0}$ the composition of $\mathrm{mom}^{\underline{k}}_{1}$ with ${\rm pr}_{1,0, \ast}$. Then we have the following : 

\begin{corollary}\label{coro:EulerFactor1}
Let $\underline{k} \in \Z^{\Sigma_E}$ be such that $k_\sigma=k_{\sigma'}$ for every $\sigma \in \Phi$. Then, 
\[\mathrm{mom}^{\underline{k}}_{0}\, \mathscr{Z}_{\infty} =  \prod_{\frakp \in \Upsilon_{F,p}} \left( 1 - q_{\frakp} \cdot p^{\sum_{\sigma \in \Sigma_{\frakp}}  k_{\sigma}}\, T_{\eta_{\frakp}}^{-1} \right) \cdot 
e^{\nord} \mathscr{Z}_0^{\underline{k}}.   \]
\end{corollary}

\begin{proof}
It follows from Theorem \ref{BIGCLASS} and Theorem \ref{proponV1toV0}.
\end{proof}

\begin{remark}\label{Interestedreadercase}
Analogously, we could prove an explicit relation between $\mathrm{mom}^{\underline{k}}_{\pmb{r}}\, \mathscr{Z}_{\infty}$ and $e^{\nord} \mathscr{Z}_{\pmb{r}}^{\underline{k}}$, where $r_\frakp$ is zero for $\frakp$ in a proper subset $\mathfrak{S}$ of $\Upsilon_{F,p}$. In this case, the formula would read as \[\mathrm{mom}^{\underline{k}}_{\pmb{r}}\, \mathscr{Z}_{\infty} =  \prod_{\frakp \in \mathfrak{S}} \left( 1 - q_{\frakp} \cdot p^{\sum_{\sigma \in \Sigma_{\frakp}}  k_{\sigma}}\, T_{\eta_{\frakp}}^{-1} \right) \cdot \prod_{\frakp \in \Upsilon_{F,p}\smallsetminus \mathfrak{S}} T_{\eta_{\frakp}}^{-r_\frakp} \cdot
e^{\nord} \mathscr{Z}_{\pmb{r}}^{\underline{k}}.   \]
We leave the details to the interested reader.
\end{remark}

\subsection{Twisting the universal cycles}

We make precise the choice of the tame level and show how we can further twist the big Hirzebruch--Zagier cycles by certain characters of conductor a power of $p$.

Let $\mathfrak{c}$ be an ideal of $\mathcal{O}_E$ assumed to be coprime with $p$ in this section, and let $\mathfrak{c}_F$ be its intersection with $\mathcal{O}_F$. Let $\chi_F$ be a finite order Hecke character for $F$ of conductor $\mathfrak{c}_F$ and let $\chi = \chi_F \circ {\rm N}_{E/F}$, which we view as a character of $K_0(\mathfrak{c})$. Consider also a unitary Hecke character $\theta$ of $E$ of conductor $\mathfrak{b}$, coprime with $p$, and whose restriction to $\mathbb{A}^\times_F$ is $\chi_F \eta_{E/F}$.
Set $K^{(p)} = K_0(\mathfrak{c})$, with $\mathfrak{c}$ stable under $\tau \in {\rm Gal}(E/F)$. As the Hecke correspondences used in Definitions \ref{def:twistedcycle1} and \ref{def:twistedcycle2} commute with $T_\eta$ and the action of $u$, we can analogously define the twisted cycles $\mathscr{Z}_{\pmb{r}}^{\underline{k},\chi}(\mathfrak{c}),\mathscr{Z}_{\pmb{r}}^{\underline{k}}(\mathfrak{c})_{\theta} \in H^{2d}(Y_G(V_{\pmb{r}}), \mathscr{V}_{\mathcal{O}_{\mathfrak{m}}}^{\underline{k},0}(\chi))$\index{$\mathscr{Z}_{\pmb{r}}^{\underline{k},\chi}(\mathfrak{c})$}, up to possibly enlarging the coefficient ring $\mathcal{O}_{\mathfrak{m}}$, where (see Definition \ref{def:multilevels}) \[V_{\pmb{r}} =K_0(\mathfrak{c}) \cdot \prod_{\frakp \in \Upsilon_{F,p}} V_{\frakp,r_\frakp},\]
as well as $\mathscr{Z}_{\pmb{r}}^{\underline{k},\chi^{-1}\theta^2}(\mathfrak{c}') \in H^{2d}(Y_G(V_{\pmb{r}}), \mathscr{V}_{\mathcal{O}_{\mathfrak{m}}}^{\underline{k},0}(\chi^{-1}\theta^2))$\index{$\mathscr{Z}_{\pmb{r}}^{\underline{k},\chi^{-1}\theta^2}(\mathfrak{c}')$} where the tame component of $V_{\pmb{r}}$ is $K_0(\mathfrak{c}')$. 
By Theorem \ref{BIGCLASS}, these cycles can still be assembled into a $p$-adic family of cycles and denote by $\mathscr{Z}_{\infty}^{\chi}$\index{$\mathscr{Z}_{\infty}^{\chi}$} and $\mathscr{Z}_{\infty,\theta}$\index{$\mathscr{Z}_{\infty,\theta}$} the corresponding twisted big Hirzebruch--Zagier cycles in  \[
 e^{\rm n.ord} H^{2d}_{\rm Iw}(Y_G(V_\infty), \mathcal{O}_{\mathfrak{m}}(\chi))  \simeq e^{\rm n.ord} H^{2d}(Y_G(K_0(p^2\mathfrak{c})),\Lambda_{G/H}(\chi)),
 \]
and similarly for $\mathscr{Z}_{\infty}^{\chi^{-1}\theta^2}$\index{$\mathscr{Z}_{\infty}^{\chi^{-1}\theta^2}$}. Note that we could further twist these  big Hirzebruch--Zagier cycles by finite order characters of $T_G(\Z_p)/\overline{\mathcal{E}}L_G(\Z_p)$ as the following shows.  

Let $\chi'$ be any finite order character of conductor $p^r$ which factors through $T_G(\Z_p/p^r)/\overline{\mathcal{E}}L_G(\Z_p/p^r)$, and denote also by $\chi'$ the induced morphism of $\Lambda_{G/H}$. Note that the decomposition 
\[
\mathcal{O}_{\mathfrak{m}}[T_G(\Z_p/p^r)/{\mathcal{E}}(p^r)L_G(\Z_p/p^r)][1/p] \cong \prod_{\chi'}\mathcal{O}_{\mathfrak{m}}[1/p],
\]
where $\chi'$ ranges through such $L_\mathfrak{m}$-valued characters of conductor $p^r$ is given by the idempotent elements
\[
e_{\chi'}=\frac{1}{|T_G(\Z_p/p^r)/{\mathcal{E}}(p^r)L_G(\Z_p/p^r)|}\sum_{\gamma \in T_G(\Z_p/p^r)/{\mathcal{E}}(p^r)L_G(\Z_p/p^r)} {\chi'}^{-1}(\gamma) \gamma^*.
\]
For any $\underline{k} \in \Z^{\Sigma_E}$ such that $k_\sigma=k_{\sigma'}$ for every $\sigma \in \Phi$, the idempotent $e_{\chi'}$ induces the map $\mathrm{mom}_{r,\chi'}^{\underline{k}}$ in cohomology, which is defined as the composition 
 \begin{align*} 
  e^{\rm n.ord}H^{2d}_{\rm Iw}(Y_G(V_\infty), \mathcal{O}_{\mathfrak{m}}(\chi)) &\xrightarrow{} e^{\rm n.ord} H^{2d}(Y_G(V_r),\mathscr{V}^{\underline{k},0}_{\mathcal{O}_{\mathfrak{m}}}(\chi)) \\
  &\to e^{\rm n.ord} H^{2d}(Y_G(K_0(p^{r+1}\mathfrak{c})), \mathscr{V}^{\underline{k},0}_{L_{\mathfrak{m}}}( \chi{\chi'}^{-1})),
 \end{align*}
  where the first map is $\mathrm{mom}_{r}^{\underline{k}}$ and the second one is given by the composition \[\frac{1}{|T_G(\Z_p/p^r)/{\mathcal{E}}(p^r)L_G(\Z_p/p^r)|}\sum_{\gamma \in T_G(\Z_p/p^r)/{\mathcal{E}}(p^r)L_G(\Z_p/p^r)} {\chi'}^{-1}(\gamma) {\rm pr}^{\gamma V_r \gamma^{-1}}_{K_0(p^{r+1}\mathfrak{c}),\star} \circ \gamma^{-1}_*.\] Here, we see $\chi'$ as a character of $K_0(p^{r+1}\mathfrak{c})$ by setting
\[ \chi'_v\left ( \begin{smallmatrix} a & b \\ c & d \end{smallmatrix} \right)= \begin{cases} \chi'_v(\tau(a)^{-1}d) & \text{ if } \frakp_v \mid \mathfrak{c}p^{r+1} \\ 1 & \text{ otherwise,} \end{cases} \]
hence we can define the local system $\mathscr{V}^{\underline{k},0}_{L_{\mathfrak{m}}}( \chi{\chi'}^{-1})$ over $Y_G(K_0(p^{r+1}\mathfrak{c}))$ as in \cite[Proposition 6.3]{GetzGore}. More generally, if $\chi'$ is a character of conductor $p^{\pmb{r}}$, for $\pmb{r} = (r_\frakp)_\frakp \in \Z_{\geq 0}^{|\Upsilon_{F,p}|}$, one can define a map \[\mathrm{mom}_{\pmb{r},\chi'}^{\underline{k}}:
  e^{\rm n.ord}H^{2d}_{\rm Iw}(Y_G(V_\infty), \mathcal{O}_{\mathfrak{m}}(\chi)) \to e^{\rm n.ord} H^{2d}(Y_G(K_0(p^{\pmb{r}+\pmb{1}}\mathfrak{c})), \mathscr{V}^{\underline{k},0}_{L_{\mathfrak{m}}}( \chi{\chi'}^{-1})), \] \index{$\mathrm{mom}_{\pmb{r},\chi'}^{\underline{k}}$}
by using the idempotent
\[
e_{\chi'}=\frac{1}{|
K_0(p^{\pmb{r}+\pmb{1}}\mathfrak{c})/ \mathcal{O}_E^{*,+} V_{\pmb{r}}|}\sum_{\gamma \in
K_0(p^{\pmb{r}+\pmb{1}}\mathfrak{c})/ \mathcal{O}_E^{*,+} V_{\pmb{r}}} {\chi'}^{-1}(\gamma) {\rm pr}^{\gamma V_{\pmb{r}} \gamma^{-1}}_{K_0(p^{\pmb{r}+\pmb{1}}\mathfrak{c}),\star} \circ \gamma^{-1}_*.
\]

\begin{definition}
For a character  $\chi'$ of conductor $p^{\pmb{r}}$, we let \[\mathscr{Z}_{\pmb{r},\chi'}^{\underline{k},\chi}:= \mathrm{mom}_{\pmb{r},\chi'}^{\underline{k}}\, \mathscr{Z}_{\infty}^{\chi} \in  e^{\rm n.ord} H^{2d}(Y_G(V_{\pmb{r}}), \mathscr{V}_{\mathcal{O}_{\mathfrak{m}}}^{\underline{k},0}(\chi {\chi'}^{-1})).\]
\end{definition} 
These twisted cycles, which won't be used in the following, can be seen as a $p$-adic counterpart of Getz--Goresky's cycles of Definition \ref{def:twistedcycle2}.

\section{Hida theory and base change}

Consider a weight $\kappa= (\underline{k},m) \in (\Z_{\geq 0})^{\Sigma_E} \times \Z$ and let $L$ be a large enough number field to contain the Galois closure of $E$. As in the previous section, we let $\mathfrak{m}$ be a prime of $L$ above $p$ and take the $\mathfrak{m}$-adic completion $L_\mathfrak{m}$, with ring of integers  $\mathcal{O}_\mathfrak{m}$. 
For every prime ideal $\frakp$ of $F$ dividing $p$ we let $\Sigma_\frakp$ be the subset of $\Sigma_F$ of ($p$-adic) places $\sigma$ sending $\frakp$ to the maximal ideal of the valuation ring of $F_\frakp$. We introduce two Hecke operators on $S_{\kappa}(K_0(\mathfrak{c}),\chi)$, for any ideal $\mathfrak{c}$ in $\mathcal{O}_E$ with $\mathfrak{p} \mid \mathfrak{c}$ : the usual Hecke operator $U_{\frakp}$ defined using the double coset operator $\left[K_0(\mathfrak{c}) \left(\begin{smallmatrix}
     \varpi_{\mathfrak{p}} & 0  \\
    0  & 1
 \end{smallmatrix} \right)K_0(\mathfrak{c}) \right]$ and the normalized $U_\frakp^{\nord}:= {p}^{{\sum_{\sigma \in \Sigma_\frakp}\frac{k_\sigma-m}{2}}}U_\frakp$.\footnote{The term nearly ordinary goes back to Hida's first works on families for Hilbert modular forms. It is to stress that, except in the case of parallel weight Hilbert modular forms, the right notion of ordinarity is not slope-zero but slope $p^{-{\sum_{\sigma \in \Sigma_F}\frac{k_\sigma-m}{2} }}$. With the development of Hida theory in more general settings \cite{HidaPEL, BrascaRosso}, it is more customary to simply say ordinary, and implicitly referring to an optimally normalized $U_p$ operator. As we work with Hilbert modular forms, we prefer to  use Hida's original terminology. See \cite[\S 2]{HidaFourier}.} 
 
 \begin{lemma}\label{TetaVSUp}
 The map
 \[S_{\kappa}(K_0(\mathfrak{c}),\chi) \to H^{2d}( Y_G(K_0(\mathfrak{c})),\mathscr{V}^{\underline{k},{-m}}(
\chi^{-1}_0)_\C),\, f \mapsto  [\omega_J(f^{-\iota})], \]
is Hecke-equivariant in the sense that \[  \omega_J((f_{|K_0(\mathfrak{c}) x K_0(\mathfrak{c})})^{-\iota}) = T^{\text{naive}}_{x^{\iota}} \cdot  \omega_J(f^{-\iota}),\]
where we denoted $T^{\text{naive}}_{x^{\iota}}  = \pi_{2,\ast} \circ x^{\iota}_\ast \circ \pi_{1}^\ast$ the correspondence for  \[ \xymatrix{ & Y_G(x^\iota K_{0}(\mathfrak{c})x^{-\iota} \cap K_{0}(\mathfrak{c})) \ar[dl]_{\pi_1} \ar[r]^-{\cdot x^\iota} & Y_G(K_{0}(\mathfrak{c}) \cap x^{-\iota} K_{0}(\mathfrak{c})x^{\iota})\ar[dr]^{\pi_2} & \\  Y_G(K_0(\mathfrak{c})) & & &  Y_G(K_0(\mathfrak{c})),
}
\]
with $\pi_i$'s the natural degeneracy maps and  $x^{\iota}=\mathrm{det}(x)x^{-1}$.
 \end{lemma}

\begin{proof}

The Hecke equivariance follows from \cite[Proposition 7.13]{GetzGore} and the Remark at p.60 of \textit{loc.cit.}.
\end{proof}

\noindent Twist the map of Lemma \ref{TetaVSUp} with the Atkin--Lehner involution to get 
 \begin{align}\label{twistedmapHeckeequivariance}
     S_{\kappa}(K_0(\mathfrak{c}),\chi) \to H^{2d}( Y_G(K_0(\mathfrak{c})),\mathscr{V}^{\underline{k},{-m}}(
\chi_0)_\C),\, f \mapsto  [\omega_J(W_\mathfrak{c}^*f^{-\iota})].
 \end{align}
Then, the Hecke correspondence $T_{\eta}$ introduced in Definition \ref{definitionTeta} and $U_\frakp^{\nord}$ are related as follows.

\begin{proposition}\label{prop:TetaUp}
Under the map \eqref{twistedmapHeckeequivariance}, if $\frakp \mid \mathfrak{c}$, the normalized correspondence $$ T_{\eta_\frakp} = p^{\sum_{\sigma \in \Sigma_\frakp}\frac{k_\sigma - m}{2} }T^{\text{naive}}_{\eta_\frakp},$$ with $\eta_\frakp = \left(\begin{smallmatrix}
     \varpi_{\mathfrak{p}} & 0  \\
    0  & 1
 \end{smallmatrix} \right)$, corresponds to $U_\frakp^{\nord}$ on $S_{\kappa}(K_0(\mathfrak{c}),\chi)$, and $T_\eta = \prod_\mathfrak{p} U_\frakp^{\nord}$.
\end{proposition}
\begin{proof}
    By the commutative diagram \eqref{commdiagAL}, it is enough to check how the adjoint Hecke correspondence
\[ T^{\text{naive},*}_{\eta_\frakp}  = W_{\mathfrak{c},*} \circ T^{\text{naive}}_{{\eta_\frakp}} \circ W_{\mathfrak{c}}^* \]
acts on $\omega_J(f^{-\iota})$. Note that $T^{\text{naive},*}_{\eta_\frakp}$ is associated to the element ${ \left(\begin{smallmatrix}
     1 & 0  \\
    0  & \varpi_{\mathfrak{p}}
 \end{smallmatrix} \right)}$ and by Lemma \ref{TetaVSUp} it acts on $\omega_J(f^{-\iota})$ as $U_\frakp$. The result then follows from comparing the two normalization factors.
\end{proof}

\noindent In view of Proposition \ref{prop:TetaUp} we define an idempotent $e^{\rm n.ord}:= \lim_{n} \prod_{\frakp \mid p} {{U_{\frakp}^{\nord}}}^{n!}$, that we denote as the nearly-ordinary operator of Lemma \ref{lem:ordinaryIwasawaCoho}. 
Hida proved that it is well-defined on \[H^{\ast}( Y_G(K_{11}(\mathfrak{c}p^{\alpha})),\mathscr{V}^{\underline{k},m}_{\mathcal{O}_\mathfrak{m}})\] for all $\mathfrak{c}$ and ${\alpha} \geq 0$. 
As in \cite[\S 3.2]{DimitrovAJM}, we let $\mathcal{H}^{G}_{11}(\underline{k},{m})$ be  \[
 \mathrm{Hom}_{\mathcal{O}_\mathfrak{m}}\left (\varinjlim_{\alpha \geq 1} e^{\rm n.ord} H^{\star}_c( Y_G(K_{11}(\mathfrak{c}p^{\alpha})),\mathscr{V}^{\underline{k},{m}}_{\mathcal{O}_\mathfrak{m}} \otimes_{\mathcal{O}_\mathfrak{m}} L_\mathfrak{m}/\mathcal{O}_\mathfrak{m}), L_\mathfrak{m}/\mathcal{O}_\mathfrak{m} \right ).
\] 
Hida's main result \cite[Theorem 2.3]{HidaFamilies89} is that the space
\[
\mathcal{H}^{G}_{11}(\underline{k},{m})
\]
 is naturally a finite and torsion-free algebra over the Iwasawa algebra 
\[
\Lambda_G:=\mathcal{O}_\mathfrak{m}  \llbracket \varprojlim_{\alpha, \beta \geq 1} \mathbb{A}_{E,f}^{\times}/E^{\times}(\mathbb{A}^{\times}_{E,f} \cap K_{11}(\mathfrak{c}p^{\alpha})) \times (\mathcal{O}_E/p^{\beta})^\times \rrbracket
\] 
By \cite[Lemma 2.1]{HidaFamilies89}, the action of $(z,u) \in \varprojlim_{\alpha, \beta \geq 1} \mathbb{A}_{E,f}^{\times}/E^{\times}(\mathbb{A}^{\times}_{E,f} \cap K_{11}(\mathfrak{c}p^{\alpha})) \times (\mathcal{O}_E/p^{\beta})^\times $ corresponds to the action of the matrix  ${\matrix{uz}{0}{0}{z}}$, where with a slight abuse of notation $u$ denotes a lift to $\prod_\frakp \mathcal{O}^{\times}_{E,\frakp}$ (hence $u$ has non-trivial component only at $p$). 

\begin{remark}\leavevmode
\begin{enumerate}
    \item If $p$ does not divide the cardinality of $\mathbb{A}_{E,f}^{\times}/E^{\times}(\mathbb{A}^{\times}_{E,f} \cap K_{11}(\mathfrak{c}p^{\alpha}))$ then this group (and so the associated Iwasawa algebra) splits canonically as a torsion-part and a non-torsion part.
    \item Note that the variable $z$ is essentially the cyclotomic variable (assuming Leopoldt conjecture for $E$) and interpolates the weight $m$, and $u$ in $\prod_\frakp \mathcal{O}^{\times}_{E,\frakp}$ interpolates the weight $\underline{k}$.
\end{enumerate}
\end{remark}

Hida also shows that $\mathcal{H}^{G}_{11}(\underline{k},{m})$ is independent of the weight, in the sense that we have canonical isomorphisms (see \cite[Theorem 3.1]{DimitrovAJM})
\[  \mathcal{H}^{G}_{11}(\underline{k},{m}) \cong  \mathcal{H}^{G}_{11}(\underline{0},{0}) \otimes_{\mathcal{O}_\mathfrak{m}} \mathcal{O}_\mathfrak{m} v_{\underline{k}}^{\rm hw},\]
with $v_{\underline{k}}^{\rm hw}$ the highest weight vector in $V^{(\underline{k},{m})}$.  For ease of notation, we shall often drop the dependence on the weight from the notation, thus writing only $\mathcal{H}^{G}_{11}$.

 We refer to $\mathcal{H}^{G}_{11}(\underline{k},{m})$ as the cuspidal $p$-adic ordinary Hecke algebra as it is isomorphic to the $p$-adic Hecke algebra constructed using cuspidal Hilbert modular forms as in \cite[\S 3]{HidaFourier}: see the proofs of Theorem 3.2 and 3.3 in  \cite[\S 11]{HidaAnnals88}), which use the duality between the space of modular forms and the Hecke algebra, neatly stated in \cite[Thm 2.28]{HidaHilbertBook}.

We define $\mathcal{H}^{H}_{11}$ and $\Lambda_H$ similarly. the module $\mathcal{H}^{H}_{11}$ is of finite type over $\Lambda_H$ and interpolates systems of Hecke eigenvalues for nearly-ordinary Hilbert modular forms for $F$.

\begin{definition}\label{def:HidaFamilies}
 A Hida family $\mathbf{f}$ for $F$, resp. for $E$, is an irreducible component of $
 \mathcal{H}^{H}_{11}\otimes_{\mathcal{O}_\mathfrak{m}} L_\mathfrak{m}$, resp. $
  \mathcal{H}^{G}_{11}\otimes_{\mathcal{O}_\mathfrak{m}} L_\mathfrak{m}$. We shall denote by $\mathcal{K}_\mathbf{f}$\index{$\mathcal{K}_\mathbf{f}$} the fraction field of the coefficient of $\mathbf{f}$. 
\end{definition}
We invert $p$ as in what follows we will need to avoid congruences between base change forms and non-base change forms.

\subsection{Base change families}

Recall that, for a given weight $\kappa = (\underline{k},m)$ for $F$, we let $\hat{\kappa} = (\hat{\underline{k}},m)$ be the weight for $E$ for which $\hat{k}_\sigma =\hat{k}_{\sigma'}$ for every $\sigma \in \Phi$.
We have the following proposition, see \cite[\S 8.2]{GetzGore}.

\begin{proposition}\label{prop:basechangeHecke}
There is a base change map
\[
\mathrm{BC}:\mathcal{H}^{G}_{11}(\hat{\underline{k}},{m})  \rightarrow \mathcal{H}^{H}_{11}(\underline{k},{m})
\]
defined as follows: for every prime ideal $\mathfrak{q} \subset \mathcal{O}_E$, with $\mathfrak{q}$ coprime to $\mathfrak{c}$ and $ \mathfrak{d}_{E/F} \mathcal{O}_E$, which is above $\mathfrak{r}$ in $\mathcal{O}_F$  we have
\begin{align*}
    \mathrm{BC}(T_{\mathfrak{q}})&:= \begin{cases} T_{\mathfrak{r}}   & \text{ if } \mathfrak{r} \text{ splits}, \mathfrak{r}=\mathfrak{q} \mathfrak{q}',\\
    T_{\mathfrak{r}^2}-N_{F/\mathbb{Q}}(\mathfrak{r}) S_{\mathfrak{r}}  &  \mbox{ if } \mathfrak{r} \mbox { inert}, \mathfrak{r}=\mathfrak{q},\end{cases} \\ 
    \mathrm{BC}(S_{\mathfrak{q}})&:= \begin{cases} S_{\mathfrak{r}}     & \text{ if } \mathfrak{r} \text{ splits}, \mathfrak{r}=\mathfrak{q} \mathfrak{q}',\\
     S_{\mathfrak{r}^2}  &  \mbox{ if } \mathfrak{r} \mbox { inert}, \mathfrak{r}=\mathfrak{q}.\end{cases} 
\end{align*}
\end{proposition}

\noindent Over $\mathbb{C}$, there is a natural section of $\mathrm{BC}$ (\textit{cf}. \cite[Lemma 8.1]{GetzGore}).

\begin{definition}
 Let $f_E$ be a primitive normalized Hilbert modular form of weight $(\hat{\underline{k}},m)$. We say that $f_E$ is the base change of a (primitive) Hilbert modular form $f$ for $F$ if the corresponding map of the Hecke algebra $\mathcal{H}^{G}_{11}(\hat{\underline{k}},{m})$ factors through $\mathrm{BC}$.
\end{definition}
\noindent This is compatible with the characterization given in Proposition \ref{AsaiForBC}: if $f_E$ is the base change of $f$ then the corresponding cuspidal automorphic representations are related by \[\widehat{\pi(f)} = \pi(f_E).\] 
Note that the conductor of $f$ is not necessarily the conductor of  $f_E$ intersected with $\mathcal{O}_F$ (\textit{cf}. \cite[\S E.6]{GetzGore}). Also note that $\mathrm{BC}$ maps nearly ordinary forms to nearly ordinary forms.

We define base change Hida families. Using the map of Proposition \ref{prop:basechangeHecke}, we could define a Hida family for $E$ to be a base change family from $F$ if it factors through $\mathrm{BC}$. We however refine this definition to better suit the construction of the adjoint $p$-adic $L$-function given in \S \ref{sec:constructionpadicLF}. We first decompose 
\[
\mathbb{A}_{E,f}^{\times}/E^{\times}(\mathbb{A}^{\times}_{E,f} \cap K_{11}(\mathfrak{c}p^{\alpha})) = {\mathbb{A}_{E,f}^{(p)}}^{\times}/E^{\times}(\mathbb{A}^{\times}_{E,f} \cap K_{11}(\mathfrak{c})) \times (E\otimes \mathbb{Z}_p)^{\times}/ E^{\times}( (E\otimes \mathbb{Z}_p)^{\times} \cap  K_{11}(p^{\alpha}))
\]
and similarly do it for $F$.
\begin{definition}\label{def:projtoIwAlgG/H}
 Let $\chi_0$ be a character of $\mathbb{A}_{E,f}^{\times}/E^{\times}(\mathbb{A}^{\times}_{E,f} \cap K_{11}(\mathfrak{c}))$ with values in $\mathcal{O}_{\mathfrak{m}}$. The nearly-ordinary Hecke-algebra of type $\Phi$ and tame character $\chi_0$ is
 \[
 \mathcal{H}^{G/H}_{11}:=  \mathcal{H}^{G}_{11} \otimes_{\Lambda_G} \Lambda_{G/H}
 \]
 where the map 
 \[
 \Lambda_G \rightarrow \Lambda_{G/H} 
 \]
 is induced by the map 
 \[
  \varprojlim_{\alpha, \beta \geq 1} \mathbb{A}_{E,f}^{\times}/E^{\times}(\mathbb{A}^{\times}_{E,f} \cap K_{11}(\mathfrak{c}p^{\alpha})) \times (\mathcal{O}_E/p^{\beta})^\times \rightarrow 
 T_G(\Z_p)/ \overline{\mathcal{E}} L_G(\Z_p)
 \]
 that sends $(1,u)$ to the image of the matrix  ${\matrix{u}{}{}{1}}$ in  $T_G(\Z_p)/ \overline{\mathcal{E}} L_G(\Z_p)$ and $(z,1)$, after writing $z=z^{(p)}z_p$, to  $\chi_0(z^{(p)}){\matrix{z_p}{}{}{z_p}} \in \Lambda_{G/H}$, where the matrix is modulo $\overline{\mathcal{E}} L_G(\Z_p)$.
\end{definition}

Our goal is to define base change families starting from irreducible components of $  \mathcal{H}^{G/H}_{11} \otimes_{\mathcal{O}_{\mathfrak{m}}} L_{\mathfrak{m}}$. Before we do so, notice that if $f$ is a cuspidal Hilbert modular form for $F$ of level  $K_{11}(\mathfrak{c}_F p^{\alpha})$ with trivial Nebentypus at $p$, then the base-change $f_E$ has weights constrained to have $k_{\sigma}=k_{\sigma'}$ and trivial Nebentypus at $p$. Moreover, its adjoint character (see Definition \ref{def_of_adjoint_character}) factors through ${\rm N}_{E/F}$ and so has $p$-component trivial on $\overline{\mathcal{E}} L_G(\Z_p)$. Motivated by this, we give the following : 
\begin{definition}
Let $\Lambda_{H/Z_H}:=\mathcal{O}_\mathfrak{m}\llbracket T_H(\Z_p)/Z_H(\Z_p)\rrbracket$
 and let $\chi_{F,0}$ be a character of $\mathbb{A}_{F,f}^{\times}/F^{\times}(\mathbb{A}^{\times}_{F,f} \cap K_{11}(\mathfrak{c}_F))$ with values in $\mathcal{O}_{\mathfrak{m}}$. Define
 \[
  \mathcal{H}^{H/Z_H}_{11}:=\mathcal{H}^{H}_{11} \otimes_{\Lambda_H} \Lambda_{H/Z_H}
 \]
 where the map $\Lambda_H \rightarrow \Lambda_{H/Z_H}$
 is induced by the map 
 \[
  \varprojlim_{\alpha, \beta \geq 1} \mathbb{A}_{F,f}^{\times}/F^{\times}(\mathbb{A}^{\times}_{F,f} \cap K_{11}(\mathfrak{c}_Fp^{\alpha})) \times (\mathcal{O}_F/p^{\beta})^\times \rightarrow 
 T_H(\Z_p)/ Z_H(\Z_p)
 \]
 that sends $(1,u)$ to the image of the matrix  ${\matrix{u}{}{}{1}}$ in  $T_H(\Z_p)/ Z_H(\Z_p)$ and $(z,1)$, with $z=z^{(p)}z_p$, to $\chi_{F,0}(z^{(p)})$.
\end{definition}

As $\mathrm{BC}$ induces a map $ \mathcal{H}^{G/H}_{11} \to  \mathcal{H}^{H/Z_H}_{11}$, we can define base change families as follows.

\begin{definition}\label{def:BaseChangefamily}
 We say that a family, ({\it i.e.} an irreducible component of  $  \mathcal{H}^{G/H}_{11} \otimes_{\mathcal{O}_{\mathfrak{m}}} L_{\mathfrak{m}}$) is a base change family   (from $F$) if it factors through $ \mathcal{H}^{H/Z_H}_{11}$ via the map $\mathrm{BC}$ of Proposition \ref{prop:basechangeHecke}, for two compatible levels. 
\end{definition}

\begin{remark}\leavevmode
\begin{enumerate}
    \item We note that the tame character of a base change family necessarily factors through ${\rm N}_{E/F}$.
    \item  As we are fixing the component outside $p$ of the Nebentypus of the family and we later force $m=0$, all the Hilbert modular forms for $F$ coming from our families will have trivial central character/Nebentypus at $p$ and only an adjoint character (in the sense of Definition \ref{def_of_adjoint_character}). There is no loss of generality for the range of applications in this manuscript as the adjoint $L$-function is invariant by central twists.
\end{enumerate}
\end{remark}

\begin{proposition}\label{prop:primitiveBaseChange}
There is an idempotent $\mathbf{1}_{\mathrm{BC}}$ in $ \mathcal{H}^{G/H}_{11} \otimes \mathrm{Frac}(\Lambda_{G/H})$ such that 
\[
\mathbf{1}_{\mathrm{BC}} \mathcal{H}^{G/H}_{11}
\]
consists exactly only of primitive base change families.
\end{proposition}
\begin{proof}
As the Hecke operators are commutative and the nearly-ordinary Hecke algebra is of finite type, then 
\[
\mathcal{H}^{G/H}_{11} \otimes \mathrm{Frac}(\Lambda_{G/H})
\]
is semi-local. Then $\mathbf{1}_{\mathrm{BC}}$ is the idempotent associated with the sum of the idempotent of base change families which are primitive. 
\end{proof}

\subsection{The \texorpdfstring{$p$}{p}-stabilisation and forms of level \texorpdfstring{$V_{{{\pmb{r}}}}$}{Vr}}\label{sec:pstabilisation}

We write down explicitly the $p$-stabilisation of a nearly-ordinary cuspidal Hilbert modular form (for $F$) $f \in S_{\kappa}(K_0(\mathfrak{c}),\chi, \chi')$. Analogous definitions apply to forms for $E$. We refer the reader to \cite[Lemma 5.3]{HidaFourier} for details.

Suppose $f$ is nearly ordinary at $\mathfrak{p}$, {\it i.e.} 
$ U_\frakp f = a(\frakp,f)f$ and $\left| a(\frakp,f){p}^{{\sum_{\sigma \in \Sigma_\frakp}\frac{k_\sigma-m}{2}}}\right|_p=1$. There are three possible cases :

\begin{itemize}

  \item[(St)] If the corresponding automorphic representation is special at $\frakp$, then the level $\mathfrak{c}$ must be divisible exactly by $\mathfrak{p}$.  

  \item[(RamPS)] If $\frakp$ divides the conductor of $\chi$ or $\chi'$, then $\mathfrak{c}$ must be divisible by $\mathfrak{p}$.  

  \item[(UnrPS)] If none of the above is true, then the two Satake parameters ${\alpha_{1,\frakp}},{\alpha_{2,\frakp}}$ are both non-zero. 
 
\end{itemize}
In case (UnrPS), only one of these, say ${\alpha_{1,\frakp}}$, is such that $\left|{\alpha_{1,\frakp}}\Norm_{F/\Q}(\mathfrak{p})^{1/2}{p}^{{\sum_{\sigma \in \Sigma_\frakp}\frac{k_\sigma-m}{2}}}\right|_p=1$. 
We shall write ${\alpha_{\frakp}}={\alpha_{1,\frakp}}$\index{${\alpha_{\frakp}}$}, ${\alpha^\circ_{\frakp}}={\alpha_{1,\frakp}}{p}^{{\sum_{\sigma \in \Sigma_\frakp}\frac{k_\sigma-m}{2}}}$\index{${\alpha^\circ_{\frakp}}$}, and $\beta_{\frakp}={\alpha_{2,\frakp}}$. 

We define the $\frakp$-stabilisation $f_{{\alpha_{\frakp}}}$ of $f$ by \[
f_{{\alpha_{\frakp}}}:=f-\Norm_{F/\Q}(\mathfrak{p})^{1/2}{\beta_{\frakp}} f|[\varpi_{\frakp}]
\]
for $[\varpi_{\frakp}]$ the operator of \cite[\S 7.B]{HidaFourier}, given by the slash action of the matrix  $\left(\begin{smallmatrix}
     \varpi_{\frakp} & 0  \\
    0  & 1
 \end{smallmatrix} \right)$. Note that we have 
 \[
 a(\frakp,f_{{\alpha_{\frakp}}})={\alpha_{\frakp}}\Norm_{F/\Q}(\mathfrak{p})^{1/2}.
 \]
 
 In general, we shall write $f_\alpha$ for the composite of all the $\frakp$-stabilisation of $f$, for $\frakp$ for which $f$ is not primitive at $\frakp$.\\

We conclude with a lemma comparing forms (for $E$) of level $V_{\pmb{r}}$ with forms of level $K_0(\mathfrak{c} p^{\pmb{r}})$, with $\mathfrak{c}$ coprime with $p \mathcal{O}_E$. 
 Here, for $\pmb{r} \in \Z_{\geq 0}^{|\Upsilon_{F,p}|}$, we let $p^{\pmb{r}} =\prod_{\frakp \in \Upsilon_{F,p}} \prod_{\mathfrak{P}|\frakp} \mathfrak{P}^{r_\frakp}$ and  
\begin{align*}
       V_{\pmb{r}}&= K_{11}(\mathfrak{c})  \cdot \prod_{\frakp \in \Upsilon_{F,p}} V_{\frakp,r_\frakp}. 
\end{align*}  

Firstly, observe that, if $\chi$ and $\chi'$ are characters of $p$-part conductor $p^{{\pmb{r}}}$ such that at each prime $\mathfrak{P} \mid p^{\pmb{r}}$  
\begin{align}\label{eqcondateachpVr}
 \chi_\mathfrak{P}(d)\chi'_{\mathfrak{P}}(a^{-1}d) = 1,\text{ for } d \equiv \tau(a) \text{ mod } \mathfrak{P}^{r_\frakp},   
\end{align} then any form $f \in S_\kappa(K_0(\mathfrak{c}p^{\pmb{r}+\pmb{1}}),\chi, \chi')$ is invariant by $V_{\pmb{r}}$.

\begin{lemma}\label{FromGamma11toV}
Let $f$ be a modular form in $S_{\kappa}(K_0(\mathfrak{c}p^{{\pmb{r}}}),\chi, \chi')$ for two characters $\chi$ and $\chi'$ with $p$-part of conductor $p^{{\pmb{r}}_1}$ and $p^{{\pmb{r}}_2}$ respectively. Suppose ${\pmb{r}} = \max \left\{ {\pmb{r}}_1, {\pmb{r}}_2 \right\}$. Then the injection $f \mapsto f$
\[
S_{\kappa}(K_0(\mathfrak{c}p^{{\pmb{r}}}),\chi, \chi') \hookrightarrow S_\kappa(K_0(\mathfrak{c}p^{\pmb{r}+\pmb{1}}),\chi, \chi')
\] is an isomorphism on the ordinary parts. 
\end{lemma}
\begin{proof}
By the control theorem of Hida \cite[Corollary 3.3]{HidaFourier} every ordinary form of $S_\kappa(K_0(\mathfrak{c}p^{\pmb{r}+\pmb{1}}),\chi, \chi')$ is in $S_{\kappa}(K_0(\mathfrak{c} p^{{\pmb{r}}}),\chi, \chi')$, as $\chi$ and $\chi'$ are both characters modulo $p^{{\pmb{r}}}$. So the natural injection is an isomorphism.
\end{proof}

\begin{remark}\label{remark_on_Vr_invariant_forms}
Let $\mathfrak{c}_F$ be the intersection of $\mathfrak{c}$ with $\mathcal{O}_F$. Consider a nearly ordinary form (for $F$) $f \in \mathcal{S}_{(\underline{k},0)}(K_0(\mathfrak{c}_F p^{\pmb{r}+\pmb{1}}),\chi_F,\chi_F')$ such that, at each $\frakp \in \Upsilon_{F,p}$, $\chi_F$ has trivial $\frakp$-component and the $\frakp$-part of $\chi_F'$ has conductor exactly $\frakp^{r_\frakp}$. Let $f_E \in  \mathcal{S}_{(\hat{\underline{k}},0)}(K_0(\mathfrak{c} p^{\pmb{r}+ \pmb{1}}),\chi,\chi')$ be the base change of $f$. Then $f_E$ has Nebentypus $\chi = \chi_F \circ {\rm N}_{E/F}$ with trivial component at every prime above $p$. Moreover, by \cite[Proposition E.9]{GetzGore}, it has adjoint character $\chi' = \chi_F' \circ {\rm N}_{E/F}$, whose $\mathfrak{P}$-part has conductor $\mathfrak{P}^{r_{\frakp}}$ if $\mathfrak{P}|\frakp$, for every $\frakp \in \Upsilon_{F,p}$. In particular, \eqref{eqcondateachpVr} is satisfied and $f_E $ is invariant by  $V_{\pmb{r}}$.
\end{remark}

\subsection{The idempotent for a Hida family with  type \texorpdfstring{$J$}{J}}\label{sec:idempotent}
Let $\mathbf{1}_{\mathbf{f}}$ be the idempotent corresponding to the Hida family $\mathbf{f}$ of Hilbert modular forms for $E$. It acts on $\varinjlim_\alpha H^{2d}_c( Y_G(K_{11}(\mathfrak{c}p^{\alpha})),\mathscr{V}^{\underline{k},{m}}_{\mathcal{O}_\mathfrak{m}} )$, and over this space there is an action of the Hecke algebra at infinity, generated by the operators $[K_{\infty}^+ g_{\infty} K_{\infty}^+]$, for $g_{\infty} \in \left\{ \matrix{\pm 1}{0}{0}{1} \right\}^{\Sigma_E}$. As $p$ is odd, the action of this operator can be diagonalised: for an $F$-type $J$, we let $\mathbf{1}^\varepsilon$ be the idempotent corresponding to the eigenspace for $\varepsilon$\index{$\varepsilon$}, where the only requirement on $\varepsilon$ is that, for every $\sigma \in J$,
\[
\varepsilon(\matrix{- 1}{0}{0}{1})_\sigma \varepsilon(\matrix{- 1}{0}{0}{1})_{\sigma'} =-1. 
\]
We could take, for example $\varepsilon$ such that 
$\varepsilon(\matrix{- 1}{0}{0}{1})_\sigma=1$ if $\sigma \in J$ and $-1$ otherwise. We finally define $\mathbf{1}^\varepsilon_{\mathbf{f}}$\index{$\mathbf{1}^\varepsilon_{\mathbf{f}}$} to be \[ \mathbf{1}^\varepsilon_{\mathbf{f}} : = \mathbf{1}^\varepsilon \mathbf{1}_{\mathbf{f}}.\] 

Recall we fixed an isomorphism $i_p: \C \simeq \overline{\Q}_p$. Suppose that the Hida family is made of primitive forms (see \cite[Theorem 3.4]{HidaFourier}) and let $f$ be a classical specialisation of the family corresponding to a morphism
\[
P: \mathcal{H}^{G}_{11} \rightarrow L_\mathfrak{m}.
\]
Then the space $\mathbf{1}^\varepsilon_{\mathbf{f}}  \varinjlim_\alpha H^{2d}_c( Y_G(K_{11}(\mathfrak{c}p^{\alpha})),\mathscr{V}^{\underline{k},{m}}_{\mathcal{O}_\mathfrak{m}} ) \otimes_P L_\mathfrak{m} $ is one-dimensional over $L_\mathfrak{m}$, see the bottom of p.~38 in \cite{Hidanoncritical}. 
Moreover by {\it loc. cit.} the surjection $H^{2d}_c( Y_G(K_{11}(\mathfrak{c}p^{\alpha})),\mathscr{V}^{\underline{k},{m}}_{\C} ) \twoheadrightarrow H^{2d}_{\rm cusp}( Y_G(K_{11}(\mathfrak{c}p^{\alpha})),\mathscr{V}^{\underline{k},{m}}_{\C} ) $ admits a canonical section. Hence, if we start with a nearly ordinary form $f \in S_{\kappa}(K_0(\mathfrak{c}p^{\alpha}),\chi) $, the class $\mathbf{1}^\varepsilon\omega_J(W^*_{\mathfrak{c}p^\alpha}f^{-\iota})$ belongs to $\mathbf{1}^\varepsilon_{\mathbf{f}}  \varinjlim_\alpha H^{2d}_c( Y_G(K_{11}(\mathfrak{c}p^{\alpha})),\mathscr{V}^{\underline{k},{m}}_{L_\mathfrak{m}} ) \otimes_{\iota_p^{-1}} \C$ and it is a generator of this one dimensional space over $\mathbb{C}$. We note that the presence of the Atkin--Lehner involution is motivated by the fact that, in view of Proposition \ref{prop:TetaUp}, the map of Lemma \ref{TetaVSUp} becomes $U_p$-equivariant after twisting by the Atkin--Lehner involution. 
 If $\omega^\varepsilon(f)^{\mathrm{alg}}$\index{$\omega^\varepsilon(f)^{\mathrm{alg}}$} is a generator of $\mathbf{1}^\varepsilon_{\mathbf{f}}H^{2d}_c( Y_G(K_{11}(\mathfrak{c}p^{\alpha})),\mathscr{V}^{\underline{k},{m}}_{L_\mathfrak{m}} )$, we write
 \[
\omega^\varepsilon(f)^{\mathrm{alg}} = \frac{\mathbf{1}^\varepsilon i_p(\omega_J(W^*_{\mathfrak{c}p^\alpha}f^{-\iota}))}{i_p(\Omega^\varepsilon(f))}.
 \]
We call $\Omega^\varepsilon(f)\in \C^\times$\index{$\Omega^\varepsilon(f)$} the period for $f$ corresponding to $\varepsilon$. Similar definition can be made for Hilbert modular forms over $F$. Under the hypothesis on $\varepsilon$, this is non-zero by the discussion in \cite[\S 6]{Hidanoncritical}.

Let $\mathbf{f}$ be a Hida family for $F$ corresponding to an irreducible component of $ \mathcal{H}^{H/Z_H}_{11} \otimes_{\mathcal{O}_{\mathfrak{m}}} L_{\mathfrak{m}}$, with idempotent $\mathbf{1}_{\mathbf{f}}$. Denote by $\mathbf{f}_E$ the base-change Hida family obtained by composing  $\mathbf{1}_{\mathbf{f}}$ with $\mathrm{BC}$, which we assume to be made of primitive forms. Denote $\mathbf{f}_E'$ the component of $ \mathcal{H}^{G}_{11} \otimes_{\mathcal{O}_\mathfrak{m}} L_\mathfrak{m}$ obtained from $\mathbf{f}_E$. Recall that 
 \[ \Lambda_{G/H}=\mathcal{O}_{\mathfrak{m}}\llbracket T_G(\Z_p)/\overline{\mathcal{E}}L_G(\Z_p)\rrbracket,\] where $L_G = \{ {\matrix {\alpha} { } {} {\tau(\alpha)}}\,:\, \alpha \in {\rm Res}_{\mathcal{O}_E/\Z} \mathbf{G}_{\rm m} \}$. Then, the idempotent $\mathbf{1}_{\mathbf{f}_E}$ acts on $\varinjlim_r H^{2d}_c( Y_G(V_r),\mathscr{V}^{\underline{k},{m}}_{\mathcal{O}_\mathfrak{m}} )$. Let $f_E$ be a classical specialisation of $\mathbf{f}_E$ and $\mathbf{f}_E'$ corresponding to $P :  \mathcal{H}^{G/H}_{11} \to L_\mathfrak{m}$ and $P':  \mathcal{H}^{G}_{11} \to  \mathcal{H}^{G/H}_{11} \xrightarrow[]{P} L_\mathfrak{m}$. If the $p$-part of its adjoint character has conductor $p^{\pmb{r}}$, then by Lemma \ref{FromGamma11toV} and Remark \ref{remark_on_Vr_invariant_forms}, we have an isomorphism 
 \[\mathbf{1}^\varepsilon_{\mathbf{f}_E'}  H^{2d}_c( Y_G(K_{11}(\mathfrak{c}p^{\pmb{r}})),\mathscr{V}^{\underline{k},{m}}_{\mathcal{O}_\mathfrak{m}} ) \otimes_{P'} L_\mathfrak{m} \to   \mathbf{1}^\varepsilon_{\mathbf{f}_E}  H^{2d}_c( Y_G(V_{\pmb{r}}),\mathscr{V}^{\underline{k},{m}}_{\mathcal{O}_\mathfrak{m}} ) \otimes_P L_\mathfrak{m},\]
 showing that the right hand side is one dimensional over $L_\mathfrak{m}$. Then a generator $\omega^\varepsilon(f_E)^{\mathrm{alg}}$ of  $\mathbf{1}^\varepsilon_{\mathbf{f}_E}  H^{2d}_c( Y_G(V_{\pmb{r}}),\mathscr{V}^{\underline{k},{m}}_{\mathcal{O}_\mathfrak{m}} ) \otimes_P L_\mathfrak{m}$ satisfies the equality
\[
\omega^\varepsilon(f_E)^{\mathrm{alg}} = \frac{\mathbf{1}^\varepsilon i_p(\omega_J(W^*_{\mathfrak{c}p^{\pmb{r}+\pmb{1}}}f_E^{-\iota}))}{i_p(\Omega^\varepsilon(f_E))}.
 \]
If $f_E$ is the base change of a form $f$ from $F$, then the period $\Omega^\varepsilon(f_E)$  differs from the product $\Omega^+(f)\Omega^-(f)$ by an algebraic non-zero number, where $\Omega^+(f)$  (resp. $\Omega^-(f)$) corresponds to the choice character $\varepsilon$ of $\left\{ \matrix{\pm 1}{0}{0}{1} \right\}^{\Sigma_F}$ such that
 \[ \varepsilon(\matrix{- 1}{0}{0}{1})_\sigma =1 \]
 (resp.
 \[
 \varepsilon(\matrix{- 1}{0}{0}{1})_\sigma =-1 ),
 \]
  for all $\sigma \in \Sigma_F$. Moreover, the ratio should be a $p$-adic unit, \cite[Conjecture 5.1]{Hidanoncritical}, see \cite{TilouineUrban} for the state of the art on this conjecture.

We shall compare families arising from our Iwasawa cohomology and those defined previously in this section.

\begin{lemma}\label{lemma:compareHeckeAlg}
 Let $f$ be a primitive, nearly-ordinary and base-change Hilbert modular cusp form of weight $(\hat{\underline{k}},0)$ and level $V_{\pmb{r}}$. 
Let ${\mathfrak{m}_f}$ be the maximal ideal of the abstract Hecke algebra corresponding to the $\overline{\mathbb{F}}_p$-valued system of eigenvalues of $f$. Suppose the image of the residual Galois representation of $f$ is not solvable.
 Let $\mathbf{f}$ be the unique base change primitive Hida family passing through it valued in $\mathcal{K}_{\mathbf{f}} : = {\rm Frac}(\mathbf{I}_\mathbf{f})$, with $\mathbf{I}_\mathbf{f}$ finite flat over an irreducible component of $\Lambda_{G/H}$, and denote by $\mathbf{1}_{\mathbf{f}}$ its idempotent.  \begin{enumerate}
     \item We have an isomorphism of $\mathcal{K}_{\mathbf{f}} $-modules
  \[
 \mathbf{1}_{\mathbf{f}}   \mathcal{H}^{G/H}_{11} \otimes_{\Lambda_{G/H}} \mathcal{K}_{\mathbf{f}}   \cong 
 \mathbf{1}_{\mathbf{f}} H^{2d}_{\rm Iw}(Y_G(V_\infty), \mathcal{O}_{\mathfrak{m}})\otimes_{\Lambda_{G/H}} \mathcal{K}_{\mathbf{f}}.
 \]
 \item The isomorphism is equivariant for the action of $\left\{ \matrix{\pm 1}{0}{0}{1} \right\}^{\Sigma_E}$. Moreover, for every character $\epsilon$ on $\left\{ \matrix{\pm 1}{0}{0}{1} \right\}^{\Sigma_E}$,  $\mathbf{1}^\varepsilon_{\mathbf{f}} H^{2d}_{\rm Iw}(Y_G(V_\infty), \mathcal{O}_{\mathfrak{m}})\otimes_{\Lambda_{G/H}} \mathcal{K}_{\mathbf{f}}$ is one dimensional over $\mathcal{K}_{\mathbf{f}}$.
 \end{enumerate}
\end{lemma}
\begin{proof}
First, note that by Control Theorem \ref{ControlTheorem}, there is an irreducible component of $\mathcal{H}_{\rm Iw }\otimes_{\Lambda_{G/H}} \mathrm{Frac}(\Lambda_{G/H})$
that shares the same Hecke eigenvalues with $\mathbf{f}$ on a dense set of points. By a slight abuse of notation, we also denote by $ \mathbf{1}_{\mathbf{f}}$ the idempotent in $\mathcal{H}_{\rm Iw } \otimes_{\Lambda_{G/H}} \mathrm{Frac}(\Lambda_{G/H})$ corresponding to this irreducible component.

By \cite[Theorem A and Remark 1.1.1 (3)]{CarTam} that is, using an argument with the Hochschild--Serre spectral sequence if $\underline{k} \neq \underline{0}$, the cohomology of the ${\mathfrak{m}_f}$-isotypical component is concentrated only in middle degree. 
Using this fact, one can extend \cite[Theorem 4.4 and 6.6]{DimitrovENS} to allow ${\mathfrak{m}_f}$ in these theorems. 
We can then use \cite[Theorem 3.8]{DimitrovAJM} (or rather its proof, see the references therein)
to get 
\begin{align*}\label{DimitrovControl}
   \mathcal{H}^{G}_{11}(\hat{\underline{k}},0)_{\mathfrak{m}_f} \simeq  \varprojlim_{r \geq 1} e^{\rm n.ord} H^{2d}( Y_G(K_{11}(\mathfrak{c}p^{r})),\mathscr{V}^{\hat{\underline{k}},0}_{\mathcal{O}_\mathfrak{m}} )_{\mathfrak{m}_f}.
\end{align*}
This and the map of the inverse system induced by projection from level $K_{11}(\mathfrak{c}p^{{\pmb{r}}})$ to $V_{{{\pmb{r}}}}$ induce an isomorphism 
\[ \mathcal{H}^{G/H}_{11}(\hat{\underline{k}},0)_{\mathfrak{m}_f} \simeq  \varprojlim_{r \geq 1} e^{\rm n.ord} H^{2d}( Y_G(V_r),\mathscr{V}^{\hat{\underline{k}},0}_{\mathcal{O}_\mathfrak{m}} )_{\mathfrak{m}_f},  \]
which extends to a neighbourhood of $f$. We then obtain the desired isomorphism after tensoring by  $\mathcal{K}_\mathbf{f}$, completing the proof of (1).

The equivariance for the action of $\left\{ \matrix{\pm 1}{0}{0}{1} \right\}^{\Sigma_E}$ follows from the fact that the action of the Hecke algebra at infinity commutes with the ordinary idempotent $e^{\rm n.ord}$. 
Moreover, a verbatim translation of the proof of parts (b) and (c) of \cite[Theorem 4.7]{ShethControl}  shows that the moment map ${\rm mom}_{\pmb{r}}^{\underline{k}}$ gives an isomorphism (up to enlarging $\mathcal{O}_\mathfrak{m}$ so that it contains the Hecke eigenvalues of $f$)
\[ e^{\rm n.ord} H^{2d}_{\rm Iw}( Y_G(V_\infty),\mathcal{O}_\mathfrak{m} )_{\mathfrak{m}_f} \otimes_{\Lambda_{G/H,\pmb{r}}} \mathcal{O}_\mathfrak{m} v_{\hat{\underline{k}}}^{\rm hw} \simeq    e^{\rm n.ord} H^{2d}( Y_G(V_{\pmb{r}}),\mathscr{V}^{\hat{\underline{k}},0}_{\mathcal{O}_\mathfrak{m}} )_{\mathfrak{m}_f}. \]
This implies that  $\mathbf{1}_{\mathbf{f}} H^{2d}_{\rm Iw}(Y_G(V_\infty), \mathcal{O}_{\mathfrak{m}})\otimes_{\Lambda_{G/H}} \mathcal{K}_{\mathbf{f}}$ is finite dimensional over $\mathcal{K}_{\mathbf{f}}$. Furthermore, if $\varepsilon$ denotes any character of $\left\{ \matrix{\pm 1}{0}{0}{1} \right\}^{\Sigma_E}$ and if $P:  \mathcal{H}_{11}^{G/H} \to L_\mathfrak{m}$ is the point of $\mathbf{f}$ corresponding to the primitive form $f$ of level $V_{\pmb{r}}$, we then have \begin{align*}
    \mathbf{1}^\varepsilon_{\mathbf{f}} H^{2d}_{\rm Iw}(Y_G(V_\infty), \mathcal{O}_{\mathfrak{m}})\otimes_{P} L_{\mathfrak{m}} &\simeq \mathbf{1}^\varepsilon_{\mathbf{f}}H^{2d}( Y_G(V_{\pmb{r}}),\mathscr{V}^{\hat{\underline{k}},{0}}_{  \mathcal{O}_{\mathfrak{m}}})\otimes_{P} L_{\mathfrak{m}} \\ 
  &\simeq  \mathbf{1}^\varepsilon_{\mathbf{f}}H^{2d}_c( Y_G(V_{\pmb{r}}),\mathscr{V}^{\hat{\underline{k}},{0}}_{  \mathcal{O}_{\mathfrak{m}}})\otimes_{P} L_{\mathfrak{m}}.
\end{align*}
As $f$ is primitive, by Lemma \ref{FromGamma11toV} the latter is one dimensional. This proves (2).
\end{proof}

Lemma \ref{lemma:compareHeckeAlg} allows us to relate the chosen generator $\omega^\varepsilon(f)^{\mathrm{alg}}$ of $\mathbf{1}^\varepsilon_{\mathbf{f}}H^{2d}_c( Y_G(V_{\pmb{r}}),\mathscr{V}^{\hat{\underline{k}},{0}}_{  \mathcal{O}_{\mathfrak{m}}})\otimes_{P} L_{\mathfrak{m}}$ with the specialization of a generator of $\mathbf{1}^\varepsilon_{\mathbf{f}} H^{2d}_{\rm Iw}(Y_G(V_\infty), \mathcal{O}_{\mathfrak{m}})\otimes_{\Lambda_{G/H}} \mathcal{K}_{\mathbf{f}}$. Fix a generator $G_\mathbf{f}$ so that when evaluating at the point $P$ corresponding to $f$, it will give the following cohomology class 
\begin{align}\label{SettingaGen}
    G_\mathbf{f}[P] =  \frac{\mathbf{1}^\varepsilon \iota_p(\omega_J(W^*_{\mathfrak{c}p^\alpha}f^{-\iota}))}{\iota_p(\Omega^\varepsilon_p(f))} \in \mathbf{1}^\varepsilon_{\mathbf{f}}H^{2d}_c( Y_G(V_{\pmb{r}}),\mathscr{V}^{\hat{\underline{k}},{0}}_{  \mathcal{O}_{\mathfrak{m}}})\otimes_{P} L_{\mathfrak{m}}
\end{align}
A priori, $G_\mathbf{f}[P]$ and $\omega^\varepsilon(f)^{\mathrm{alg}}$ are not equal, but the periods $\Omega^\varepsilon_p(f)$ and $\Omega^\varepsilon(f)$ differ by a non-zero element of $L_\mathfrak{m}$.

\section{Adjoint \texorpdfstring{$p$}{p}-adic \texorpdfstring{$L$}{L}-functions}

\subsection{Rankin--Selberg integrals}\label{sec:RSintegrals}

In this section, we introduce the auxiliary analytic calculations used to determine the Euler factor at $p$ appearing in the interpolation formula of our $p$-adic $L$-function.

Let $J$ be an $F$-type for $E$ and define $$w_J = (\gamma_\sigma)_{\sigma \in \Sigma_E} \in G(\mathbb{R}), \text{ where }  \gamma_\sigma = \begin{cases} \matrix{- 1}{0}{0}{1} & \text{ if } \sigma \in J \\ \matrix{1}{0}{0}{1} & \text{ if } \sigma \not\in J. \end{cases}$$ We denote by $\iota_J : G(\mathbb{A}) \to G(\mathbb{A})$ the map given by right multiplication by $w_J$. Note that, as explained in \cite[Section 6.11]{GetzGore}, $\iota_J$ induces the map on $G(\mathbb{R})/ K_\infty \to G(\mathbb{R})/ K_\infty,\,  (z_\sigma) \mapsto (z_\sigma')$, where $$z_\sigma' = \begin{cases} z_\sigma & \text{ if } \sigma \in J \\ \overline{z}_\sigma & \text{ if } \sigma \not\in J. \end{cases}$$ 
Let $\chi_F$ be a unitary Hecke character of $\A^\times_F$. Denote $\chi : = \chi_F \circ {\rm N}_{E/F}$ and let $\theta$ be a unitary Hecke character of $\mathbb{A}^\times_E$ such that its restriction to $\mathbb{A}^\times_F$ is $\chi_F \eta_{E/F} $. Let $\kappa = (\underline{k},0)$ be a weight for $F$ and let $\hat{\kappa}= (\hat{\underline{k}},0)$ be the corresponding weight for $E$. For each $g \in S_{\hat{\kappa}}(K_0(\mathfrak{c}),\chi)$ and $s \in \C$ with ${\rm Re}(s)$ sufficiently large, let 
\[ I(g,s) := \int_{F^\times_+ \backslash \A^\times_{F,+}} \int_{F  \backslash \A_{F}} ( g \otimes \theta^{-1} ) (\iota_J \left(\begin{smallmatrix} y & x \\ 0 & 1 \end{smallmatrix} \right) ) |y|^s_{\A_F} dx d^\times y. \]
The integral $I(g,s)$, which is the key analytic input of \cite[Theorem 10.1]{GetzGore}, builds a bridge between the Asai $L$-function of $g$ twisted by $\theta^{-1}$ and the Hirzebruch--Zagier cycles, as it calculates the former and its residue at $s=0$ is related to the value of the cohomological pairing between the cycles and $\omega_J(g^{-\iota})$. Recall that if $\mathfrak{b}$ is the conductor of  $\theta$, we denote by  $\mathfrak{c}'=\mathfrak{c}\mathfrak{b}^2$ and $\mathfrak{c}'_F$ its intersection with $\mathcal{O}_F$.

\begin{proposition}\label{AsaiIR} If $g$ is a simultaneous eigenform for all Hecke operators,
\[ \zeta^{\mathfrak{c}'_F}_F(2s+2) I(g,s) = C(s) \cdot L({\rm As}(g \otimes \theta^{-1}),s+1),\]   where \[C(s) := \frac{|D_F \Norm_{F/\Q}(D_{E/F})^{1/2}|^{s+2} |D_F|^{1/2} \prod_{\sigma \in \Sigma_F} \Gamma(k_\sigma + 2 + s)}{ (4 \pi)^{\sum_{\sigma \in \Sigma_F} (k_\sigma + 2 + s) }}.\]
\end{proposition}
\begin{proof}
    This is \cite[Lemma 10.3]{GetzGore} together with \eqref{DefAsaiLF}.
\end{proof}

We now let $B'$ be the mirabolic subgroup of $H$, i.e. the subgroup of the upper-triangular Borel of $H$ given by matrices $\{ \left(\begin{smallmatrix} y & x \\ 0 & 1 \end{smallmatrix} \right)  \}$. Also $H(\R)^+$ be the identity component of $H(\R)$ and let $B'(\A)^+$, resp. $B'(\Q)^+$, denote the intersection of $B'(\A)$, resp. $B'(\Q)$, with $H(\A)^+ = H(\R)^+ H(\A_f)$. Then, as the left Haar measure on $B'(\A)$ corresponds to $|y|^{-1}_{\A_F} dx d^\times y$, the integral can be written as \begin{align}
    I(g,s) &= \int_{B'(\Q)+ \backslash B'(\A)^+} ( g \otimes \theta^{-1} ) (\iota_J( b) ) |{\rm det}(b)|^{s+1}_{\A_F} d b \nonumber \\ 
    &= \int_{B'(\Q)+ \backslash B'(\A)^+} ( g^{-\iota} \otimes \theta^{-1} ) (\iota_J( b) ) |{\rm det}(b)|^{s+1}_{\A_F} d b, \label{eq:integralwithtwisting}
\end{align}
where in the second equality we have used that, since $(\chi \theta^{-2}){|_{\A^\times_F}} = \chi_F^2 \chi_F^{-2} \eta_{E/F}^{-2} = 1$, 
\[ (g^{-\iota} \otimes \theta^{-1}) (\iota_J(b)) = (g \otimes \theta^{-1} )({\rm det}(\iota_J(b))^{-1}\iota_J(b)) = \chi^{-1} \theta^2(({\rm det}(\iota_J(b)))  (g \otimes \theta^{-1} )(\iota_J(b)).  \]
Let $\mathcal{N} : H(\A) \to \A_F^\times$ be the character $\mathcal{N}(h) =
    y $ if $ h = b z m$ with $b = \left(\begin{smallmatrix} y & \star \\ 0 & 1 \end{smallmatrix} \right) \in B'(\A)$, $z \in Z_H(\A)$, and $m \in K_0(\mathfrak{c}_F') K^+_{\infty,H}$ and trivial otherwise and define the (normalized) Eisenstein series 
    \begin{align}\label{definition_Eisenstein_series}
        \mathcal{E}^*(h,s) := \zeta_F^{\mathfrak{c}_F'}(2s) \cdot \!\!\!\! \!\!\!\! \sum_{\gamma \in\mathcal{O}_F^\times B'(\Q) \backslash H(\Q)} \!\!\!\! \!\!\!\!\left | \mathcal{N}(\gamma h) \right |_{\A_F}^s.
        \end{align}
$\mathcal{E}^*(h,s)$ is absolutely convergent for ${\rm Re}(s) > 1$ and can be meromorphically continued to the whole complex plane with a simple pole at $s=1$ (\textit{cf}. \cite[\S 10.2]{GetzGore} or \cite[\S 6]{HidaFourier}). Following \cite[Appendix A]{Hidanoncritical}, its residue at $s=1$ is explicitly \begin{align}\label{RES1}
        {\rm Res}_{s=1}\mathcal{E}^*(h,s) =  | (\mathcal{O}_F/ \mathfrak{c}_F')^\times | \frac{(2 \pi)^d R_F}{2^2 D_F {\rm N}_{F/\Q}(\mathfrak{c}_F')^2} =: C_\mathcal{E}, 
    \end{align}\index{$C_\mathcal{E}$}
where recall that $R_F$ is the regulator of $F$. In the following Proposition, set $\mathfrak{c}$ to be stable under $\tau \in {\rm Gal}(E/F)$.

\begin{proposition}\label{FromINTtoRES}
If $g$ is a simultaneous eigenform for all Hecke operators, \begin{align*}\langle [\omega_J(g^{-\iota} \otimes \theta^{-1})], \mathscr{Z}^{\hat{\underline{k}}, \chi^{-1} \theta^{2}}(\mathfrak{c}') \rangle = & \frac{C(0)}{C_\mathcal{E}} {\rm Res}_{s=1}L({\rm As}(g \otimes \theta^{-1}),s) \\ = & \frac{\zeta^{\mathfrak{c}'_F}_F(2)}{C_\mathcal{E}} {\rm Res}_{s=0} I(g,s). 
\end{align*}   
In particular,
\begin{align*}
\langle [\omega_J(g^{-\iota})], \mathscr{Z}^{\hat{\underline{k}}}(\mathfrak{c})_{ \theta} \rangle = & \frac{[K_0(\mathfrak{c}'):\mathcal{O}_E^{\ast, +}K_{11}(\mathfrak{c}')] G(\theta) C(0)}{C_\mathcal{E}} {\rm Res}_{s=1}L({\rm As}(g \otimes \theta^{-1}),s). \end{align*} 
\end{proposition}
\begin{proof}
This statement is contained in the proof of \cite[Theorem 10.1]{GetzGore} and we sketch it for the convenience of the reader. Recall that we have a decomposition $H(\A)^+ = H(\Q)^+ B'(\A)^+ K_0(\mathfrak{c}_F') K^+_{\infty,H}$. Using \eqref{eq:integralwithtwisting}, the invariance of $g^{-\iota}$ by right-multiplication of $K_0(\mathfrak{c}_F')$, and unfolding $\mathcal{E}^*(h,s)$, we have (\textit{cf}. \cite[Proposition 10.4]{GetzGore})
\[ \zeta^{\mathfrak{c}'_F}_F(2s+2) I(g,s) = \int_{Y_H(K_0(\mathfrak{c}_F'))}\!\!\!\!\!\!\!\!\!\!\!\! \!\!\!\!\!\!\!\!\!\! \mathcal{N}(\iota_J(h_\infty))^{\underline{k} + \mathbf{2}}\mathcal{E}^*(h,s+1)   {\rm det} (\iota_J(h_\infty))^{- \frac{\hat{\underline{k}}}{2} - \mathbf{1}} j(\iota_J(h_\infty),\mathbf{i})^{\hat{\underline{k}} + \mathbf{2}}(g^{-\iota} \otimes \theta^{-1}) (\iota_J(h))d h,  \] 
with $j(\iota_J(h_\infty),\mathbf{i})$ denoting the automorphy factor for $E$ defined as in \cite[p. 57]{GetzGore} and $d h = d h_\infty d h_f$ is the Haar measure chosen so that $d h_f$ assigns volume 1 to $K_0(\mathfrak{c}'_F)$, while invariant measure $d h_\infty$ on $H(\R)$ gives measure one to $K_{\infty,H}^+$. By \cite[Proposition 9.2]{GetzGore} and \eqref{RES1}, the residue at $s=0$ of the right hand side equals to 
\[ C_\mathcal{E} \cdot \int_{Y_H(K_0(\mathfrak{c}_F'))} \iota^{\hat{\underline{k}},*}( \omega_J(g^{-\iota} \otimes \theta^{-1})). \]
This and Proposition \ref{AsaiIR} give \[\int_{Y_H(K_0(\mathfrak{c}_F'))} \iota^{\hat{\underline{k}},*}( \omega_J(g^{-\iota} \otimes \theta^{-1})) = C_1 \cdot {\rm Res}_{s=1}L({\rm As}(g \otimes \theta^{-1}),s) ,  \]
where \[C_1 = \frac{C(0)}{C_\mathcal{E}} = \frac{|D_F \Norm_{F/\Q}(D_{E/F})^{1/2}|^{2} |D_F|^{1/2} \prod_{\sigma \in \Sigma_F} \Gamma(k_\sigma + 2 )}{C_\mathcal{E} (4 \pi)^{\sum_{\sigma \in \Sigma_F} (k_\sigma + 2) }}.\]
The result then follows from Lemma \ref{Pairingequalintegral}.
\end{proof}

Now let $f_E \in S_{\hat{\kappa}}(K_0(\mathfrak{c}),\chi)$ be an eigenform which is the base change to $E$ of a Hilbert cuspidal eigenform $f$.  Recall that if $f$ is nearly ordinary at $\mathfrak{p}$, then $f_E$ is nearly ordinary at primes of $E$ above $\mathfrak{p}$.  

\begin{lemma}\label{RSintegralpstab}
 For each $\frakp \mid p$, suppose that $\pi(f)_\frakp$ is unramified and that $f$ is nearly-ordinary at $\frakp$, with Satake parameters $\alpha_\frakp,\beta_\frakp$. Let $f_\alpha, f_{E,\alpha}$ be the corresponding $p$-stabilisations.
 Then 
 \[
 I(f_E,s) =  \left( \prod_{\frakp \mid p} \frac{1 + \Norm_{F/\Q}(\mathfrak{p})^{-s-1}}{(1-\eta_{E/F}(\mathfrak{p})\Norm_{F/\Q}(\mathfrak{p})^{-s-1}) 
(1-\eta_{E/F}(\mathfrak{p})\frac{\beta_{\mathfrak{p}}}{\alpha_{\mathfrak{p}}}\Norm_{F/\Q}(\mathfrak{p})^{-s-1})} \right) I(f_{E,\alpha},s). \]
In particular, 
\[ {\rm Res}_{s=0} I(f_E,s)=  \left( \prod_{\frakp \mid p} \frac{1 + \Norm_{F/\Q}(\mathfrak{p})^{-1}}{(1-\eta_{E/F}(\mathfrak{p})\Norm_{F/\Q}(\mathfrak{p})^{-1}) 
(1-\eta_{E/F}(\mathfrak{p})\frac{\beta_{\mathfrak{p}}}{\alpha_{\mathfrak{p}}}\Norm_{F/\Q}(\mathfrak{p})^{-1})} \right) {\rm Res}_{s=0} I(f_{E,\alpha},s) \]
\end{lemma}
\begin{proof}

By \cite[Lemma 10.3]{GetzGore}, we need to compare
\[
\sum_{\mathfrak{n} \subset \mathcal{O}_F} \frac{a(\mathfrak{n}\mathcal{O}_E,f_E \otimes \theta^{-1})}{\Norm_{F/\mathbb{Q}}(\mathfrak{n})^{s+2}}
\]
and 
\[
\sum_{\mathfrak{n} \subset \mathcal{O}_F} \frac{a(\mathfrak{n}\mathcal{O}_E,f_{E,\alpha}\otimes \theta^{-1})}{\Norm_{F/\mathbb{Q}}(\mathfrak{n})^{s+2}}.
\]
By the Euler product expansion in \cite[(10.2.2)]{GetzGore}\footnote{Note that in \cite[(10.2.2)]{GetzGore} and the line above there is a typo as the summation over $j$ starts from $j=0$ and not $j=1$.} of these two series, this reduces to a calculation at each prime $\mathfrak{p}$ of $F$ above $p$. Notice that the Euler factor at $\frakp$ of the latter series is \begin{align*}
    \sum_{j=0}^\infty a(\mathfrak{p}^j \mathcal{O}_E, f_{E,\alpha} \otimes \theta^{-1})\Norm_{F/\Q}(\mathfrak{p})^{-js-2j} &= \sum_{j=0}^\infty \alpha_{\frakp}^{2j} \theta^{-j}(\mathfrak{p}) \Norm_{F/\Q}(\mathfrak{p})^{-js-j} \\&= \sum_{j=0}^\infty \alpha_{\frakp}^{2j} (\chi_F \eta_{E/F})^{-j}(\mathfrak{p}) \Norm_{F/\Q}(\mathfrak{p})^{-js-j} \\ 
    &= (1 - \alpha_\frakp^2 (\chi_F \eta_{E/F})^{-1}(\frakp)\Norm_{F/\Q}(\mathfrak{p})^{-s-1})^{-1} \\ 
    &= (1 -\eta_{E/F}(\frakp) \tfrac{\alpha_\frakp}{\beta_\frakp} \Norm_{F/\Q}(\mathfrak{p})^{-s-1})^{-1}, 
\end{align*} 
where we have used that $a(\mathfrak{p}^j \mathcal{O}_E, f_{E,\alpha}) = \alpha_{\mathfrak{p} \mathcal{O}_E}( f_E)^j \Norm_{E/\Q}(\mathfrak{p})^{j/2}=\alpha_{\mathfrak{p}}^{2j}\Norm_{F/\Q}(\mathfrak{p})^j$ and that $\chi_F(\frakp) = \alpha_\frakp \beta_\frakp$. On the other hand, the Euler factor at $\frakp$ of the former power series, which is
\[ \sum_{j=0}^\infty a(\mathfrak{p}^j \mathcal{O}_E, f_{E} \otimes \theta^{-1})\Norm_{F/\Q}(\mathfrak{p})^{-js-2j}, \]
equals to (\cite[(10.2.5)]{GetzGore})
\[ \frac{L_\frakp(\mathrm{As}(f_E \otimes \theta^{-1}), s+1)}{\zeta_{F,\frakp}(2s+2)}. \]
Thanks to Proposition \ref{AsaiForBC}, this further reads as 
\[\frac{L_\frakp(\mathrm{Ad}(f) \otimes \eta_{E/F}, s+1)\zeta_{F,\frakp}(s+1)}{\zeta_{F,\frakp}(2s+2)}, \]
which is explicitly 
\[
\frac{1 + \Norm_{F/\Q}(\mathfrak{p})^{-s-1}}{(1-\eta_{E/F}(\mathfrak{p})\frac{\alpha_{\mathfrak{p}}}{\beta_{\mathfrak{p}}}\Norm_{F/\Q}(\mathfrak{p})^{-s-1})(1-\eta_{E/F}(\mathfrak{p})\Norm_{F/\Q}(\mathfrak{p})^{-s-1}) 
(1-\eta_{E/F}(\mathfrak{p})\frac{\beta_{\mathfrak{p}}}{\alpha_{\mathfrak{p}}}\Norm_{F/\Q}(\mathfrak{p})^{-s-1})}.
\]
We thus conclude by comparing the two factors at $\frakp$. The formula on the residues follows from this and from the fact that 
\begin{align}\label{eq7.3}
 \frac{1 + \Norm_{F/\Q}(\mathfrak{p})^{-s-1}}{(1-\eta_{E/F}(\mathfrak{p})\Norm_{F/\Q}(\mathfrak{p})^{-s-1}) 
(1-\eta_{E/F}(\mathfrak{p})\frac{\beta_{\mathfrak{p}}}{\alpha_{\mathfrak{p}}}\Norm_{F/\Q}(\mathfrak{p})^{-s-1})}\end{align}
has a pole at $s=0$ if and only if \[\beta_{\mathfrak{p}}= \eta_{E/F}(\mathfrak{p}) \alpha_{\mathfrak{p}}\Norm_{F/\Q}(\mathfrak{p}).\]
As $\frakp$ is spherical unramified for $\pi(f)$, $\alpha_{\mathfrak{p}}$ and $\beta_{\mathfrak{p}}$ are Weil numbers of the same weight (\textit{cf}. \cite[Theorem 1]{BlasiusRamanujan}), thus taking any complex absolute value we get a contradiction, implying that \eqref{eq7.3} has no pole at $s=0$.
\end{proof}

For any $\frakp \in \Upsilon_{F,p}$ and $r \geq 0$, let $\gamma_{\frakp,r} \in \GL_2(E_\frakp)$ be the matrix given by \[ \gamma_{\frakp,r}:= \eta_\frakp^{-r} u_\frakp \eta_\frakp^r = \begin{cases} \left( \begin{smallmatrix}
1 & \pi_\frakp^{-r} \delta_\frakp \\ & 1
\end{smallmatrix}\right) & \text{ if } \frakp \text{ is inert in }E, \\ \left( \left( \begin{smallmatrix}
1 & \pi_\frakp^{-r} \\ & 1
\end{smallmatrix}\right),\left( \begin{smallmatrix}
1 & \\ & 1 
\end{smallmatrix}\right)\right) & \text{ if } \frakp \text{ is split in }E. 
\end{cases}\]\index{$\gamma_{\frakp,r}$}
Given $\pmb{r} = (r_\frakp)_\frakp \in \Z_{\geq 0}^{|\Upsilon_{F,p}|}$, we also let $\gamma_{\pmb{r}}$\index{$\gamma_{\pmb{r}}$} be the matrix with $\frakp$-component $\gamma_{\frakp,r_\frakp}$ and $1$ elsewhere.

Let $f_{E,\alpha}$ be a nearly ordinary base change cusp form of weight $\hat{\kappa}$,  level $V_{\pmb{r}}$, and $U_\mathfrak{P}$-eigenvalue $\alpha_\mathfrak{P} {\rm N}_{E/\Q}(\mathfrak{P})^{1/2}$. This implies that the conductor of the $\mathfrak{P}$-part of the adjoint character $\chi'$ is at most $\mathfrak{P}^{r_\frakp}$ for $\mathfrak{P}$ lying above $\frakp$. In the following we can and do choose $\pmb{r}$ minimal so that  $\chi_\mathfrak{P}'$ is of conductor exactly $\mathfrak{P}^{r_\frakp}$. This means that $f_{E,\alpha} \in S_{\hat{\kappa}}(K_0(\mathfrak{c}p^{\pmb{r}+\pmb{1}}),\chi,\chi')$, with $\mathfrak{c}$ coprime with $p \mathcal{O}_E$.

\begin{lemma}\label{lemmaforactionofuonRSfinalversion} For any $\frakp \in \Upsilon_{F,p}$, suppose that the $\frakp$-component of $\chi_F$ is spherical and let $\mathfrak{P}$ be a prime of $E$ above $\frakp$.
Suppose that $f_{E,\alpha} $ has level $V_{\pmb{r}}$, with minimal $\pmb{r}$. \begin{enumerate}
     \item If $\pi(f_{E,\alpha})_{\mathfrak{P}}$ is either an unramified principal series or Steinberg, then 
 \[I(\gamma_{\frakp,r_\frakp} \cdot f_{E,\alpha} ,s) = I( f_{E,\alpha} ,s). \]
 \item If $\pi(f_{E,\alpha})_{\mathfrak{P}}$ is a ramified principal series, then \[ I(\gamma_{\frakp,r_\frakp} \cdot f_{E,\alpha} ,s) =G(\chi_{\mathfrak{P}}') I( f_{E,\alpha} ,s), \]
where $G(\chi_{\mathfrak{P}}')$ denotes the Gauss sum of the restriction to $\mathcal{O}_\frakp^\times$ of $\chi_{\mathfrak{P}}'$.
 \end{enumerate} 
 \end{lemma}

\begin{remark}
    Note that if both $\mathfrak{P}$ and $\mathfrak{Q}$ are primes of $E$ which lie above $\frakp$, then $\pi(f_{E,\alpha})_{\mathfrak{P}}$ and $\pi(f_{E,\alpha})_{\mathfrak{Q}}$ are of the same type (and indeed isomorphic).
\end{remark}
\begin{proof}[Proof of Lemma \ref{lemmaforactionofuonRSfinalversion}]
As in Lemma \ref{RSintegralpstab}, we need to compare the Euler factors at $\frakp$ of the two series given by \cite[Lemma 10.3]{GetzGore}. We prove the Lemma in the case when $\frakp$ is inert in $E$, with the split case following similarly. By abusing notation, we also call $\frakp$ the prime for $E$ above it. We distinguish two cases.

Suppose that $\pi(f_{E,\alpha})_\frakp$ is either an unramified principal series or Steinberg (which correspond to the cases (UnrPS) and (St) of \S \ref{sec:pstabilisation}). Then $r:= r_\frakp = 0$ and we notice that the Fourier coefficients of  $\gamma_{\frakp,0} \cdot f_{E,\alpha} = f_{E,\alpha}(\,\cdot \gamma_{\frakp,0})$ are the same as those of $f_{E,\alpha}$, hence $I(\gamma_{\frakp,0} \cdot f_{E,\alpha} ,s) = I(f_{E,\alpha} ,s)$. As $\frakp$ is inert in $E$, then, for  $\left(\begin{smallmatrix}
   y  & x  \\
     & 1 
\end{smallmatrix} \right) \in \GL_2(F_\frakp)$, we have that \[ \left(\begin{smallmatrix}
   y  & x  \\
     & 1 
\end{smallmatrix} \right)\gamma_{\frakp,0} = \left(\begin{smallmatrix}
   y  & x + y \delta_\frakp  \\
     & 1 
\end{smallmatrix} \right), \]
and $x + y \delta_\frakp $ is always integral when $x,y$ are, proving that the Fourier expansion \eqref{eq:FourierExpansion} of $\gamma_{\frakp,0} \cdot f_{E,\alpha}$ equals the one of $f_{E,\alpha}$, hence that the two Euler factors at $\frakp$ agree.

Suppose $\pi(f_{E,\alpha})_\frakp$ is a ramified principal series (which corresponds to the case (RamPS) of \S \ref{sec:pstabilisation})  with the induction of the characters $({\chi'_{\mathfrak{p}}}, {\chi_{\frakp} \chi'_{\mathfrak{p}}}^{-1})$, and note that by the definition of $V_{\pmb{r}}$  the restriction to $\mathcal{O}_{E_\frakp}^\times$ of the character $\chi'_{\mathfrak{p}}$ has conductor (at most) $\mathfrak{p}^r$, and by hypothesis we suppose that its conductor is exactly $\mathfrak{p}^r$ (otherwise we choose a smaller $r$ at the beginning). 

We calculate the local integral by making use of the explicit formulae for the Whittaker model of $\pi(f_{E,\alpha})_\frakp$. Indeed, by \cite[Eq. (D.5.9)]{GetzGore} the Fourier coefficients are calculated via the Whittaker functions is a tensor product of local Whittaker functions (see Theorem D.1 of {\it loc. cit.}).
Subsequently, we perform calculations inspired by \cite[\S 5.3]{HsiehAJM}.
Let $W^{\ord}_{f_{E,\alpha},\frakp}$ be the unique (up to scalar) ordinary vector in the local  Whittaker model at $\frakp$ of $\pi(f_{E,\alpha})_\frakp$, normalised as in \cite[Corollary 2.3]{HsiehAJM}. In particular, by \textit{loc. cit.}, if the local representation is a principal series associated with the induction of the characters $({\chi'_{\mathfrak{p}}}, {\chi'_{\mathfrak{p}}}^{-1}\chi_{\mathfrak{p}})$, the local Whittaker model satisfies, for $\left(\begin{smallmatrix}
   y  & x  \\
     & 1 
\end{smallmatrix} \right) \in \GL_2(E_\frakp), $
\[
W^{\ord}_{f_{E,\alpha},\frakp}\left( \left( \begin{smallmatrix}
   y  & x  \\
     & 1 
\end{smallmatrix} \right) \right)=e_\frakp(x)  W^{\ord}_{f_{E,\alpha},\frakp}\left( \left( \begin{smallmatrix}
   y  &   \\
     & 1 
\end{smallmatrix} \right) \right) =\begin{cases} e_\frakp(x)\chi'_{\mathfrak{p}}(y)|y|_{E_\frakp}^{1/2} & \text{if } y \in E_{\frakp}^{\times} \cap \mathcal{O}_{E_\frakp},\\ 0 & \text{otherwise,}
\end{cases}
\]
where $e_\frakp(-) $ is the local additive character defined in Section \ref{sec:Fourierexpansion}.
Recall that with the normalisation in Section \ref{sec:AdLfunc} above, in this case with $f_{E,\alpha}$ and $y=\pi_{\frakp}$, the value $\chi'(\pi_{\frakp})$ is $\alpha(f_{E,\alpha})_{1,\frakp}$ and $\alpha(f_{E,\alpha})_{1,\frakp}|\pi_{\frakp}|_{E_\frakp}^{1/2}$ is exactly equal to $a(\mathfrak{p} \mathcal{O}_E, f_{E,\alpha}) $. Also, notice that when we twist the form $f_{E,\alpha}$ by $\theta^{-1}$, the $\mathfrak{p}$-component of the associated automorphic representation is $\pi(f_{E,\alpha})_\frakp \otimes \theta^{-1}_\frakp$, which has Satake parameters $\theta^{-1}_\frakp(\pi_\frakp)\alpha(f_{E,\alpha})_{1,\frakp}$, $\theta^{-1}_\frakp(\pi_\frakp)\alpha(f_{E,\alpha})_{2,\frakp}$. Recall that, by hypothesis, the central character $\chi_\frakp$ of $\pi(f_{E,\alpha})_\frakp$ is unramified and the restriction to $F_\frakp^\times$ of $\theta_\frakp$ agrees with $\chi_{F,\frakp}\eta_{E_\frakp/F_\frakp}$, hence it is unramified as well.

We now evaluate the local integral at $\frakp$ given in the proof of \cite[Lemma 10.3]{GetzGore} of the Rankin--Selberg integral $I(\gamma_{\frakp,r} \cdot f_{E,\alpha} ,s)$.  Note that the integral is over $F_\frakp^\times$ but we do use the Whittaker model at $\frakp$ of the base change $f_{E,\alpha} \otimes \theta^{-1}$. 
All steps are the same as in the proof of \cite[Lemma 10.3]{GetzGore} till the top of page 162.  Consider an element $\left( \begin{smallmatrix}
   y  & x \\
    0 & 1 
\end{smallmatrix} \right) \in \GL_2(F_\frakp)$. We shall use the following equality: $
\left( \begin{smallmatrix}
   y  & x \\
    0 & 1 
\end{smallmatrix} \right) \gamma_{\frakp,r} = \left( \begin{smallmatrix}
   y  & y\pi_\frakp^{-r} \delta_\frakp + x \\
     0 & 1 
\end{smallmatrix} \right).$
To lighten notation, we let $g = f_{E,\alpha} \otimes \theta^{-1}$  in the following calculation:
\begin{align*}
  \int_{F_{\frakp}^{\times}} a(y,\gamma_{\frakp,r} \cdot g) |y|_{F_\frakp}^{s+2}{\textup{d}^{\times}y} =& \int_{F_{\frakp}^{\times}} W_{g,\frakp}^{\ord}(\left( \begin{smallmatrix}
   y  &   \\
     & 1 
\end{smallmatrix} \right)\gamma_{\frakp,r})  |y|_{F_\frakp}^{s+2}{\textup{d}^{\times}y} \\
= & \int_{F_{\frakp}^{\times}} W_{g,\frakp}^{\ord}(\left( \begin{smallmatrix}
   y  &   \\
     & 1 
\end{smallmatrix} \right)) e_\frakp(y\pi_\frakp^{-r} \delta_\frakp)  |y|_{F_\frakp}^{s+2}{\textup{d}^{\times}y}  \\
=& \int_{F_{\frakp}^{\times} \cap \mathcal{O}_\frakp} (\theta_\frakp^{-1}\chi'_\frakp)(y) |y|^{1/2}_{E_\frakp} e_\frakp(y\pi_\frakp^{-r} \delta_\frakp)|y|_{F_\frakp}^{s+2}{\textup{d}^{\times}y}\\
=& \sum_{n\geq 0} \eta_{E/F}^n(\frakp)(\chi_{F,\frakp}^{-1} {\chi'_{\frakp}})^n(\pi_\frakp)|\pi^n_\frakp|_{F_\frakp}^{s+3}   \int_{ \mathcal{O}^{\times}_\frakp}  {\chi'_{\frakp}}(y)  e_\frakp(y\pi_\frakp^{n-r} \delta_\frakp) {\textup{d}^{\times}y}, 
\end{align*} 
where we have used that the restriction to $F_\frakp^\times$ of $\theta_\frakp$ agrees with $\chi_{F,\frakp}\eta_{E_\frakp/F_\frakp}$.
Recall that $\pi(f_{E,\alpha})_\frakp$ is a principal series associated to the characters $(\chi_\frakp',{\chi'_{\mathfrak{p}}}^{-1}\chi_\frakp)$, with $\chi'_\frakp$ a character on $\mathcal{O}_\frakp^\times$ with exponent of the conductor equal to $r$. Recall also that, as $f_{E,\alpha}$ is a base change form, $\chi_\frakp' = \chi_{F,\frakp}' \circ {\rm N}_{E_\frakp/F_\frakp}$ and the exponent of the conductor of $\chi_{F,\frakp}'$ equals $r$ by \cite[Proposition E.9]{GetzGore}. Then the restriction to $\mathcal{O}_\frakp^\times$ of $\chi_\frakp'$ has conductor $\frakp^{r}$ and the inner integral is
\[\int_{ \mathcal{O}^{\times}_\frakp}  \chi_\frakp'(y)  e_\frakp(y\pi_\frakp^{n-r}  \delta_\frakp) {\textup{d}^{\times}y} = \begin{cases} G(\chi_\frakp')  & \text{if } n = 0 \\0 & \text{ if } n > 0, 
        \end{cases} \]
by for instance \cite[Lemma 7-4]{RamaValenza}. Plugging this in, we have that  
\begin{align*}
  \int_{F_{\frakp}^{\times}} a(y,\gamma_{\frakp,r} \cdot g) |y|_{F_\frakp}^{s+2}{\textup{d}^{\times}y} =&  G(\chi_\frakp'). 
\end{align*} 
Note that in this case the local $L$-factor at $\frakp$ is $1$ for  
$I(f_{E,\alpha},s)$.
 
\end{proof}

\subsection{A Hecke-equivariant pairing}\label{sec:Heckequivariantpairing}
We keep the notation used in \S \ref{sec:RSintegrals}. In particular, we let $\chi_F$ be a unitary Hecke character of $\A_F^\times$, $\chi := \chi_F \circ {\rm N}_{E/F}$ and let $\theta$ be a unitary Hecke character of $\A_E^\times$ such that its restriction to $\A_F^\times$ is $\chi_F \eta_{E/F}$. Set also $\mathfrak{c}$ to be coprime to $p \mathcal{O}_E$ and stable under $\tau \in {\rm Gal}(E/F)$.

Note that the pairing in intersection cohomology $\langle\,, \,\rangle_{IH}$ and the Poincar\'e pairing $\langle\,, \,\rangle$ introduced in \S \ref{sec:cohomologygroups} for the variety $Y_G(K_0(\mathfrak{c}p^{\pmb{r}+\pmb{1}}))$ are not Hecke-equivariant for $T_{\eta}$. Define 
\begin{align}\label{def:HeckePpairing}
    [  - , -] =\langle - , W^*_{\mathfrak{c}p^{\pmb{r}+\pmb{1}}} - \rangle.
\end{align}
To ease the notation, let $\mathscr{Z}_{\infty}$ denote the big cycle $\mathscr{Z}_{\infty}^{\chi^{-1}\theta^2}(\mathfrak{c}') \in e^{\rm n.ord} H^{2d}_{\rm Iw}(Y_G(V_\infty), \mathcal{O}_{\mathfrak{m}}(\chi^{-1}\theta^2))$, with $\chi^{-1}\theta^2$ supposed to be unramified at $p$. Similarly, the class at finite level $V_{{\pmb{r}}}$ will be denoted by $\mathscr{Z}_{\pmb{r}}$.  In what follows, to highlight its level, we also denote by $\mathcal{E}_{\mathfrak{c}'p^{\pmb{r}+\pmb{1}}}^*(h,s)$ the normalized Eisenstein series of level $K_0(\mathfrak{c}_F'p^{\pmb{r}+\pmb{1}})$ introduced in \eqref{definition_Eisenstein_series}.

\begin{lemma}\label{lemma:integraleval}
 For $\pmb{r} = (r_\frakp)_\frakp \in \Z_{\geq 0}^{|\Upsilon_{F,p}|}$, we let $f_{E,\alpha} \in S_{(\underline{k},0)}(K_0(\mathfrak{c}p^{\pmb{r}+\pmb{1}}),\chi,\chi')$ be a nearly ordinary form such that it has level $V_{\pmb{r}}$ and it is the base change of $f_\alpha$. We suppose that the Nebentypus $\chi_F$ of $f_\alpha$ has trivial component at $\frakp$ for any $\frakp \in \Upsilon_{F,p}$ and that $\chi^{-1}\theta^2$ is unramified at $p$. Set $f = f_{E,\alpha} \otimes \theta^{-1}$ and let $\gamma_{\pmb{r}}$ be the matrix with $\frakp$-component $\gamma_{\frakp,r_\frakp}$ and $1$ elsewhere.  \begin{enumerate}
     \item If $r_\frakp \geq 1$ for all $\frakp \in \Upsilon_{F,p}$, we have 
\[
[  \omega_J(W_{\mathfrak{c}p^{\pmb{r}+\pmb{1}}}^* f^{-\iota}), \mathrm{mom}^{\underline{k}}_{\pmb{r}}\, \mathscr{Z}_{\infty}] = \frac{1}{{\alpha^{\circ}_p}^{2\pmb{r}}}
\frac{\zeta^{p\mathfrak{c}'_F}_F(2)}{a_{\underline{k}} C_{\mathcal{E},\pmb{r}}} {\rm Res}_{s=0} I( \gamma_{\pmb{r}} \cdot f_{E,\alpha}, s), \] where $a_{\underline{k}}$ is the constant of Lemma \ref{rmk:uflag}, $C_{\mathcal{E},\pmb{r}} = \mathrm{Res}_{s=1}\mathcal{E}_{\mathfrak{c}'p^{\pmb{r}+\pmb{1}}}^*(h,s)$, and ${\alpha^{\circ}_p}^{2\pmb{r}} = \prod_\frakp {\alpha^{\circ}_\frakp}^{2 r_\frakp}$, with $\alpha^{\circ}_\frakp$ the $U_{\frakp}^{^{\nord}}$-eigenvalue of $f_{\alpha}$.
\item If $\pmb{r} = \pmb{0}$, then we have \[
[  \omega_J(W_{\mathfrak{c}p}^*\, f^{-\iota}), \mathrm{mom}^{\underline{k}}_{0}\, \mathscr{Z}_{\infty}] = \mathcal{E}^{\pmb{0}}_{p}(f_\alpha)
\frac{\zeta^{p \mathfrak{c}'_F}_F(2)}{a_{\underline{k}} C_{\mathcal{E},\pmb{0}}} {\rm Res}_{s=0}  I( \gamma_{\pmb{0}} \cdot f_{E,\alpha}, s), \]
with $\mathcal{E}^{\pmb{0}}_{p}(f_\alpha) : =  \prod_{\frakp \in \Upsilon_{F,p}}  \left( 1 - p^{\sum_{\sigma \in \Sigma_{\frakp}}  (1+k_{\sigma})}\,\eta_{E/F}(\frakp) {\alpha^{\circ}_{\frakp}}^{-2} \right) $\index{$\mathcal{E}^{\pmb{0}}_{p}(f_\alpha)$}.
 \end{enumerate} 
\end{lemma}
\begin{proof}
We suppose for simplicity that $\pmb{r} = (r)_\frakp$ and that $\chi^{-1}\theta^2$ is trivial, with the general case having identical proof but considerably heavier notation. 

Recall that, for $r\geq 1$,
\[\mathrm{mom}^{\underline{k}}_{r}\, \mathscr{Z}_{\infty}=T_{\eta}^{-r} e^{\nord} \,\mathscr{Z}_r^{\underline{k}} = T_{\eta}^{-r} e^{\nord} ((\eta^r)_\ast z_r^{\underline{k}}),
\]
with 
\[
z_r^{\underline{k}}= u_\star \circ {\iota_r}_\star \circ  \iota^{\dagger,  \underline{k}}_{\mathcal{O}_{\mathfrak{m}}}(\mathbf{1}_{Y_H(u K_{r}u^{-1} \cap H)}) \in H^{2d} ( Y_G( K_{r}), \mathscr{V}^{\underline{k},0}_{\mathcal{O}_{\mathfrak{m}}}).
\]
As $T_\eta$ is self-adjoint for the pairing $[\,,\,]$, we have
\begin{align*}
[  \omega_J(W_{\mathfrak{c}p^{r+1}}^* f^{-\iota}), \mathrm{mom}^{\underline{k}}_{r}\, \mathscr{Z}_{\infty}] &= \langle \omega_J(W_{\mathfrak{c}p^{r+1}}^* f^{-\iota}),W_{\mathfrak{c}p^{r+1}}^* T_{\eta}^{-r} e^{\nord}(\eta^r)_\ast u_\star \circ {\iota_r}_\star \circ  \iota^{\dagger,  \underline{k}}_{\mathcal{O}_{\mathfrak{m}}}(\mathbf{1}_{Y_H(u K_{r}u^{-1} \cap H)}) \rangle \\
 &=\langle   T_\eta^{-r} e^{\nord}  \omega_J(W_{\mathfrak{c}p^{r+1}}^* f^{-\iota}), W_{\mathfrak{c}p^{r+1}}^*(\eta^r)_\ast u_\star \circ {\iota_r}_\star \circ  \iota^{\dagger,  \underline{k}}_{\mathcal{O}_{\mathfrak{m}}}(\mathbf{1}_{Y_H(u K_{r}u^{-1} \cap H)}) \rangle
\\
\text{(Proposition \ref{prop:TetaUp})} &= \langle \omega_J(W_{\mathfrak{c}p^{r+1}}^* (e^{\nord}  { U_p^{^{\nord}}}^{-r}  f)^{-\iota}), W_{\mathfrak{c}p^{r+1}}^*(\eta^r)_\ast u_\star \circ {\iota_r}_\star \circ  \iota^{\dagger,  \underline{k}}_{\mathcal{O}_{\mathfrak{m}}}(\mathbf{1}_{Y_H(u K_{r}u^{-1} \cap H)}) \rangle \\
&= \frac{1}{{\alpha_p^{\circ}}^{2r}} \langle    \omega_J(W_{\mathfrak{c}p^{r+1}}^* f^{-\iota}), W_{\mathfrak{c}p^{r+1}}^*(\eta^r)_\ast u_\star \circ {\iota_r}_\star \circ  \iota^{\dagger,  \underline{k}}_{\mathcal{O}_{\mathfrak{m}}}(\mathbf{1}_{Y_H(u K_{r}u^{-1} \cap H)}) \rangle \\
&= \frac{1}{{\alpha_p^{\circ}}^{2r}} \langle    \omega_J(W_{\mathfrak{c}p^{r+1}}^* f^{-\iota}), W_{\mathfrak{c}p^{r+1}}^*\gamma_{r,\ast} (\eta^r)_\ast \circ {\iota_r}_\star \circ \iota^{\dagger,  \underline{k}}_{\mathcal{O}_{\mathfrak{m}}}(\mathbf{1}_{Y_H(u K_{r}u^{-1} \cap H)}) \rangle \\ 
&= \frac{1}{{\alpha_p^{\circ}}^{2r}} \langle    \omega_J(W_{\mathfrak{c}p^{r+1}}^* f^{-\iota}), ({}^t u^{-1})_\ast W_{\mathfrak{c}p^{r+1}}^* (\eta^r)_\ast \circ {\iota_r}_\star \circ \iota^{\dagger,  \underline{k}}_{\mathcal{O}_{\mathfrak{m}}}(\mathbf{1}_{Y_H(u K_{r}u^{-1} \cap H)}) \rangle \\ 
&= \frac{1}{{\alpha_p^{\circ}}^{2r}} \langle    ({}^t u^{-1})^\ast \omega_J(W_{\mathfrak{c}p^{r+1}}^* f^{-\iota}), W_{\mathfrak{c}p^{r+1}}^* (\eta^r)_\ast \circ {\iota_r}_\star \circ \iota^{\dagger,  \underline{k}}_{\mathcal{O}_{\mathfrak{m}}}(\mathbf{1}_{Y_H(u K_{r}u^{-1} \cap H)}) \rangle \\ 
&= \frac{1}{{\alpha_p^{\circ}}^{2r}} \langle \omega_J(W_{\mathfrak{c}p^{r+1}}^*( \gamma_r \cdot f)^{-\iota}), W_{\mathfrak{c}p^{r+1}}^* (\eta^r)_\ast \circ {\iota_r}_\star \circ \iota^{\dagger,  \underline{k}}_{\mathcal{O}_{\mathfrak{m}}}(\mathbf{1}_{Y_H(u K_{r}u^{-1} \cap H)}) \rangle
\end{align*}
where we have used that the $U_{\frakp}^{^{\nord}}$-eigenvalue of $f_{E,\alpha}$ is the square of the $U_{\frakp}^{^{\nord}}$-eigenvalue of $f_{\alpha}$, and where, in the sixth equality, we have used the adjointness of pushforward and pullback for Poincar\'e duality.
By the proof of Lemma 7.8 in \cite{GetzGore}\footnote{We observe that the power $[k+2m]/2$ in \emph{loc. cit. } corresponds to our $m$.} and the compatibility between Petersson product and Poincar\'e pairing, we have 
\[
\langle W_{\mathfrak{c}p^{r+1}}^* \clubsuit
, W_{\mathfrak{c}p^{r+1}}^* \diamondsuit \rangle = \mathrm{N}_{E/\mathbb{Q}}(\mathfrak{c}p^r)^m \langle  \clubsuit
,  \diamondsuit \rangle= \langle  \clubsuit
,  \diamondsuit \rangle,
\]
as we have supposed $m=0$. Combining this with the commutative diagram \eqref{commdiagAL} we are left to evaluate
\begin{align}\label{intermediateLemma73}
    \langle \omega_J(( \gamma_r \cdot f)^{-\iota}),  (\eta^r)_\ast \circ {\iota_r}_\star \circ \iota^{\dagger,  \underline{k}}_{\mathcal{O}_{\mathfrak{m}}}(\mathbf{1}_{Y_H(u K_{r} u^{-1} \cap H)}) \rangle
\end{align} 

Now, recall that as $\eta \in H(\Q_p)$, the maps $(\eta^r)_\ast$ and ${\iota_r}_\star$ commute. Then $(\eta^r)_\ast$ defines an isomorphism between $\mathbf{1}_{Y_H(u K_{r}u^{-1} \cap H)}$ and $\mathbf{1}_{Y_H(\gamma_r V_{r} \gamma_r^{-1}  \cap H)}$ and 
\begin{align*}
    \text{\eqref{intermediateLemma73}} &= \int_{Y_H(\gamma_r V_{r} \gamma_r^{-1}  \cap H)} \iota^{\dagger,\underline{k},\star} \left ( \omega_J( \gamma_r \cdot f^{-\iota}) \right )  \\
     \text{(Proposition \ref{FromINTtoRES})} &= \frac{\zeta^{p\mathfrak{c}'_F}_F(2)}{a_{\underline{k}}C_{\mathcal{E},\pmb{r}}} {\rm Res}_{s=0} I( \gamma_r \cdot f, s),
\end{align*}
where $a_{\underline{k}}$ is the constant of Lemma \ref{rmk:uflag} used to re-normalize the branching map. This completes the proof of the case when $r \geq 1$.

If $r =0$, by Corollary \ref{coro:EulerFactor1}, we have 
\[\mathrm{mom}^{\underline{k}}_{0}\, \mathscr{Z}_{\infty} =  \prod_{\frakp \in \Upsilon_{F,p}} \left( 1 - q_{\frakp} \cdot p^{\sum_{\sigma \in \Sigma_{\frakp}}  k_{\sigma}}\, T_{\eta_{\frakp}}^{-1} \right) \cdot 
e^{\nord} \mathscr{Z}_0^{\underline{k}}.   \]
Using that the pairing is Hecke equivariant and Proposition \ref{prop:TetaUp}, which says that $T_{\eta}$ corresponds to $U_p^{^{\nord}}$, we need to calculate 
\[
\prod_{\frakp \in \Upsilon_{F,p}} \left( 1 - q_{\frakp} \cdot p^{\sum_{\sigma \in \Sigma_{\frakp}}  k_{\sigma}}\, {U_{\frakp}^{^{\nord}}}^{-1} \right) f_{E,\alpha}\otimes \theta^{-1} 
\]
which is
\[
\prod_{\frakp \in \Upsilon_{F,p}} \left( 1 - p^{\sum_{\sigma \in \Sigma_{\frakp}}  (1+k_{\sigma})}\,\eta_{E/F}(\frakp) {\alpha^{\circ}_{\frakp}}^{-2} \right) f_{E,\alpha}\otimes \theta^{-1},
\]
where we used $q_{\frakp}=p^{\sum_{\sigma \in \Sigma_{\frakp}}1}$, that the $U_{\frakp}^{^{\nord}}$-eigenvalue of $f_{E,\alpha}$ is the square of the $U_{\frakp}^{^{\nord}}$-eigenvalue of $f_{\alpha}$, and that the restriction of $\theta$ to $\A_F^\times$ is $\eta_{E/F}\chi_F$, with $\chi_F(\frakp)=1$.
\end{proof}

For simplicity, here we did not consider the intermediate cases where $\pmb{r}$ is zero on any possible proper subset of $\Upsilon_{F,p}$, however it can be done  by adding the appropriate Euler factor which appears in Remark \ref{Interestedreadercase}.

\subsection{The adjoint \texorpdfstring{$p$}{p}-adic \texorpdfstring{$L$}{L}-function}\label{sec:constructionpadicLF}

We let $\chi_F$ be a unitary Hecke character of $\A_F^\times$ which has trivial component at every place of $F$ above $p$. We let $\chi := \chi_F \circ {\rm N}_{E/F}$ and let $\theta$ be a unitary Hecke character of $\A_E^\times$ with conductor $\mathfrak{b}$ coprime with $p$ and such that its restriction to $\A_F^\times$ is $\chi_F \eta_{E/F}$. We also let $\mathfrak{b}_F$ denote $\mathfrak{b} \cap \mathcal{O}_F$. If $C_{\mathcal{E}}$ denotes the number defined in \eqref{RES1} for $\mathfrak{c}'=\mathfrak{c} \mathfrak{b}^2$ and $a_{\underline{k}}$ is the constant of Lemma \ref{rmk:uflag}, we define 
\begin{align*}
     C_{\underline{k} , \theta}^{E/F} := \frac{|D_F \Norm_{F/\Q}(D_{E/F})^{1/2}|^{2} |D_F|^{1/2} [K_0(\mathfrak{c}'):\mathcal{O}_E^{\ast, +}K_{11}(\mathfrak{c}')] G(\theta)}{a_{\underline{k}} \pi^d 2^{\sum_{\sigma \in \Sigma_F} (k_\sigma + 2) }C_{\mathcal{E}}}.
\end{align*}
\noindent In the following, suppose that $L_\mathfrak{m}$ contains the values of $\theta$. Moreover, if  $V_{\pmb{r}}= K_{11}(\mathfrak{c})  \cdot \prod_{\frakp \in \Upsilon_{F,p}} V_{\frakp,r_\frakp}$, $V_{\pmb{r}}'= K_{0}(\mathfrak{c})  \cdot \prod_{\frakp \in \Upsilon_{F,p}} V_{\frakp,r_\frakp}$, and $\tilde{{\rm pr}}: Y_G(V_{\pmb{r}}) \to Y_G(V_{\pmb{r}}')$ denotes the natural degeneracy map, we consider the pullback $\mathscr{Z}^\flat_{\infty,\theta}:= \tilde{{\rm pr}}^* \mathscr{Z}_{\infty,\theta} \in e^{\rm n.ord} H^{2d}_{\rm Iw}(Y_G(V_\infty), \mathcal{O}_{\mathfrak{m}})$ of the big cycle $\mathscr{Z}_{\infty,\theta}$ (twisted by $\theta$ as in Definition $\ref{def:twistedcycle2}$).

Let $\mathbf{f}$ be a primitive Hida family with tame character $\chi_F$, which we recall being an irreducible component of $ \mathcal{H}^{H/Z_H}_{11} \otimes_{\mathcal{O}_\mathfrak{m}} L_\mathfrak{m},$ and let $\mathbf{f}_E$ be its base change to $E$ as given in Definition \ref{def:BaseChangefamily}. We denote by  $\mathcal{K}_\mathbf{f}$ their field of coefficients. Note that $\mathbf{f}_E$ has tame character $\chi$. Fix a type $J$ for $E/F$, let $\varepsilon$ be the corresponding character on $\left\{ \matrix{\pm 1}{0}{0}{1} \right\}^{\Sigma_E}$, and denote by $\mathbf{1}^\varepsilon_{\mathbf{f}_{E}}$ the idempotent corresponding to $\mathbf{f}_E$ and $\varepsilon$. For a primitive Hida family $\mathbf{f}$ with image of the residual Galois representation not solvable, we can give the following :

\begin{definition}
 Let $G_{\mathbf{f}_{E}}$ be the generator of $\mathbf{1}^\varepsilon_{\mathbf{f}_{E}} H^{2d}_{\rm Iw}(Y_G(V_\infty), \mathcal{O}_{\mathfrak{m}}) \otimes_{\Lambda_{G/H}} \mathcal{K}_\mathbf{f}$ chosen after Lemma \ref{lemma:compareHeckeAlg}, so that when evaluating at a point, it satisfies \eqref{SettingaGen}.
Then there exists a unique element of $\mathcal{K}_\mathbf{f}$, that we denote by $L_p(\mathrm{Ad}(\mathbf{f})\otimes \eta_{E/F})$, so that we can write 
\[ \mathbf{1}^\varepsilon_{\mathbf{f}_{E}} \mathscr{Z}^\flat_{\infty,\theta}= L_p(\mathrm{Ad}(\mathbf{f})\otimes \eta_{E/F})G_{\mathbf{f}_{E}} .\]
\end{definition}
We call $L_p(\mathrm{Ad}(\mathbf{f})\otimes \eta_{E/F})$ the adjoint $p$-adic $L$-function of $\mathbf{f}$ at $s=1$ because of the following interpolation property.
\begin{theorem}\label{thmmain}  
Let $\mathbf{f}$ be a primitive Hida family with tame character $\chi_F$ with image of the residual Galois representation not solvable.
For a point $P$ of $\mathbf{f}_E$ corresponding to a cusp form $f_{E,\alpha}$ of weight $(\underline{k},0)$, minimal level $V_{\pmb{r}}$, and base change of $f_\alpha$, we let $g=f_{E,\alpha} \otimes \theta^{-1}$ and distinguish the following cases. 
\begin{itemize}
    \item If  $\pi(f_{E,\alpha})_{\mathfrak{P}}$ is a ramified principal series with adjoint character $\chi'_{\mathfrak{P}}$ of conductor $\mathfrak{P}^{r_\frakp}$ for $\mathfrak{P}|\frakp$, with $r_\frakp \geq 1$, at each  $\mathfrak{P}$, then $L_p(\mathrm{Ad}(\mathbf{f})\otimes \eta_{E/F}) [P]$ equals to \[\frac{G(\chi_{\mathfrak{P}}')\prod_{\frakp}{\rm N}_{F/\Q}(\frakp)^{2r_\frakp +2} {\rm Res}_{s=1}\zeta_F^{\mathfrak{d}_{E/F}p\mathfrak{c}_F\mathfrak{b}_F}(s) \prod_{\mathfrak{q} |\mathfrak{b}'} L_{\mathfrak{q}}({\rm As}(g),1)  C_{\underline{k} , \theta}^{E/F}}{{\alpha^{\circ}_p}^{2\pmb{r}} |(\mathcal{O}_F/p^{\pmb{r}+\pmb{1}})^\times|}\cdot \frac{L^{\ast, \mathfrak{d}_{E/F} \mathfrak{b}_F}(\mathrm{Ad}(f_\alpha) \otimes \eta_{E/F}, 1)}{\Omega^\varepsilon_p(f_{E,\alpha})} ,
\]
with $\mathfrak{b}' := \prod_{\substack{ \mathfrak{q} | \mathfrak{d}_{E/F} \\ \mathfrak{q} \nmid \mathfrak{b}_F}} \mathfrak{q}$. 
\item If  $\pi(f_{E,\alpha})_{\mathfrak{P}}$ is Steinberg at each  $\mathfrak{P}$ and $f_\alpha$ has $U_\frakp$-eigenvalue $\alpha_\frakp {\rm N}_{F/\Q}(\frakp)^{1/2}$ for each $\frakp$, then $L_p(\mathrm{Ad}(\mathbf{f})\otimes \eta_{E/F}) [P]$ equals to \[\frac{ \prod_{\frakp | p}  \left( 1 -\eta_{E/F}(\frakp) \alpha_\frakp^{-2} \right) p^{2d} {\rm Res}_{s=1}\zeta_F^{\mathfrak{d}_{E/F}p\mathfrak{c}_F\mathfrak{b}_F}(s) \prod_{\mathfrak{q} |\mathfrak{b}'} L_{\mathfrak{q}}({\rm As}(g),1)  C_{\underline{k} , \theta}^{E/F}}{ |(\mathcal{O}_F/p)^\times|}\cdot \frac{L^{\ast, \mathfrak{d}_{E/F} \mathfrak{b}_F}(\mathrm{Ad}(f_\alpha) \otimes \eta_{E/F}, 1)}{\Omega^\varepsilon_p(f_{E,\alpha})} .
\]
\item If  $\pi(f_{E,\alpha})_{\mathfrak{P}}$ is an unramified principal series at each  $\mathfrak{P}$, then $f_\alpha$ is the $p$-stabilisation of a nearly-ordinary form $f$ at each $\frakp$ with Satake parameters $\alpha_\frakp, \beta_\frakp$ and  $L_p(\mathrm{Ad}(\mathbf{f})\otimes \eta_{E/F})[P]$ is 
\begin{align*}
    \frac{\prod_{\frakp \mid p}  (1-\eta_{E/F}(\mathfrak{p})\alpha_\frakp^{-2}) (1-\eta_{E/F}(\mathfrak{p})\Norm_{F/\Q}(\mathfrak{p})^{-1}) 
(1-\eta_{E/F}(\mathfrak{p})\frac{\beta_{\mathfrak{p}}}{\alpha_{\mathfrak{p}}}\Norm_{F/\Q}(\mathfrak{p})^{-1}) }{ |(\mathcal{O}_F/p)^\times|} & p^{2d}\cdot \\
 {\rm Res}_{s=1}\zeta_F^{\mathfrak{d}_{E/F}p\mathfrak{c}_F\mathfrak{b}_F}(s) \prod_{\mathfrak{q} |\mathfrak{b}'} L_{\mathfrak{q}}({\rm As}(f_E),1)  C_{\underline{k} , \theta}^{E/F}&\cdot \frac{L^{\ast, \mathfrak{d}_{E/F} \mathfrak{b}_F}(\mathrm{Ad}(f) \otimes \eta_{E/F}, 1)}{\Omega^\varepsilon_p(f_{E,\alpha})}
\end{align*}
\end{itemize}
\end{theorem}

\begin{proof}
The proof revolves around calculating explicitly $\mathbf{1}^\varepsilon_{\mathbf{f}_{E}} \mathscr{Z}_{\infty,\theta}^\flat$ in terms of the fixed basis $G_{\mathbf{f}_E}$. To calculate the coefficient, we use the  Hecke-equivariant Poincar\'e pairing $[ \phantom{e}, \phantom{e}]$ defined in Section \ref{sec:Heckequivariantpairing}, then use the integral presentation of the adjoint $L$-value (see Proposition \ref{FromINTtoRES}) as well as the zeta integral computations in \S \ref{sec:RSintegrals}. Recall that both $\mathbf{1}^\varepsilon_{\mathbf{f}_{E}}$ and $[ \phantom{e}, \phantom{e}]$ are $U_p$-equivariant. 

Using that the evaluation of  $G_{\mathbf{f}_{E}}$ at a point $P$,  corresponding to a form $f_{E,\alpha}$ of weight $(\underline{k},0)$ and level $V_{\pmb{r}}$, satisfies \eqref{SettingaGen}, the $p$-adic $L$-function at $P$ is explicitly 
\[ L_p(\mathrm{Ad}(\mathbf{f})\otimes \eta_{E/F})[P] = \frac{[ \omega_J(W^*_{\mathfrak{c}p^{\pmb{r}+\pmb{1}}}f_{E,\alpha}^{-\iota}), {\rm mom}^{\underline{k}}_{\pmb{r}} \,\mathscr{Z}_{\infty,\theta}^\flat]}{\Omega^\varepsilon_p({f}_{E,\alpha})}. \]
Indeed, the evaluation at $P$ factors through the moment map of Definition \ref{def:mom} as the specialisation map to weight $(\underline{k},0)$ is given by taking the cup product with the highest weight vector: by \cite[Corollary 3.4 and Theorem 3.3]{HidaFamilies89} the map $j_\ast$ of \cite[(8.2b)]{HidaAnnals88}, which is the cup product with the highest weight vector, is the inverse (on the ordinary part) of the projection to the highest weight vector.
As $\tilde{{\rm pr}}^*$ and the correspondence $\mathcal{P}_{\mathfrak{b}}$ (introduced before Definition \ref{def:twistedcycle2}) commute with ${\rm mom}^{\underline{k}}_{\pmb{r}}$, we have (by \cite[Proposition 9.5]{GetzGore} - see also Proposition \ref{FromINTtoRES})
\begin{align}\label{Oneoftheintermediatestepsmainthm}
     L_p(\mathrm{Ad}(\mathbf{f})\otimes \eta_{E/F})[P] = [K_0(\mathfrak{c}'):\mathcal{O}_E^{\ast, +}K_{11}(\mathfrak{c}')] G(\theta) \frac{[ \omega_J(W^*_{\mathfrak{c}p^{\pmb{r}+\pmb{1}}}g^{-\iota}), {\rm mom}^{\underline{k}}_{\pmb{r}} \,\mathscr{Z}_{\infty}^{\chi^{-1}\theta^2}]}{\Omega^\varepsilon_p({f}_{E,\alpha})},
\end{align}
where, up to possibly enlarge $L$ so that $L_\mathfrak{m}$ contains the values of $\theta$, we have \[ \frac{ \omega_J(W^*_{\mathfrak{c}p^{\pmb{r}+\pmb{1}}}g^{-\iota})}{\Omega^\varepsilon_p({f}_{E,\alpha})} \in \mathbf{1}^\varepsilon_{\mathbf{f}_{E}} H^{2d}_c( Y_G(V_{\pmb{r}}),\mathscr{V}^{\underline{k},{0}}_{L_\mathfrak{m}} ).\]
We now evaluate the right hand side of \eqref{Oneoftheintermediatestepsmainthm} by distinguishing two cases. 
\begin{enumerate}
\item Suppose that $\pi(f_{E,\alpha})_{\mathfrak{P}}$ is a ramified principal series, with $f_{E,\alpha}$ base change of $f_\alpha$,  having adjoint character $\chi'_{\mathfrak{P}}$ of conductor $\mathfrak{P}^{r_\frakp}$ for $\mathfrak{P}|\frakp$, with $r_\frakp \geq 1$, at each  $\frakp \in \Upsilon_{F,p}$. Then, by Lemmas \ref{lemma:integraleval} and \ref{lemmaforactionofuonRSfinalversion}, we have
\begin{align*}
     \frac{[ \omega_J(W^*_{\mathfrak{c}p^{\pmb{r}+\pmb{1}}}g^{-\iota}), {\rm mom}^{{\underline{k}}}_{\pmb{r}} \,\mathscr{Z}_{\infty}^{\chi^{-1}\theta^2}]}{\Omega^\varepsilon_p({f}_{E,\alpha})} &=  \frac{1}{{\alpha^{\circ}_p}^{2\pmb{r}}}
\frac{\zeta^{p\mathfrak{c}'_F}_F(2)}{a_{\underline{k}} C_{\mathcal{E},\pmb{r}}} \frac{{\rm Res}_{s=0} I( \gamma_{\pmb{r}} \cdot f_{E,\alpha}, s)}{\Omega^\varepsilon_p({f}_{E,\alpha})}\\ 
     &= \frac{G(\chi_{\mathfrak{P}}')}{{\alpha^{\circ}_p}^{2\pmb{r}}}
\frac{\zeta^{p\mathfrak{c}'_F}_F(2)}{a_{\underline{k}} C_{\mathcal{E},\pmb{r}}} \frac{{\rm Res}_{s=0} I( f_{E,\alpha} ,s)}{\Omega^\varepsilon_p({f}_{E,\alpha})},
\end{align*}
where $a_{\underline{k}}$ is the constant of Lemma \ref{rmk:uflag},  ${\alpha^{\circ}_p}^{2\pmb{r}} = \prod_\frakp {\alpha^{\circ}_\frakp}^{2 r_\frakp}$, with $\alpha^{\circ}_\frakp$ the $U_{\frakp}^{^{\nord}}$-eigenvalue of $f_{\alpha}$, and $G(\chi_{\mathfrak{P}}')$ denotes the Gauss sum of the restriction to $\mathcal{O}_\frakp^\times$ of $\chi_{\mathfrak{P}}'$. We have also denoted by $ C_{\mathcal{E},\pmb{r}}$ the residue at $s=1$ of $\mathcal{E}_{\mathfrak{c}'p^{\pmb{r}+\pmb{1}}}^*(h,s)$ of level $\mathfrak{c}'p^{\pmb{r}+\pmb{1}}$, which relates to $C_\mathcal{E}$ by the formula
\[
C_{\mathcal{E},\pmb{r}} =\frac{|(\mathcal{O}_F/p^{\pmb{r}+\pmb{1}})^\times|}{\prod_{\frakp \in \Upsilon_{F,p}}{\rm N}_{F/\Q}(\frakp)^{2r_\frakp +2}} C_{\mathcal{E}}.
\]  
By Proposition \ref{AsaiIR}, we have \begin{align*}
    \zeta^{p\mathfrak{c}'_F}_F(2){\rm Res}_{s=0} I( f_{E,\alpha} ,s) &= C(0) {\rm Res}_{s=1} L({\rm As}(g),s)\\
    &=C(0)L^{\mathfrak{d}_{E/F} \mathfrak{b}_F}(\mathrm{Ad}(f_\alpha) \otimes \eta_{E/F}, 1) {\rm Res}_{s=1}\zeta_F^{\mathfrak{d}_{E/F}p\mathfrak{c}_F\mathfrak{b}_F}(s) \prod_{\mathfrak{q} |\mathfrak{b}'} L_{\mathfrak{q}}({\rm As}(g),1),
\end{align*} 
with $\mathfrak{b}' := \prod_{\substack{ \mathfrak{q} | \mathfrak{d}_{E/F} \\ \mathfrak{q} \nmid \mathfrak{b}_F}} \mathfrak{q}$. As \[ C(0) =  \frac{|D_F \Norm_{F/\Q}(D_{E/F})^{1/2}|^{2} |D_F|^{1/2}}{\pi^d 2^{\sum_{\sigma \in \Sigma_F} (k_\sigma + 2) }}\prod_{\sigma \in \Sigma_F}L_{\sigma}(\mathrm{Ad}(f_\alpha) \otimes \eta_{E/F}, 1), \]
the result follows.
 \item If $\pi(f_{E,\alpha})_{\mathfrak{P}}$ is either an unramified principal series or Steinberg at each $\mathfrak{P}$, then we have that $\pmb{r} = \pmb{0}$. The $p$-adic $L$-function evaluated at $P$ is then  
 \begin{align*}
     L_p(\mathrm{Ad}(\mathbf{f})\otimes \eta_{E/F})[P] = [K_0(\mathfrak{c}'):\mathcal{O}_E^{\ast, +}K_{11}(\mathfrak{c}')] G(\theta) \frac{[ \omega_J(W^*_{\mathfrak{c}p}g^{-\iota}), {\rm mom}^{\underline{k}}_{0} \,\mathscr{Z}_{\infty}^{\chi^{-1}\theta^2}]}{\Omega^\varepsilon_p({f}_{E,\alpha})}.
\end{align*}
We proceed as in the previous case, with the difference that by Lemmas \ref{lemma:integraleval} and \ref{lemmaforactionofuonRSfinalversion}, we get 
\begin{align*}
     \frac{[ \omega_J(W^*_{\mathfrak{c}p}g^{-\iota}), {\rm mom}^{\underline{k}}_{0} \,\mathscr{Z}_{\infty}^{\chi^{-1}\theta^2}]}{\Omega^\varepsilon_p({f}_{E,\alpha})} &= \frac{\mathcal{E}^{\pmb{0}}_{p}(f_\alpha)\zeta^{p\mathfrak{c}'_F}_F(2)}{a_{\underline{k}} C_{\mathcal{E},\pmb{0}}} \frac{{\rm Res}_{s=0} I( f_{E,\alpha} ,s)}{\Omega^\varepsilon_p(f_{E,\alpha})}, 
\end{align*}
where $\mathcal{E}^{\pmb{0}}_{p}(f_\alpha) =  \prod_{\frakp \in \Upsilon_{F,p}}  \left( 1 - p^{\sum_{\sigma \in \Sigma_{\frakp}}  (1+k_{\sigma})}\,\eta_{E/F}(\frakp) ({{\alpha^{\circ}_\frakp}})^{-2} \right)$, with $\alpha^{\circ}_\frakp$ the $U_{\frakp}^{^{\nord}}$-eigenvalue of $f_{\alpha}$. If $\pi(f_{E,\alpha})_{\mathfrak{P}}$ is Steinberg, we conclude as in (1). If  $\pi(f_{E,\alpha})_{\mathfrak{P}}$ is an unramified principal series, then $f_\alpha$ and $f_{E,\alpha}$ are $p$-stabilisations of nearly-ordinary cusp forms $f$ and $f_E$. If $\alpha_\frakp$, $\beta_\frakp$ are the Satake parameters of $\pi(f)_\frakp$, the value $\alpha^{\circ}_\frakp$ equals $p^{\sum_{\sigma \in \Sigma_{\frakp}}  k_{\sigma}/2} \alpha_\frakp {\rm N}_{F/\Q}(\frakp)^{1/2}$, hence   
\[\mathcal{E}^{\pmb{0}}_{p}(f_\alpha) = \prod_{\frakp \in \Upsilon_{F,p}}  \left( 1 -\eta_{E/F}(\frakp) \alpha_\frakp^{-2} \right).\]
Moreover, we apply Lemma \ref{RSintegralpstab} to have
\[{\rm Res}_{s=0} I(f_{E,\alpha},s) =  \left( \prod_{\frakp \mid p} \frac{(1-\eta_{E/F}(\mathfrak{p})\Norm_{F/\Q}(\mathfrak{p})^{-1}) 
(1-\eta_{E/F}(\mathfrak{p})\frac{\beta_{\mathfrak{p}}}{\alpha_{\mathfrak{p}}}\Norm_{F/\Q}(\mathfrak{p})^{-1})}{1 + \Norm_{F/\Q}(\mathfrak{p})^{-1}} \right) {\rm Res}_{s=0} I(f_E,s). \]
Write $\zeta^{p\mathfrak{c}'_F}_F(2) = \prod_\frakp(1 - {\rm N}_{F/\Q}(\frakp)^{-2}) \zeta^{\mathfrak{c}'_F}_F(2)$. Then on the product $\zeta^{\mathfrak{c}'_F}_F(2){\rm Res}_{s=0} I(f_E,s)$ we argue as in the previous cases and get 
\begin{align*}
    \zeta^{\mathfrak{c}'_F}_F(2){\rm Res}_{s=0} I(f_E,s) &= C(0) {\rm Res}_{s=1} L({\rm As}(f_E \otimes \theta^{-1}),s)\\
    &=C(0)L^{\mathfrak{d}_{E/F} \mathfrak{b}_F}(\mathrm{Ad}(f) \otimes \eta_{E/F}, 1) {\rm Res}_{s=1}\zeta_F^{\mathfrak{d}_{E/F}\mathfrak{c}_F\mathfrak{b}_F}(s) \prod_{\mathfrak{q} |\mathfrak{b}'} L_{\mathfrak{q}}({\rm As}(f_E \otimes \theta^{-1}),1).
\end{align*} 
We then conclude by noticing that \[ \prod_\frakp\frac{ 1 - {\rm N}_{F/\Q}(\frakp)^{-2}}{1 + \Norm_{F/\Q}(\mathfrak{p})^{-1}}{\rm Res}_{s=1}\zeta_F^{\mathfrak{d}_{E/F}\mathfrak{c}_F\mathfrak{b}_F}(s) = {\rm Res}_{s=1}\zeta_F^{\mathfrak{d}_{E/F}p\mathfrak{c}_F\mathfrak{b}_F}(s).\]
\end{enumerate}
\end{proof}

\remark{ (\texorpdfstring{$\Lambda$}{Lambda}-adic Hilbert modular forms with coefficients in cohomology})

\label{BeforeBigPairing}

Let $\mathcal{K}$ be a large enough extension of $\mathrm{Frac}(\Lambda_{G/H})$ which contains all the coefficient fields of the primitive  base-change families of Proposition \ref{prop:primitiveBaseChange}. Fix a type $J$ for $E/F$. For each base change family $\mathbf{f}_{E}$ contributing to $ \mathcal{H}^{G/H}_{11}$  as in Proposition \ref{prop:primitiveBaseChange}, choose a character $\theta_{\mathbf{f}_{E}}$ as in Theorem \ref{thmmain}.

\begin{definition}\label{DefofGGfamilyelement}
We define an element of $e^{\nord} H^{2d}_{\rm Iw}(Y_G(V_\infty), \mathcal{O}_{\mathfrak{m}}) \otimes_{\Lambda_{G/H}} \mathcal{K}$ :
\[
\sum_{ \mathbf{f}_{E}} \mathbf{1}^\varepsilon_{\mathbf{f}_{E}} \mathscr{Z}_{\infty,\theta_{\mathbf{f}_{E}}}^\flat,
\]
where the sum runs over all the base change families from $F$ of Proposition \ref{prop:primitiveBaseChange}.
\end{definition}
This can be seen as a family version of $\Phi_{Q([Z]),\chi_E}$ of \cite[Theorem 8.4]{GetzGore} (notation as in \S 8.6 of {\it loc. cit.}).
Similarly to Theorem 8.5 of {\it loc. cit.}, after identifying $e^{\nord} H^{2d}_{\rm Iw}(Y_G(V_\infty), \mathcal{O}_{\mathfrak{m}}) \otimes_{\Lambda_{G/H}} \mathcal{K}$ with a space of $\Lambda$-adic modular forms, we can calculate the ``Fourier coefficient'' of this form, applying Theorem \ref{thmmain}. It would be interesting to upgrade Definition \ref{DefofGGfamilyelement} to have a proper $\Lambda$-adic form, and not simply a sum of components in each base change family. Indeed, after applying the idempotents $\mathbf{1}^\varepsilon_{\mathbf{f}_{E}}$, one looses all the information on congruences (and intersection) between families. 
To obtain a $\Lambda$-adic form, one would need to construct a version of $\mathscr{Z}_{\infty}$ with coefficients in compactly supported cohomology (resp. intersection cohomology) and then pair it with our $\mathscr{Z}_{\infty}$ (resp. with itself). The main obstacle is the generalisation of Proposition \ref{Cartesian} to suitable compactifications.

\printindex

\bibliographystyle{alpha}
\bibliography{finalversion}
\end{document}